\documentclass[11pt]{amsart}
\usepackage{amscd,amsxtra,amssymb,mathrsfs, bbm}
\usepackage{stmaryrd}

\newtheorem{theorem}{Theorem}[section]
\newtheorem{corollary}[theorem]{Corollary}
\newtheorem{lemma}[theorem]{Lemma}
\newtheorem{proposition}[theorem]{Proposition}

\theoremstyle{definition}
\newtheorem{definition}[theorem]{Definition}
\newtheorem{remark}[theorem]{Remark}
\newtheorem{example}[theorem]{Example}

\DeclareMathAlphabet{\mathpzc}{OT1}{pzc}{m}{it}

\DeclareMathOperator{\SL}{\mathsf{SL}}

\DeclareMathOperator{\rk}{rk}

\DeclareMathOperator{\coker}{\mathsf{coker}}

\renewcommand{\ker}{\mathsf{ker}}
\newcommand{\im}{\mathsf{im}}

\renewcommand{\dim}{\mathsf{dim}}
\DeclareMathOperator{\tor}{\mathsf{tor}}
\DeclareMathOperator{\Coh}{\mathsf{Coh}}
\DeclareMathOperator{\SSt}{\mathsf{SS}}

\DeclareMathOperator{\VB}{\mathsf{VB}}

\DeclareMathOperator{\Pic}{\mathsf{Pic}}
\DeclareMathOperator{\Iso}{\mathsf{Isom}}

\DeclareMathOperator{\CM}{\mathsf{CM}}

\DeclareMathOperator{\Rep}{\mathsf{Rep}}

\DeclareMathOperator{\Ob}{\mathsf{Ob}}

\DeclareMathOperator{\add}{\mathsf{add}}

\DeclareMathOperator{\Hom}{\mathsf{Hom}}

\DeclareMathOperator{\Tor}{\mathsf{Tor}}

\DeclareMathOperator{\Ext}{\mathsf{Ext}}

\DeclareMathOperator{\Aut}{\mathsf{Aut}}

\DeclareMathOperator{\End}{\mathsf{End}}
\DeclareMathOperator{\Mat}{\mathsf{Mat}}

\input xy
\xyoption{all}
\xyoption{arc}

\newcommand{\kk}{k}

\newcommand{\FF}{\mathbb{F}}
\newcommand{\GG}{\mathbb{G}}
\renewcommand{\mod}{\mathsf{mod}}

\renewcommand{\AA}{\mathbb{A}}
\newcommand{\DD}{\mathbb{D}}
\renewcommand{\SS}{\mathbb{S}}
\newcommand{\TT}{\mathbb{T}}
\newcommand{\II}{\mathbb{I}}

\newcommand{\PP}{\mathbb{P}}
\newcommand{\XX}{\mathbb{X}}
\newcommand{\YY}{\mathbb{Y}}

\newcommand{\RR}{\mathbb{R}}
\newcommand{\LL}{\mathbb{L}}

\newcommand{\QQ}{\mathbb Q}
\newcommand{\CC}{\mathbb C}
\newcommand{\ZZ}{\mathbb Z}

\newcommand{\kA}{\mathcal{A}}

\newcommand{\kF}{\mathcal{F}}
\newcommand{\kG}{\mathcal{G}}
\newcommand{\kH}{\mathcal{H}}

\newcommand{\kO}{\mathcal{O}}
\newcommand{\kL}{\mathcal{L}}
\newcommand{\kP}{\mathcal{P}}

\newcommand{\kN}{\mathcal{N}}

\newcommand{\kT}{\mathcal{T}}
\newcommand{\kS}{\mathcal{S}}

\newcommand{\kX}{\mathcal{X}}

\newcommand{\lar}{\longrightarrow}
\newcommand{\overr}{\overrightarrow}
\newcommand{\overl}{\overleftarrow}

\newcommand{\cA}{\mathsf{A}}
\newcommand{\cB}{\mathsf{B}}
\newcommand{\cC}{\mathsf{C}}
\newcommand{\cD}{\mathsf{D}}
\newcommand{\cI}{\mathsf{J}}

\begin{document}

\title[The composition  algebra of a weighted projective line]{The composition Hall algebra of a weighted projective line}

\author{Igor Burban}
\address{
Mathematisches Institut,
Universit\"at Bonn,
Endenicher Allee 60,
D-53115 Bonn,
Germany
}
\email{burban@math.uni-bonn.de}

\author{Olivier Schiffmann}
\address{
D\'epartement de Math\'ematiques \\
Universit\'e  Paris Sud \\
B\^atiment 425 \\
91405 Orsay Cedex \\
 France
}
\email{Olivier.Schiffmann@math.u-psud.fr}


\begin{abstract}
In this article, we deal with   properties of the reduced Drinfeld double of the
composition subalgebra of the Hall algebra of the category of coherent sheaves
on a weighted projective line.
This  study is motivated by applications in the theory of
 quantized  enveloping algebras of some Lie algebras.
 We obtain
 a new realization of the quantized  enveloping algebras
of affine Lie algebras of simply-laced types as well as  some  new embeddings between them.
  Moreover, our approach allows to derive   new results  on the structure of  the quantized enveloping
algebras of the  toroidal algebras of types $D_4^{(1,1)}$, $E_6^{(1,1)}$,
$E_7^{(1,1)}$ and $E_8^{(1,1)}$. In particular,  our method leads  to a construction of
a modular action  and allows to define  a PBW-type basis for that classes of algebras.
\end{abstract}

\maketitle

\vspace{.1in}

\section{Introduction}
By the works of Ringel \cite{Ringel} and Green \cite{Green},
the study  of Hall algebras of the category of nilpotent representations of a
finite quiver plays an important role in the theory of quantized Kac-Moody algebras.
In this article,   we deal with  the Hall algebra of the category of coherent sheaves
on a weighted projective line $\XX$.

As it was established  by Kapranov  \cite{Ka} and elaborated by Baumann and Kassel
 \cite{BK}, the Hall algebra of the category of coherent sheaves on the classical
projective line $\PP^1$ is closely related
with Drinfeld's new realization of the quantized enveloping algebra
$U_q(\widehat{\mathfrak{sl}}_2)$.  In our previous joint work  \cite{DrinfeldBeck}
it was shown that the Drinfeld--Beck map  $U_q\bigl(\widehat{\mathfrak{sl}}_2\bigr)
\rightarrow U_q\bigl(\mathfrak{L sl}_2\bigr) $ (see \cite{Drinfeld, Beck})
comes from  the derived equivalence
$D^b(\Rep(\overr{\Delta})\bigr) \rightarrow D^b\bigl(\Coh(\PP^1)\bigr)$, where
$\overr{\Delta}$ is the Kronecker quiver.
An attempt to generalize this result on  other quantized enveloping algebras
leads us to the study of Hall algebras of
weighted projective lines.

Weighted projective lines and coherent sheaves on them were introduced by Geigle and
Lenzing in \cite{GL}. This notion turned out to be quite useful  from
various points of view. In particular, the so-called \emph{domestic} weighted projective lines
 provide a ``geometric realization'' of  the derived category of representation of
 affine Dynkin quivers. Weighted projective lines of \emph{tubular} type lead
 to a very interesting class of the  canonical tubular algebras \cite{RingelBook}.
 Geometrically, they correspond to the category of equivariant coherent sheaves on an elliptic curve
 with respect to the action of a finite automorphism group \cite{GL}.
 As it was shown in the  earlier work of the second-named author \cite{S1}, the Hall algebra of such
 weighted projective lines is closely related with the quantized enveloping
algebras of the  \emph{toroidal algebras} of types $\widehat{\widehat{D}}_4$, $\widehat{\widehat{E}}_6$,
$\widehat{\widehat{E}}_7$ and $\widehat{\widehat{E}}_8$.

The main results of this article are the following. Let $\XX$ be a weighted projective line
over a finite field $\FF_q$ and $H(\XX) = H\bigl(\Coh(\XX)\bigr)$ be its Hall algebra.
In the  work of the second-named author
\cite{S1},  the composition subalgebra $U(\XX) \subset H(\XX)$ was introduced.
We give an alternative definition of  $U(\XX)$  and show  that it  is a \emph{topological bialgebra},
i.e.~a subalgebra of $H(\XX)$
closed under Green's comultiplication $H(\XX) \to H(\XX) \widehat\otimes H(\XX)$. The core
of our proof is the technique of the so-called stability conditions on abelian categories.

Next, we  prove  that the subalgebra $\overline{U}(\XX)_{\mathrm{tor}}$ of  $U(\XX)$, which
generated by the classes
of the skyscraper sheaves, is isomorphic to $\mathcal{Z} \otimes
U_q^+\bigl(\widehat{\mathfrak{sl}}_{p_1}\bigr)
\otimes \dots \otimes U_q^+\bigl(\widehat{\mathfrak{sl}}_{p_n}\bigr)$, where
$\mathcal{Z} = \CC[z_1, z_2, \dots,z_l, \dots]$ is the Macdonald's ring of symmetric functions
and $\underline{p} = (p_1, p_2, \dots, p_n)$ is the \emph{weight type} of the weighted projective line  $\XX$.

We study in details
further properties of $U(\XX)$.  In particular, we
show that the functor  $\Coh(\PP^1) \rightarrow \Coh(\XX)$, which is right  adjoint to the functor
 of reduction of weights of Geigle and Lenzing \cite{GeiglLenzPerpCat}, induces an injective morphism
of the reduced Drinfeld doubles $U_q\bigl(\mathfrak{L} \mathfrak{sl}_{2}\bigr) \cong DU(\PP^1) \rightarrow DU(\XX)$. In particular, this implies
that the reduced Drinfeld double of the composition algebra of an arbitrary weighted projective line $\XX$
contains a subalgebra isomorphic to  $U_q(\widehat{\mathfrak{sl}}_2)$. More generally,
for any weighted projective line $\YY$ of type dominated by the type of $\XX$, we construct
an injective algebra homomorphism $DU(\YY) \to DU(\XX)$. These results are based on the
notion of perpendicular categories introduced by Geigle and Lenzing in \cite{GeiglLenzPerpCat}.
We also discuss some general properties of the reduced Drinfeld double of a perpendicular
subcategory of a hereditary abelian category, which can be applied in the context of quivers as well.
For example, this technique allows to construct a new  embedding of the quantized enveloping algebras
$U_q(\widehat{A}_3) \to U_q(\widehat{D}_4)$.

Using a recent result of Cramer \cite{Cramer},
 we show that for a domestic weighted projective line $\XX$ of affine Dynkin type
$\Delta$,  any derived equivalence $D^b\bigl(\Rep(\overr{\Delta})\bigr) \rightarrow
D^b\bigl(\Coh(\XX)\bigr)$ induces an isomorphism of the reduced Drinfeld doubles
of the composition algebras
$DC(\overr\Delta) \rightarrow DU(\XX)$, commuting with the Coxeter transformation.
 This differs from the isomorphism given by
Drinfeld \cite{Drinfeld} and Beck \cite{Beck} and might
 be interesting from the point of view of quantum groups.

Next, for a tubular weighted projective line $\XX$, we show that
the  group of the exact auto-equivalences of the derived category
$\Aut\bigl(D^b\bigl(\Coh(\mathbb{X})\bigr)\bigr)$ acts on the reduced Drinfeld double $DU(\XX)$ by
algebra automorphisms. In particular, it leads to a very interesting modular action
on $DU(\XX)$.  Using this  action, we  construct a monomial basis of $DU(\XX)$.

The Hall algebra
of a weighted projective line of was also considered
in the  recent paper \cite{PapierBarbare}, where applications to the
Drinfeld-Beck isomorphism were studied.

\vspace{1mm}
\noindent
\emph{Notation}. Throughout the paper, $\kk = \FF_q$ is a finite field with $q$ elements and
$\widetilde\QQ = \QQ[v, v^{-1}]/(v^{-2}- q) \cong \QQ[\sqrt{q}]$.
For a Kac-Moody Lie algebra $\mathfrak{g}$ we denote by $U_v(\mathfrak{g})$ the $\ZZ[v, v^{-1}]$
algebra, which is the
integral form of the corresponding quantized enveloping algebra, whereas
$U_q(\mathfrak{g}) := U_v(\mathfrak{g}) \otimes_{\ZZ[v, v^{-1}]} \widetilde{\QQ}$.

\vspace{2mm}
\noindent
\emph{Acknowledgement}.   The research of the first-named
author was supported by the DFG project Bu--1866/2--1. Parts of this work were done
during the authors stay at  the Mathematical Research Institute in Oberwolfach within
the ``Research in Pairs'' programme. We are also grateful to the anonymous referee for
his/her valuable comments and suggestions.

\section{Hereditary categories, their Hall algebras and Drinfeld doubles}

In this section we recall some basic facts on Hall algebras
of hereditary categories. Here we follow very closely the notations of our preceding
 article \cite{DrinfeldBeck}.
Let $\cA$ be an
essentially small hereditary abelian $\kk$--linear category such that for any pair  $M, N \in \mathrm{Ob}(\cA)$
the $\kk$--vector spaces $\Hom_\cA(M, N)$ and $\Ext^1_\cA(M, N)$ are finite dimensional.

\begin{itemize}
\item
Let $\cI = {\cI}_\cA := \bigl(\mathrm{Ob}(\cA)/\cong\bigr)$ be the set of  isomorphy classes of objects in $\cA$.
\item
For an object $X \in \mathrm{Ob}(\cA)$,  we denote by $[X]$ its image in $\cI$ and set  $a_X = \bigl|\Aut_\cA(X)\bigr|$.
\item For any triple of  objects $X, Y, Z \in \mathrm{Ob}(\cA)$ we denote
$$P^{Z}_{X, Y} = \left| \left\{(f, g) \in \Hom_{\cA}(Y, Z) \times \Hom_{\cA}(Z, X) \, \, \left| \, \,
0 \to Y \stackrel{f}\rightarrow  Z \stackrel{g}\rightarrow X \to 0 \right.
\mbox{\,  is exact} \right\}\right|
$$
and set   $F^{Z}_{X, Y} = \frac{\displaystyle P^{Z}_{X, Y}}{\displaystyle a_X \cdot a_Y}$. Note
that both numbers $P^{Z}_{X, Y}$ and $F^{Z}_{X, Y}$ are integers.
\item Let $K = K_0(\cA)$ be the K--group of $\cA$. For an object $X \in \mathrm{Ob}(\cA)$,  we denote by
$\overline{X}$ its image in $K$.
\item
Let $\widetilde\QQ[K]$ be the group algebra of $K$. For a class $\alpha \in K$ we denote by $K_\alpha$ the corresponding element in $\widetilde{\QQ}[K]$.
\item Let $\langle \,-\,, \,-\,\rangle: K \times K \rightarrow \ZZ$ be the
Euler form: for $X, Y \in \Ob(\cA)$ we have
$$
\bigl\langle \overline{X}, \overline{Y}\bigr\rangle := \dim_{\kk} \Hom_\cA(X, Y) - \dim_\kk \Ext^1_\cA(X, Y).
$$
Next,  let $(\,-\,, \,-\,): K \times K \rightarrow \ZZ$ be  the  symmetrization of
$\langle \,-\,, \,-\,\rangle$, i.e.~$
(\alpha, \beta) := \langle\alpha, \beta\rangle + \langle\beta, \alpha\rangle$ for all $\alpha, \beta
\in K.
$
\end{itemize}

\begin{definition}
Following  the  work of Ringel \cite{Ringel},   the  extended twisted \emph{Hall algebra} of the abelian
category
$\cA$ is
an associative algebra over $\widetilde{\QQ}$  defined
as follows.
\begin{itemize}
\item As a vector space over $\widetilde{\QQ}$, we have
$
\overline{H}(\cA) := \bigoplus\limits_{[Z] \in \cI} \widetilde{\QQ}[Z]$ and
$
H(\cA) :=  \overline{H}(\cA) \otimes_{\widetilde{\QQ}} \widetilde{\QQ}[K].
$
\item The multiplication  $\circ$ in $H(\cA)$ is defined by the following formulae:
\begin{itemize}
\item For any $\alpha, \beta \in K$ we have: $K_\alpha \circ K_\beta = K_{\alpha + \beta}$.
\item For any $\alpha \in K$ and $[X] \in \cI$ we have:
$
K_\alpha \circ [X] = v^{-(\alpha, \bar{X})} [X] \circ K_\alpha.
$
\item For   $[X], [Y] \in \cI$ the product $\circ$ is defined to be
$$
[X] \circ [Y] =  v^{-\langle \bar{X}, \bar{Y}\rangle} \sum\limits_{[Z] \in \cI} F^{Z}_{X, Y}\; [Z].
$$
\end{itemize}
\item
As it was shown in \cite{Ringel},  the multiplication  $\circ$ is associative and the element
$1 := [0] \otimes K_0$ is the unit element in $H(\cA)$.  In what follows, we
shall use the notation $[X]K_\alpha$ for the  element $[X] \otimes K_\alpha  \in
H(\cA)$.
\item Note that $H(\cA)$ is naturally a  $K$--graded algebra. For $\gamma \in K$ we shall
denote by $H(\cA)[\gamma]$ the corresponding graded peace. In these notations,
$\widetilde{\QQ}[K] = H(\cA)[0]$.
\end{itemize}
\end{definition}

\begin{definition}
According to  the  work of Green \cite{Green},
the Hall algebra $H(\cA)$ has a natural structure of a topological bialgebra endowed
with a bialgebra pairing.
\begin{itemize}
\item There exists a comultiplication
$H(\cA) \stackrel{\Delta}\longrightarrow H(\cA) \widehat\otimes H(\cA)$,
given by the formula:
$$
\Delta\bigl([Z] K_\alpha\bigr) = \sum\limits_{[X], \,  [Y] \in \cI}
v^{-\langle \bar{X}, \bar{Y}\rangle}
\frac{\displaystyle P^{Z}_{X, Y}}{a_Z} [X] K_{\bar{Y} + \alpha} \otimes [Y] K_\alpha.
$$
Here we refer to \cite[Section 2]{HallEll} for the definition of the completed tensor product
$H(\cA) \widehat\otimes H(\cA)$.
The map $\Delta$ is coassociative:  $(\Delta \otimes \mathbbm{1})\circ  \Delta = (\mathbbm{1} \otimes \Delta) \circ \Delta$. Moreover, $\Delta$ is  an algebra homomorphism.
\item There  exists a $\widetilde\QQ$--linear algebra homomorphism
  $H(\cA) \stackrel{\eta}\lar \widetilde\QQ$ given by the formula
$\eta\bigl([Z] K_\alpha\bigr) = \delta_{Z, \, 0}$. For any $a \in H(\cA)$ it satisfies
the equality $(\eta \otimes \mathbbm{1}) \circ \Delta (a) = (\mathbbm{1}\otimes \eta) \circ \Delta (a)
= a$.
\item
There exists  a  bilinear pairing $(\, -\,,\,-\,): H(\cA) \times H(\cA) \rightarrow
\widetilde{\QQ}$  given by the formula
$$
\Bigl([X] K_\alpha, [Y] K_\beta\Bigr) = v^{-(\alpha, \beta)}
\frac{\displaystyle \delta_{X, Y}}{\displaystyle a_X}.
$$
This pairing is non-degenerate on $\bar{H}(\cA)$ and symmetric.
Moreover,
for any elements $a, b, c \in H(\cA)$,
the expression $\bigl(a \otimes  b, \Delta(c)\bigr)$
 takes a  finite value  and the equalities
 $(a \circ b, \, c) = \bigl(a \otimes b, \Delta(c)\bigr)$ and $ (a, 1) = \eta(a)
$ are fulfilled.
\item If $\cA$ is moreover a finite length hereditary  category (for instance, the category of representations
of a finite quiver) then $H(\cA)$ is a true bialgebra over $\widetilde{\QQ}$ with the multiplication
$\circ$, unit $1$, comultiplication $\Delta$ and counit $\eta$.  Moreover,
by a work of Xiao \cite{Xiao}, the Hall algebra $H(\cA)$ has a natural Hopf algebra structure.
\end{itemize}
\end{definition}

\begin{remark}
The fact that the map $\Delta$ is a homomorphism of algebras, was proven by Green \cite{Green}
in the case when $\cA$ is the category of representations of a finite quiver. The
case of general hereditary abelian categories can be treated in a similar way,
see \cite{RingelGreen, OlivierNotes}.
\end{remark}

\noindent
Our next goal is to introduce the \emph{reduced} Drinfeld double of the topological bialgebra
$H(\cA)$.
 To define it,  consider the pair of   algebras $H^\pm(\cA)$, where  we use the notation
$$
H^+(\cA) = \bigoplus\limits_{Z \in \cI} \widetilde{\QQ}[Z]^+  \otimes_{\widetilde{\QQ}} \widetilde{\QQ}[K]
\,\, \mbox{and} \, \,
H^-(\cA) = \bigoplus\limits_{Z \in \cI} \widetilde{\QQ}[Z]^-  \otimes_{\widetilde{\QQ}} \widetilde{\QQ}[K].
$$
In these notations  $H^\pm(\cA) = H(\cA)$  viewed as $\widetilde{\QQ}$-algebras.
Let $a = [Z] K_\gamma$ and
$$
\Delta(a) = \sum\limits_{i} a_i^{(1)} \otimes a_i^{(2)} =
 \sum\limits_{[X], \, [Y] \in \cI} v^{-\langle \bar{X}, \bar{Y}\rangle}
\frac{\displaystyle P^{Z}_{X, Y}}{a_Z} [X] K_{\bar{Y} + \gamma} \otimes [Y] K_\gamma.
$$
Then we denote
$$
\Delta(a^\pm) = \sum\limits_{i} a_i^{(1)\pm} \otimes a_i^{(2)\pm} =
 \sum\limits_{[X], \, [Y] \in \cI} v^{-\langle \bar{X}, \bar{Y}\rangle}
\frac{\displaystyle P^{Z}_{X, Y}}{a_Z} [X]^\pm K_{\pm \bar{Y} + \gamma} \otimes [Y]^\pm  K_\gamma.
$$

\begin{definition}
The Drinfeld double of the topological bialgebra
$H(\cA)$ with respect to the Green's pairing $(\,-\,,\,-\,)$ is
the associative algebra $\widetilde{D}H(\cA)$, defined as  the free product of algebras $H^+(\cA)$ and
$H^-(\cA)$ subject to the following  relations $D(a, b)$ for all  $a, b \in H(\cA)$:
$$
\sum\limits_{i, j} a_{i}^{(1)-} b_{j}^{(2)+}  \bigl(a_{i}^{(2)}, b_{j}^{(1)}\bigr) =
\sum\limits_{i, j} b_{j}^{(1)+} a_{i}^{(2)-}  \bigl(a_{i}^{(1)}, b_{j}^{(2)}\bigr).
$$
The \emph{reduced} Drinfeld double
$DH(\cA)$ is the quotient of $\widetilde{D}H(\cA)$ by the two-sided ideal
$$
I = \left\langle K_{\alpha}^+ \otimes K_{-\alpha}^- - \mathbbm{1}^+ \otimes \mathbbm{1}^-
\, \, | \, \alpha \, \in K\right\rangle.$$
 Note that if $\cA$ is a finite length abelian category, then $I$ is a Hopf ideal and the reduced
 Drinfeld double
$DH(\cA)$ is again  a Hopf algebra.
\end{definition}

\begin{proposition}\label{P:DrinfDoubleStr}
We have an isomorphism of $\widetilde{\QQ}$--vector spaces
$$
 \overline{H}^{+}(\cA) \otimes_{\widetilde{\QQ}}  \widetilde{\QQ}[K]
 \otimes_{\widetilde{\QQ}} \overline{H}^{-}(\cA) \xrightarrow{\mathsf{mult}} DH(\cA)
$$
called the triangular decomposition of $DH(\cA)$.
\end{proposition}

\begin{remark}
The notion of the reduced Drinfeld double was introduced by Xiao in \cite{Xiao}.
In the case of Hopf algebras, a proof of Proposition \ref{P:DrinfDoubleStr} can be found in the  book
of Joseph \cite{Joseph}, whereas the case of topological bialgebras was treated in
 our previous joint paper \cite{HallEll}.
\end{remark}

\noindent
The following important  theorem was recently proven by Cramer \cite{Cramer}.

\begin{theorem}\label{T:Cramer}
Let $\cA$ and $\cB$ be two $\kk$-linear finitary hereditary categories such that there exists
 an equivalence
of triangulated categories $D^b(\cA) \stackrel{\FF}\longrightarrow D^b(\cB)$.
Then there is an algebra
isomorphism
$
DH(\cA) \stackrel{\FF}\lar DH(\cB)
$
uniquely determined  by the following properties.
For any $\alpha \in K$ we have:
$\FF(K_\alpha) = K_{\FF(\alpha)}$ and
for any object $X \in \mathrm{Ob}(\cA)$ such that
$\FF(X) \cong \widehat{X}[n]$ with $\widehat{X} \in \mathrm{Ob}(\cB)$ and
$n\in \ZZ$ we have 
\begin{equation}\label{E:action-on-Db}
\FF\bigl([X]^\pm\bigr) = v^{-n \bigl\langle \bar{X}, \bar{X}\bigr\rangle}
\bigl[\widehat{X}\bigr]^{\pm \varepsilon(n)}
\bigl(K_{\widehat{X}}^{\pm \varepsilon(n)}\bigr)^n,
\end{equation}
where $\varepsilon(n) = (-1)^n$. 
\end{theorem}

\section{Perpendicular categories and Hall algebras}

\begin{definition}
Let $\cA$ be a hereditary abelian category and $\cC$ be a  full subcategory. Following
Geigle and Lenzing \cite{GeiglLenzPerpCat}, we define the \emph{perpendicular category} $\cC^{\perp}$ as follows:
$$
\cC^\perp := \left\{
X \in \Ob(\cA)\, \bigl| \, \bigr. \Hom_\cA(C, X) = 0 = \Ext^1_\cA(C, X)
\quad \mbox{for all} \quad C \in \Ob(\cC)\right\}.
$$
\end{definition}

\begin{proposition}\label{P:perpcat-is-abelian}
In the above notations, the category $\cC^\perp$ is abelian, hereditary  and
extension-closed in $\cA$.
\end{proposition}

\begin{proof} Let $X, Y \in \Ob(\cC^\perp)$ be any pair of objects and $X \xrightarrow{f} Y$ be any
morphism. We first show that $Z = \mathsf{im}(f)$ belongs
to $\cC^\perp$.
Let $X \xrightarrow{p} Z$ and $Z \xrightarrow{\imath} Y$ be the canonical morphisms. Note that
$f = \imath p$ and $T := \ker(p) \cong \ker(f)$. For any $C \in \Ob(\cC)$ we have exact sequences
$$
0 \lar \Hom_\cA(C, Z) \stackrel{\imath_*}\lar \Hom_\cA(C, Y)
\quad \mbox{\rm and} \quad
\Ext^1_\cA(C, X) \stackrel{p_*}\lar \Ext^1_\cA(C, Z) \lar \Ext^2_\cA(C, T).
$$
Since $\cA$ is hereditary, we have: $\Hom_\cA(C, Z) = 0 = \Ext^1_\cA(C, Z)$, hence
$Z \in \Ob(\cC^\perp)$.
Next, let $L = \ker(f)$ and $N = \coker(f)$. Then we have the following short exact sequences
$$
0 \lar T \stackrel{\jmath}\lar  X \stackrel{p}\lar Z \lar 0
\quad
\mbox{\rm and}
\quad
0 \lar Z \stackrel{\imath}\lar  Y \stackrel{q}\lar N \lar 0.
$$
Taking into account  that $X, Y, Z \in \Ob(\cC^\perp)$  and the assumption $\cA$ to be hereditary, the induced
long exact sequences of $\Ext_\cA^\bullet$ imply  that $T, N$ belong to $ \cC^\perp$ as well. Hence,
the perpendicular category $\cC^\perp$ is an abelian
extension-closed subcategory of $\cA$.

Using Baer's description of the bifunctor  $\Ext^1_{\cC^\perp}(\,-\,,\,-\,)$,  it is easy to see
that for any $X, Y \in \Ob(\cC^\perp)$ the canonical morphism
$
\Ext^1_{\cC^\perp}(X, Y) \to \Ext^1_{\cA}(X, Y)
$
is an isomorphism. It remains to show that the category $\cC^\perp$ is hereditary.
Let $\omega \in \Ext^2_{\cC^\perp}(X, Y)$ be an extension class represented by an exact sequence
$$
0 \lar Y \stackrel{a}\lar E \stackrel{b}\lar F \stackrel{c}\lar X \lar 0,
$$
where $E, F \in \Ob(\cC^\perp)$. Let $I = \im(b) \cong \ker(c)$, then $I \in \Ob(\cC^\perp)$.
Since the extension
$\omega$ is zero, viewed as an element of $\Ext^2_{\cA}(X, Y)$,  there exists an object
$J$ of $\cA$ such that the following diagram is commutative
$$
\xymatrix
{
          &                 & 0 \ar[d]        & 0 \ar[d]        & \\
 0 \ar[r] & Y \ar[r]^a \ar[d]_= & E \ar[r]^e \ar[d] \ar@{.>}[rd]^{b} & I \ar[r] \ar[d]^f & 0 \\
 0 \ar[r] & Y \ar[r]  & J \ar[r] \ar[d] & F \ar[r] \ar[d]^c & 0 \\
          &                 & X \ar[d] \ar[r]^=        & X \ar[d]        & \\
          &                 & 0               & 0        &
}
$$
where  all rows and columns are exact sequences in $\cA$.
 Since the subcategory $\cC^\perp$ is extension-closed in $\cA$,
it follows that $J \in \Ob(\cC^\perp)$. Hence, $\Ext^2_{\cC^\perp}(X, Y) = 0$ and the
perpendicular category
$\cC^\perp$ is hereditary.
\end{proof}

\noindent
As a corollary,  we obtain the following interesting result.

\begin{theorem}\label{T:perpcatDrinfDouble}
Let $\cA$ be a hereditary abelian category over $\kk$, $\cC$ be some full subcategory of $\cA$ and $\cC^\perp$
be its  perpendicular category. Then the embedding  functor
$ \cC^\perp \stackrel{\II}\lar \cA$ induces an injective algebra  homomorphism of the reduced Drinfeld doubles
$DH(\cC^\perp) \to DH(\cA)$.
\end{theorem}

\begin{proof}
By Proposition \ref{P:perpcat-is-abelian} we know that the category $\cC^\perp$ is an extension-closed
hereditary full abelian subcategory of $\cA$. Moreover, for any $X, Y \in \Ob(\cC^\perp)$ and
$i = 0, 1$ the
canonical morphism $
\Ext^i_{\cC^\perp}(X, Y) \to \Ext^i_{\cA}(X, Y)
$
is an isomorphism. Hence, the functor $\II$ induces an isometry $K_0(\cC^\perp) \to
K_0(\cA)$ with respect to the Euler forms. Moreover, for any $X, Y, Z \in \Ob(\cC^\perp)$
the Hall constants $F^{Z}_{X, Y}$, $P^{Z}_{X, Y}$ and $a_X$ are preserved by $\II$.
Hence, we get an injective algebra homomorphism $ H(\cC^\perp) \stackrel{\II}\lar
H(\cA)$.

Next, let $X, Y$ be objects of $\cC^\perp$  and $Z$ be an object of $\cA$
such that there exist an epimorphism
$X \xrightarrow{p} Z$ and a monomorphism $Z \xrightarrow{\imath} Y$. As in the proof of
Proposition
\ref{P:perpcat-is-abelian} it can be shown   that all three objects
$Z$, $\ker(p)$ and $\coker(\imath)$ belong to
$\cC^\perp$. This shows that the relations $D(X^+, Y^-)$ of the Drinfeld double
$DH(\cC^\perp)$ are preserved by $\II$ and we have  a well-defined injective
algebra homomorphism $ DH(\cC^\perp) \stackrel{\II}\lar  DH(\cA)$. Note that $\II$ preserves
the triangular decomposition:
$$
\overline{H}^+(\cC^\perp) \otimes_{\widetilde{\QQ}} \widetilde{\QQ}\bigl[K_0(\cC^\perp)\bigr]
\otimes_{\widetilde{\QQ}} \overline{H}^-(\cC^\perp)
\stackrel{\II}\lar \overline{H}^+(\cA) \otimes_{\widetilde{\QQ}} \widetilde{\QQ}\bigl[K_0(\cA)\bigr]
\otimes_{\widetilde{\QQ}} \overline{H}^-(\cA).
$$
\end{proof}

\medskip
\noindent
As the first application of Theorem \ref{T:perpcatDrinfDouble}, consider the
following example.

\begin{example}\label{Ex:ApplPerpCat}
Let $\overr{\Delta}$ be a  $\widehat{D}_4$--quiver with the orientation
$$
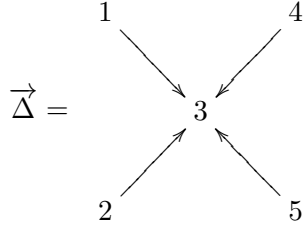
\overr{\Delta} =
\begin{array}{c}
\xymatrix{
1  \ar[rd] &  & 4  \ar[ld]\\
                & 3 &         \\
2 \ar[ru] &  & 5 \ar[lu]
}
\end{array}
$$
and $\cA = \Rep(\overr{\Delta})$ be its  category of representations.
Consider the indecomposable object $T$ with the dimension vector
 $\begin{smallmatrix} 0 &  & 0 \\  & 1 & \\ 1 & & 1
\end{smallmatrix}$ and the perpendicular category $\cB = T^\perp$.
Let $X_\alpha$, $X_\beta$, $X_\gamma$ and $X_\delta$ be the indecomposable representations
of $\overr{\Delta}$
with  the dimension vectors
$$
\alpha =
\begin{smallmatrix} 0 &  & 0 \\  & 1 & \\ 1 & & 0
\end{smallmatrix},
\quad
\beta =
\begin{smallmatrix} 0 &  & 0 \\  & 1 & \\ 0 & & 1
\end{smallmatrix},
\quad
\gamma =
\begin{smallmatrix} 1 &  & 0 \\  & 0 & \\ 0 & & 0
\end{smallmatrix}
\quad
\mbox{\rm and} \quad
\delta =
\begin{smallmatrix} 0 &  & 1 \\  & 0 & \\ 0 & & 0
\end{smallmatrix}.
$$
It is easy to see that $X_\alpha$, $X_\beta$, $X_\gamma$ and $X_\delta$ belong
to $T^\perp$. Moreover, for any object $X$ of the category $\cB$ with the dimension vector
$\begin{smallmatrix} m_1 &  & m_4 \\  & m_3 & \\ m_2 & & m_5
\end{smallmatrix}$
there exists a short exact sequence
$$
0 \lar Y \lar X \lar X_\gamma^{m_1} \oplus X_\delta^{m_4} \lar 0,
$$
where $Y$ has the dimension vector $\begin{smallmatrix} 0 &  & 0 \\  & m_3 & \\ m_2 & & m_5
\end{smallmatrix}$.  As in the proof of Proposition \ref{P:perpcat-is-abelian}
one shows that $Y$ belongs to $\cB$. Since $Y$ is supported on a subquiver of type $A_3$,
it is easy to deduce
that $Y$ splits into a direct sum of several copies of $X_\alpha$ and $X_\beta$. Hence, $\cB$
is a finite length hereditary abelian category with four simple objects
$X_\alpha$, $X_\beta$, $X_\gamma$ and $X_\delta$.

 Note that $\Ext_\cA^1(X_\gamma, X_\alpha) =  \Ext_\cA^1(X_\gamma, X_\beta) =
 \Ext_\cA^1(X_\delta, X_\alpha) =  \Ext_\cA^1(X_\delta, X_\beta) = \kk$, whereas   all the
 remaining $\Ext^1$--spaces between the simple objects of $\cB$
vanish. In the notations of Keller's work \cite{KellerAinfty} we have:
 $\cB = \mathsf{Filt}(X_\alpha, X_\beta, X_\gamma, X_\delta)$. As in \cite[Section 7]{KellerAinfty}
 one can show  that $\cB$ is equivalent to
$\Rep(\overr{\Delta}')$, where $\overr{\Delta}'$ is an $\widehat{A}_3$--quiver
with the orientation
$$
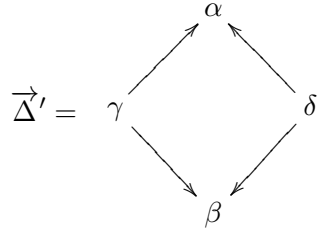
\overr{\Delta}' =
\begin{array}{c}
\xymatrix{
 & \alpha  & \\
\gamma \ar[ru] \ar[rd] &  &  \delta \ar[ld] \ar[lu] \\
 & \beta  &
}
\end{array}
$$
By Theorem  \ref{T:perpcatDrinfDouble},  we have an injective
morphism of the reduced Drinfeld doubles
\begin{equation}\label{E:funny-morph-quantalg}
DH\bigl(\Rep(\overr{\Delta}')\bigr) \stackrel{\II}\lar  DH\bigl(\Rep(\overr{\Delta})\bigr).
\end{equation}
Note that the following equalities are true in $H\bigl(\Rep(\overr{\Delta})\bigr)$:
$$
[X_\alpha] = v^{-1} [S_2] \circ [S_3] - [S_3] \circ [S_2]
\quad
\mbox{\rm and}
\quad
[X_\beta] = v^{-1} [S_5] \circ [S_3] - [S_3] \circ [S_5],
$$
whereas $[X_\gamma] = [S_1]$ and $[X_\delta] = [S_4]$.
Hence, the algebra homomorphism  $\II$  restricts to an injective homomorphism
$ DC(\overr{\Delta}') \stackrel{\II}\lar DC(\overr{\Delta})$ of the reduced Drinfeld doubles of the
composition subalgebras.
\end{example}

\begin{remark}
Passing to the generic composition algebras, the map  $\II$ constructed
in Example \ref{Ex:ApplPerpCat},  can be lifted to
an injective algebra homomorphism of quantum affine algebras $U_v(\widehat{A}_3)
\stackrel{\II}\lar
U_v(\widehat{D}_4).$
A similar game can be played
with the category of representations of other affine Dynkin quivers. However, we prefer to
postpone  its further discussion  to a future work.
\end{remark}

\section{Generalities on coherent sheaves on a weighted projective line}

In  this subsection, we recall some basic facts on the  category of coherent sheaves
on a weighted projective line, following a pioneering work of  Geigle and Lenzing \cite{GL}.
 For a set
  of positive integers $\underline{p}= \bigl\{p_1, p_2, \dots, p_n\bigr\}$
  and a set of pairwise distinct points
   $\underline{\lambda} = \{\lambda_1, \lambda_2, \dots, \lambda_n\}$
   of $\PP^1(\kk)$ normalized in such a way that $\lambda_1 = \infty = (1: 0)$,
   $\lambda_2 = 0 = (0:1)$, $\lambda_3 = 1 = (1: 1)$ and $\lambda_i = (1: \lambda_i)$ for $i \ge 3$,
 consider the ideal $I = I(\underline{p}, \underline{\lambda}) :=
 \bigl\langle x_2^{p_2} - \lambda_i x_1^{p_1}  -
  x_i^{p_i}\,| \, i \ge 3\bigr\rangle$
   and the  algebra
$
R = R(\underline{p}, \underline{\lambda}) =  \kk[x_1,x_2,\dots,x_n]/I.
$
In what follows we shall assume that
the cardinality of the base   field $\kk$  is bigger than the cardinality of $\underline{\lambda}$.

Let $\LL(\underline{p})$ be the abelian group generated by the elements
$\vec{x}_1, \vec{x}_2, \dots, \vec{x}_n$ subject to the relations
$p_1 \vec{x}_1 = p_2 \vec{x}_2 = \dots = p_n \vec{x}_n =: \vec{c}$. Then the polynomial
algebra $\kk[x_1, \dots, x_n]$ is $\LL(\underline{p})$--graded and
$I$ is an $\LL(\underline{p})$--homogeneous ideal. Hence, the algebra $R$ is $\LL(\underline{p})$--graded, too.
Note that $\LL(\underline{p}) \cong \ZZ \oplus \bigl(\bigoplus\limits_{i = 1}^n \ZZ/p_i \ZZ\bigr)$
and any  $\vec{x} \in \LL(\underline{p})$ can be uniquely written  as
\begin{equation}\label{E:can-form-inL(p)}
\vec{x} = l \vec{c} + \sum\limits_{i=1}^n
 a_i \vec{x}_i,
 \end{equation}
where $l \in \ZZ$ and $0 \le a_i < p_i$ for all $1 \le i \le n$. In what follows, the decomposition
(\ref{E:can-form-inL(p)}) will be called the canonical form of an element $\vec{x} \in \LL(\underline{p})$.
We say that $\vec{x} \in \LL(\underline{p})_+$ if in its canonical decomposition (\ref{E:can-form-inL(p)})
we have
 $l \ge 0$.

\begin{definition}
The category of coherent sheaves
$\Coh(\XX)$  on a weighted projective line
$\XX = \XX(\underline{p}, \underline\lambda)$  is  the
 Serre quotient $\mathsf{grmod}(R)/\mathsf{grmod}_0(R)$ of the category of graded Noetherian
 $R$--modules modulo the category of finite length  graded  $R$--modules. For a graded
 $R$--module $M$ we shall denote by $\widetilde{M}$ the corresponding object of $\Coh(\XX)$.
 \end{definition}

 \noindent
 For  an $\LL(\underline{p})$--homogeneous  prime  ideal $\mathfrak{p}$ consider the ring
 $$\kO_{\XX, \mathfrak{p}} := R_{\mathfrak{p}} =
 \left\{\frac{\displaystyle f}{\displaystyle g}\, \Bigl| \, \Bigr.
 g \in R  \, \;  \mbox{is} \, \; \LL(\underline{p})\mbox{--homogeneous},\,\, g \not\in
 \mathfrak{p}\right\}.$$
  Note that $\kO_{\XX, \mathfrak{p}}$ is again an  $\LL(\underline{p})$--graded
 discrete valuation ring and  we have an exact functor
 $\Coh(\XX) \to  \mathsf{grmod}(R_{\mathfrak{p}})$ mapping a coherent sheaf $ \kF$ to the module $ \kF_{\mathfrak{p}}.
 $

 \medskip
 \noindent
 The following observation is due to Geigle and Lenzing \cite[Section 1.3]{GL}.
\begin{lemma}\label{L:homog-prime-ideals}
Given two sets $\underline{p}$ and $ \underline{\lambda}$ as above,
 there are two types of homogeneous prime ideals of height one
in $R(\underline{p}, \underline{\lambda})$:
\begin{enumerate}
\item Ideals of the form $\bigl(f(x_1^{p_1}, x_2^{p_2})\bigr)$, where $f(y_1,y_2)\in \kk[y_1,y_2]$
is an irreducible   homogeneous polynomial in $y_1, y_2$, which is  different from $y_1$ and $y_2$;
\item  Exceptional prime ideals  $(x_1), (x_2), \dots, (x_n)$.
\end{enumerate}
\end{lemma}

 \begin{definition}
 A coherent sheaf $\kF$ on a weighted projective line $\XX$ is called
 \begin{itemize}
 \item locally free (or a vector bundle)  if $\kF_{\mathfrak{p}}$ is a projective object in
 $\mathsf{grmod}(R_{\mathfrak{p}})$ for all  $\LL(\underline{p})$--homogeneous  prime  ideals $\mathfrak{p}$;
 \item a skyscraper  sheaf (or a torsion sheaf) if $\kF_{\mathfrak{p}} = 0$ for all but finitely many
  $\LL(\underline{p})$--homogeneous  prime  ideals $\mathfrak{p}$.
 \end{itemize}
 \end{definition}

 \medskip
 It turns out that the category $\Coh(\XX)$ shares a lot of  common properties with
 the category of coherent sheaves on a commutative smooth projective curve.
 We list some  fundamental results  on this category,  which  are  due to Geigle and Lenzing
 \cite{GL,GeiglLenzPerpCat}, see also the  recent survey article \cite{ChenKrause}.

\medskip
\noindent
$1$.   The abelian category
$\Coh(\XX)$ is a hereditary noetherian category, which is $\Ext$--finite over the base field $\kk$.

\medskip
\noindent
$2$.
The canonical  functor $\mathsf{grmod}(R) \rightarrow \Coh(\XX)$ induces
an equivalence $\CM(R) \rightarrow \VB(\XX)$ between the category of $\LL(\underline{p})$--graded
Cohen--Macaulay $R$--modules and the category of vector bundles on $\XX$. For $\vec{x} \in \LL(\underline{p})$ we denote  $\kO(\vec{x}) = \kO_{\XX}(\vec{x}):= \widetilde{R(\vec{x})}$. Then the  map
 $$R_{\vec{y}- \vec{x}} \cong  \Hom_{\mathsf{grmod}(R)}\bigl(R(\vec{x}), R(\vec{y})\bigr)
 \xrightarrow{\mathsf{can}}  \Hom_{\XX}\bigl(\kO(\vec{x}),  \kO(\vec{y})\bigr)$$
  is an isomorphism of vector spaces for all $\vec{x}, \vec{y} \in \LL(\underline{p})$.

\medskip
\noindent
$3$.
 For any coherent sheaf $\kF$ on a weighted projective line $\XX$,
 the canonical exact sequence $$0 \lar \tor(\kF) \lar
\kF \lar \kF/\tor(\kF) \lar 0$$ splits, i.e.~any coherent sheaf on $\XX$  is a direct sum of a vector bundle and a torsion sheaf.

\medskip
\noindent
$4$. The category of torsion coherent sheaves $\Tor(\XX)$ splits into a direct sum of blocks:
$$
\Tor(\XX) = \bigoplus\limits_{\mathfrak{p} \in \mathsf{P}} \Tor_{\mathfrak{p}}(\XX),
$$
where $\mathsf{P}$ is the set of $\LL(\underline{p})$--homogeneous prime ideals
 in $R$ of height one and $\Tor_{\mathfrak{p}}(\XX)$ is the category of torsion coherent sheaves supported at the
prime ideal $\mathfrak{p}$. Note that each block  $\Tor_{\mathfrak{p}}(\XX)$ is equivalent to
the category of finite length  $\LL(\underline{p})$--graded  modules over $R_\mathfrak{p}$.

\medskip
\noindent
$5$. In the notations of Lemma \ref{L:homog-prime-ideals}, let
$\mathfrak{p} = \bigl(f(x_1^{p_1}, x_2^{p_2})\bigr)$ be  a homogeneous prime ideal of the first type.  Then
the category $\mathsf{grmod}_{\LL(\underline{p})}(R_\mathfrak{p})$ is equivalent to the category
of nilpotent representations over of the Jordan quiver over the field
$\kk(\mathfrak{p})  := \kk[y]/f(y, 1)$. Let $d$ be the degree of the homogeneous form
$f(y_1, y_2) \in \kk[y_1, y_2]$.
Then the unique simple object $\kS_{\mathfrak{p}}$ in $\Tor_{\mathfrak{p}}(\XX)$
has a locally free resolution
$$
0 \lar \kO(-d \vec{c}) \xrightarrow{f(x_1^{p_1}, \, x_2^{p_2})} \kO \lar
\kS_{\mathfrak{p}} \lar 0.
$$

\medskip
\noindent
$6$.
If  $\mathfrak{p} = (x_i)$ is  an
exceptional prime ideal for  $1 \le i \le n$,
then the category $\Tor_{\mathfrak{p}}(\XX)$ is equivalent to the category of
nilpotent finite-dimensional
representations of a cyclic quiver $\overr{C}_{p_i}$ with $p_i$ vertices over the base field $\kk$.
In particular, the category $\Tor_{\mathfrak{p}}(\XX)$ has $p_i$ simple objects
$\kS_i^{(j)}$, where $1 \le j \le p_i$. They have the following locally free resolutions:
$$
0 \lar \kO(-j \vec{x}_i) \stackrel{x_i}\lar \kO\bigl(-(j-1) \vec{x}_i\bigr) \lar
\kS_i^{(j)} \lar 0
$$
for $1 \le j \le p_i$. Note that $\Ext^1\bigl(\kS_i^{(j)}, \kS_i^{(j+1)}\bigr) = \kk$ for
all $1 \le j  \le p_i$, where $\kS_i^{(p_i+1)} := \kS_i^{(1)}$. All other $\Ext^1$--spaces between
the simple objects $\kS_i^{(j)}$ vanish.
The structure of the category of torsion coherent sheaves $\Tor(\XX)$ can be visualized by the following picture

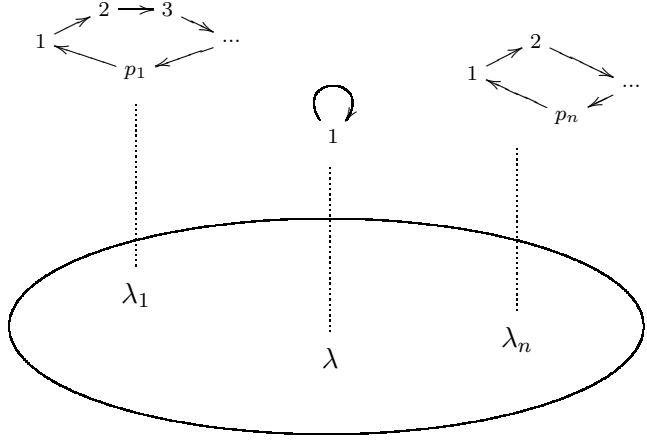
\vspace{3cm}
$$
{\xy /r0.20pc/:
{(25, 5)\ellipse(50,17){-}}
{\POS(-20,50)
@={
+(0,0)*+{\scriptstyle{1}},
+(10,5)*+{\scriptstyle{2}},
+(10,0)*+{\scriptstyle{3}},
+(10,-5)*+{\scriptstyle{\dots}},
+(-15,-5)*+{\scriptstyle{p_1}}
},
s0="x" @@{;"x"; **@{-} = "x"
\POS**{}?>*@{>}}
{\ar@{.} s0+(0,-5); s0+(0,-35)*++{\lambda_1}}
}
{\POS(50,45)
@={
+(-5,0)*+{\scriptstyle{1}},
+(10,+5)*+{\scriptstyle{2}},
+(15,-7)*+{\scriptstyle{\dots}},
+(-10,-5)*+{\scriptstyle{p_n}}
},
s0="x" @@{;"x";
**@{-}="x"
\POS**{}?>*@{>}
}
{\ar@{.} s0+(-8, -5);s0+(-8,-35)*++{\lambda_n}}
}
{\POS(20,35)*+{\scriptstyle{1}}="a" {\ar@(lu,ur)}
{\ar@{.} "a"+(-2,- 5); "a"+(-2,-35)*++{\lambda}}
}
\endxy}
$$

\medskip
\noindent
$7$.
 Consider the  element $\vec\omega := (n-2)\vec{c} - \sum\limits_{i=1}^n \vec{x}_i \in
\LL(\underline{p})$. Then the functor $\kF \mapsto \kF(\vec{\omega}) =: \tau(\kF)$ is the Auslander--Reiten translation
in the triangulated category $D^b\bigl(\Coh(\XX)\bigr)$. It means that  for any two coherent sheaves
$\kF$ and $\kG$ we have a bi-functorial isomorphism
$$
\Hom(\kF, \kG) \cong \DD\Ext^1\bigl(\kG, \kF(\vec{\omega})\bigr),
$$
where $\DD = \Hom_\kk(\,-\,, \kk)$ is the duality over the base field.
Note that for any $\LL(\underline{p})$--homogeneous prime ideal $\mathfrak{p}$ of the
first type we have $\tau(\kS_{\mathfrak{p}}) \cong \kS_{\mathfrak{p}}$, whereas for
the exceptional simple modules  we have:
$\tau\bigl(\kS_i^{(j)}\bigr) \cong  \kS_i^{(j+1)}$ for all $(i, j)$ such that
$1 \le i \le n$ and $1 \le j \le p_i$.

\medskip
\noindent
$8$.
The vector bundle $\kF := \kO \oplus
\bigoplus\limits_{i = 1}^n \bigl(\bigoplus\limits_{l = 1}^{p_i -1} \kO(l \vec{x}_i)\bigr) \oplus \kO(\vec{c})$
is a \emph{tilting object} in the derived category $D^b\bigl(\Coh(\XX)\bigr)$. In particular,
the derived category
 $D^b\bigl(\Coh(\XX)\bigr)$ is equivalent to
 $D^b\bigl(C(\underline{p}, \underline\lambda)-\mod\bigr)$, where
 $C(\underline{p}, \underline\lambda) = \End(\kF)^{\mathrm{op}}$
 is the so-called \emph{canonical algebra} of type $(\underline{p}, \underline\lambda)$ introduced
 and studied by Ringel \cite{RingelBook}. Moreover, if
  $\XX$ is a weighted projective line of domestic type $\Delta$ then $D^b\bigl(\Coh(\XX)\bigr)$
is equivalent to $D^b\bigl(\Rep(\overr{\Delta})\bigr)$, where $\Delta$ is the corresponding
affine Dynkin quiver.

\medskip
\noindent
$9$. The K--group $K_0(\XX)$ of the category $\Coh(\XX)$ is free of rank $\sum\limits_{i=1}^n (p_i -1) +2$ with a basis $$\Bigl\{\overline{\kO}, \overline{\kO(\vec{x}_1)}, \dots, \overline{\kO\bigl((p_1 -1)\vec{x}_1\bigr)},
\dots, \overline{\kO(\vec{x}_n)}, \dots, \overline{\kO\bigl((p_n -1)\vec{x}_n\bigr)},
\overline{\kO(\vec{c})}\Bigr\}.$$
Let  $x \in \PP^1(k)$ be a non-weighted point of degree one, $\kS$ a simple torsion sheaf supported at $x$
and $\delta = \overline{\kS}$ its class in the K--group. Then for any $1 \le i \le n$ we have the following
relations in $K_0(\XX)$:
$
\sum\limits_{j=1}^{p_i} \overline{\kS_i^{(j)}} = \overline{\kO} -
\overline{\kO(-\vec{c})} = \delta.
$
Moreover, the set  $$\Bigl\{\overline{\kS_1^{(1)}}, \dots, \overline{\kS_1^{(p_1-1)}},
\dots, \overline{\kS_n^{(1)}}, \dots, \overline{\kS_n^{(p_n-1)}}, \delta, \overline{\kO}\Bigr\}$$
 is also a
basis of $K_0(\XX)$. The rank function $K_0(\XX) \xrightarrow{\rk} \ZZ$ is defined by the rule
$\rk(\overline\kO) = 1$ and $\rk\bigl(\overline{\kS_i^{(j)}}\bigr) = 0$ for all $1 \le i \le n$ and $1 \le j \le p_i$. Moreover, for a coherent sheaf $\kF$ the  integer  $\rk\bigl(\overline\kF\bigr)$ coincides
with the geometrical rank of $\kF$ as defined in \cite[Section 1.6]{GL}.

\medskip
\noindent
$10$. Let $\cC = \add\bigl(\kS_i^{(1)}\bigr)$ be the Serre subcategory of $\Coh(\XX)$ generated
by an  exceptional simple object supported at the point $\lambda_i$ and
$\cC^\perp$ be the perpendicular subcategory. Let $\LL(\underline{p}')$
be the abelian group generated by the elements $\vec{y}_1, \dots, \vec{y}_n$ subject to the
relations
$
(p_1 -1) \vec{y}_1 = p_2 \vec{y}_2 = \dots = p_n \vec{y}_n
$
and $\YY = \XX(\underline{\lambda}, \underline{p}')$ be the corresponding weighted projective line.
Then we have:
\begin{enumerate}
\item The Serre quotient $\Coh(\XX)/\cC$ is equivalent to $\Coh(\YY)$. Moreover,
the canonical functor $\Coh(\XX) \stackrel{\PP}\lar \Coh(\XX)/\cC$ has an \emph{exact fully faithful}
right adjoint functor $\FF$.
\item The functor $\cC^\perp \to \Coh(\XX)/\cC$ given by the composition
of $\PP$ and the embedding $\cC^\perp \to \Coh(\XX)$, is an equivalence of categories.
Moreover, the essential image of the functor $\Coh(\XX)/\cC \stackrel{\FF}\lar \Coh(\XX)$ is the category $\cC^\perp$.
\item The functor $ \Coh(\YY) \stackrel{\simeq}\lar   \Coh(\XX)/\cC \stackrel{\FF}\lar \Coh(\XX)$ acts on objects as follows:
\begin{enumerate}
\item Let $\vec{y} = l\vec{c}' + \sum_{i=1}^n b_i \vec{y}_i \in \LL(\underline{p}')$ be
written in its canonical form (\ref{E:can-form-inL(p)}). Then
$\FF\bigl(\kO_{\YY}(\vec{y})\bigr) \cong  \kO_{\XX}\bigl(l \vec{c} + \sum_{i=1}^n b_i \vec{x}_i\bigr)$.
\item For any $2 \le i \le n$  and $1 \le j \le p_i$
we have: $\FF\bigl(\kS_i^{(j)}\bigr) \cong  \kS_i^{(j)}$.
\item Let $X$ be a object of the category $\Tor_{\lambda_1}(\YY)$. Assume it is given by a
representation of the cyclic quiver $\overr{C}_{p_1-1}$
$$
\xymatrix
{ V_1 \ar[r]^{A_1} &  V_2 \ar[r]^{A_2} &  \dots \ar[r] & V_{p_1-1} \ar@/^20pt/[lll]^{A_{p_1-1}}\\
}
$$
Then $\FF(X)$ belongs to the category $\Tor_{\lambda_1}(\XX)$ and is given by the representation
$$
\xymatrix
{
V_1 \ar[r]^{{I}} &
V_1 \ar[r]^{A_1} &  V_2 \ar[r]^{A_2} & \dots \ar[r]   & V_{p_1-1}
\ar@/^20pt/[llll]^{A_{p_1-1}}\\
}
$$
of the cyclic quiver $\overr{C}_{p_1}$,
where ${I}$ is the identity operator.
\end{enumerate}
\end{enumerate}

\medskip
\noindent
$11$. In a similar way, let  $\cC$ be the Serre subcategory of $\Coh(\XX)$ generated
by the simple objects $\kS_1^{(1)}, \dots, \kS_1^{(p_1-1)}, \dots,
\kS_n^{(1)}, \dots, \kS_n^{(p_n-1)}$. Then the perpendicular category $\cC^\perp$ is
equivalent to the Serre quotient $\Coh(\XX)/\cC$, which on its turn is equivalent to
the category $\Coh(\PP^1)$. The canonical functor $\Coh(\XX) \stackrel{\PP}\lar
\Coh(\XX)/\cC$ has an exact  fully faithful right adjoint functor $\FF$, whose
essential image is the perpendicular category $\cC^\perp$.  The functor
 $\Coh(\PP^1) \stackrel{\simeq}\lar \Coh(\XX)/\cC \stackrel{\FF}\lar
 \Coh(\XX)$ acts on objects as follows:
\begin{enumerate}
\item $\FF\bigl(\kO_{\PP^1}(l)\bigr) \cong \kO_{\XX}(l \vec{c})$ for all $l \in \ZZ$.
\item If $x \in \PP^1$ is a regular point then for any $\kS \in \Tor_x(\PP^1)$ we have:
$\FF(\kS) \cong  \kS$.
\item If $x = \lambda_i$ is a special point, $1 \le i \le n$ then
$\FF\bigl(\kO_{\PP^1}/\mathfrak{m}_x^l\bigr)$ belongs to $\Tor_{\lambda_i}(\XX)$ and
is given by the following representation of the cyclic quiver $\overr{C}_{p_i}$:
$$
\xymatrix
{ \kk^l \ar[r]^{I} &  \kk^l \ar[r]^{{I}} &  \dots \ar[r]^{I} & \kk^l \ar@/^15pt/[lll]^{J}\\
}
$$
where $I$ is the identity matrix and $J$ a nilpotent Jordan block of size $l \times l$.
\end{enumerate}

\begin{proposition}\label{P:keyprop-on-morphDD}
Let $\XX = \XX(\underline{p}, \underline{\lambda})$ be a weighted projective line,
$\cC = \add\bigl(\kS_1^{(1)}\bigr)$ be the Serre subcategory of $\Coh(\XX)$ generated by the  exceptional
simple object $\kS_1^{(1)}$ and $\YY$ be the weighted projective line such
that $\Coh(\YY)$ is equivalent to $\cC^\perp$. Then the functor
$\Coh(\YY) \stackrel{\FF}\lar  \Coh(\XX)$ induces an injective   homomorphism of the reduced Drinfeld doubles
$DH(\YY) \stackrel{\FF}\lar  DH(\XX).$
\end{proposition}

\begin{proof}
It is a  consequence  of Theorem \ref{T:perpcatDrinfDouble} and results
of Geigle and Lenzing quoted above.
\end{proof}

\section{Composition subalgebra of a weighted projective line: Part I}

\noindent
In this section we recall the definition of the composition subalgebra $U(\XX)$ of the
Hall algebra $H(\XX)$ of a weighted projective line $\XX$ introduced by second-named author
in \cite{S1}. The main
result of this part  is the fact that $U(\XX)$ is a \emph{topological bialgebra}.

\subsection{Hall algebra of a cyclic quiver}
In this subsection we collect some basic results on the Hall algebra of a cyclic quiver
$$\overr{C} = \overr{C}_p = \xymatrix
{ 1 \ar[r] &  2 \ar[r] &  \dots \ar[r] & p \ar@/^12pt/[lll]
}, \quad \mathrm{where} \quad p \ge 2.
$$
Let $\cA = \Rep(\overr{C})$
be the category of \emph{nilpotent} finite dimensional  representations of $\overr{C}$.
The following result is well-known, see for example
\cite{Drozd}.

\begin{theorem}\label{T:Drozd}
Let $\overr{L}$ be a quiver
of type $A_\infty$ with vertices labeled by the integers and arrows ordered linearly  and
 $\Rep(\overr{L})$ its  category of finite-dimensional
representations. Consider the  exact functor $
\Rep(\overr{L}) \xrightarrow{\PP} \Rep(\overr{C})
$ sending
a representation $V$ into the representation $\PP(V)$ such that $\PP(V)_i =
\bigoplus\limits_{j \,= \,  i \; \mod \; p} V_j$ for all $1 \le i \le p$. Then the indecomposable
objects of $\Rep(\overr{C})$  are precisely the images of the indecomposable objects of
$\Rep(\overr{L})$.
\end{theorem}

\noindent
Let $S_1, \dots, S_p$ be the simple objects of $\cA$. In what follows, we denote
$K_i := K_{\overline{S}_i} \in H(\cA)$ and   $\delta :=
\overline{S}_1 + \dots + \overline{S}_p \in K:=K_0(\cA)$.
The following result is due to Ringel \cite{RingelCyclic}.

\begin{theorem}\label{T:RingelonCyclicQuiver}
We have: $C(\overr{C}):= \bigl\langle [S_1], \dots, [S_n]; K_1^\pm, \dots, K_n^\pm\bigr\rangle \cong
U_q\bigl(\mathfrak{b}(\widehat{\mathfrak{sl}}_p)\bigr)$, where
$U_q\bigl(\mathfrak{b}(\widehat{\mathfrak{sl}}_p)\bigr) \cong
U_q^+\bigl(\widehat{\mathfrak{sl}}_p\bigr) \otimes_{\widetilde\QQ} \widetilde\QQ[K]$
is the Borel part of the quantized enveloping algebra of $\widehat{\mathfrak{sl}}_p$.
\end{theorem}

\noindent
The structure of the complete Hall algebra $H(\overr{C})$ has been  clarified by
the second-named author  \cite{SchiffmannCyclic} and Hubery \cite{Hubery}. Let $\overline{H}(\overr{C})$
be the non-extended Hall algebra of $\cA$.
Then for any $1 \le i \le p$ we have a linear operator
$e_i^* : \overline{H}(\overr{C}) \rightarrow \overline{H}(\overr{C})$ given by the rule
$
\bigl(e_i^*(x), y\bigr) = \bigl(x, [S_i] \otimes y\bigr)
$
for any $x, y \in \overline{H}(\overr{C})$. Here we use the fact that the Green's form
is non-degenerate on $\overline{H}(\overr{C})$. Next, we consider the following subspace
of $\overline{H}(\overr{C})$:
$$
Z(\overr{C}) = \bigcap\limits_{i=1}^p \ker(e_i^*) =
\Bigl\{x \in \overline{H}(\overr{C}) \, \bigl| \, \bigl(\Delta(x), [S_i] \otimes \,-\,\bigr) = 0
\; \mbox{for all} \;  1 \le i \le p \Bigr\}.
$$
The following results are  due to the second-named author   \cite{SchiffmannCyclic} and Hubery \cite{Hubery}.

\begin{theorem} In the above notations we have:
\begin{enumerate}
\item The vector space $Z(\overr{C})$ is the center of the Hall algebra  $H(\overr{C})$.
In particular, $Z(\overr{C})$ is a commutative algebra.
\item The canonical morphism $Z(\overr{C}) \otimes_{\widetilde\QQ} C(\overr{C})
\xrightarrow{\mathsf{mult}} H(\overr{C})$ is an isomorphism of vector spaces over
$\widetilde\QQ$.
\item The algebra $Z(\overr{C})$ is freely generated by the primitive elements
of the Hall algebra $H(\overr{C})$. This means that
$
Z(\overr{C}) = \widetilde{\QQ}[z_1, z_2, \dots, z_r, \dots]$, where $z_r \in H(\overr{C})[r\delta]$
are such that
$\Delta(z_r) = z_r \otimes \mathbbm{1} + K_{r \delta}  \otimes z_r
$
for any $r \in \ZZ_{>0}$.
In particular, $Z(\overr{C}) \otimes_{\widetilde\QQ} \widetilde\QQ[K_\delta^\pm]$ is
a commutative subbialgebra of the Hall algebra $H(\overr{C})$.
\end{enumerate}
\end{theorem}

\noindent
In the  work of Hubery \cite{Hubery} some  further properties of the center  $Z(\overr{C})$
were studied.

\begin{theorem}\label{T:Hubery}
For any $r \in \ZZ_{>0}$ consider the element
\begin{equation}
c_r := (-1)^r q^{-rp} \sum\limits_{\substack{M:\,  \underline{\dim}(M) = r\delta \\
\mathsf{top}(M) \; \mbox{\scriptsize{sq.~free}}}}
(-1)^{\dim_{\kk} \End_\cA(M)} \, |\Aut(M)| \, [M] \; \in \;  H(\overr{C})[r\delta].
\end{equation}
Then the following statements are true:
\begin{enumerate}
\item The elements $c_r$ are central and generate the algebra $Z(\overr{C})$.
\item For any $r \in \ZZ_{>0}$ we have:
$\Delta(c_r) = \sum\limits_{t=0}^r c_t K_{(r-t)\delta} \otimes c_{r-t}$.
\item For any $r, t \in \ZZ_{>0}$ we have: $(c_r, c_t) = \delta_{r, t} \, q^{-rp} \,  (1 - q^{-p})$.
\end{enumerate}
We set $z_1 = c_1$ and for any $r \ge 2$ define the element $z_r \in Z(\overr{C})[r \delta]
$ using the recursion
\begin{equation}
z_r = r c_r - \sum\limits_{l=1}^{r-1} z_l c_{r-l}.
\end{equation}
Then  the following formulae are true:
\begin{enumerate}
\item $\Delta(z_r) = z_r \otimes \mathbbm{1} + K_{r\delta} \otimes z_r$. In other words,
 $z_r$ is primitive for all
$r \in \ZZ_{>0}$.
\item $(z_r, c_r) = \frac{\displaystyle q^{-rp}}{\displaystyle 1 - q^{-rp}} =
\frac{\displaystyle 1}{\displaystyle q^{rp}-1}$ and  $(z_r, z_r) = r \frac{\displaystyle q^{-rp}}{\displaystyle 1 - q^{-rp}} =
\frac{\displaystyle r}{\displaystyle q^{rp}-1}$.
\item $(z_r, a b) = 0$ for any pair $a, b \in \overline{H}(\overr{C})$ such
that both $a$ and $b$ are non-scalar.
\end{enumerate}
\end{theorem}

\noindent
Let $\overr{J} = \overr{C}_1$ be the Jordan quiver. Then we have an exact functor
$\Rep(\overr{J}) \xrightarrow{\FF} \Rep(\overr{C})$ mapping  a representation
$(V, N)$ into the representation
$$
\FF(V, N) :=
\xymatrix
{ V \ar[r]^{I} &  V \ar[r]^{{I}} &  \dots \ar[r]^{I} & V \ar@/^15pt/[lll]^{N}\\
}
$$
where $I$ is the identity map. For any $r \in \ZZ_{>0}$ and any
partition $\lambda$ of $r$ let $I_\lambda$ be the corresponding
representation of $\overr{J}$. Consider the following element
of  $H(\overr{J})$:
\begin{equation}
p_r = \sum\limits_{\lambda | r} n_q\bigl(l(\lambda)\bigr) \bigl[I_\lambda\bigr],
\end{equation}
where $l(\lambda)$ is the length of $\lambda$ and
$n_q(i) = 1$ for $i = 1$ and $(1-q)  \dots  (1-q^{i-1})$ for $i \ge 2$.

\medskip
\noindent
Recall the following standard
facts on the classical Hall algebra $H(\overr{J})$, see \cite[I.5 Ex.~25 and III.4 (4.11)]{Macdonald}.

\begin{theorem}\label{T:classicalHallprimit}
For the elements $p_r \in H(\overr{J})$ the following properties are true:
\begin{enumerate}
\item
$\Delta(p_r) = p_r \otimes \mathbbm{1} + K_{r \delta} \otimes p_r$
for all $r \in \ZZ_{>0}$.
\item $(p_r, p_t) = \delta_{r, t} \frac{\displaystyle  r}{\displaystyle  q^r -1}$.
\end{enumerate}
\end{theorem}

\begin{proposition}
In the above notations,  set $t_r = \FF(p_r) \in H(\overr{C})$ for all
$r \in \ZZ_{>0}$. Then there exists an element $u_r \in C(\overr{C})$ such that
$
t_r = \frac{\displaystyle 1}{\displaystyle 1 - q^{-rp}} z_r + u_r.
$
In particular, the difference $\tau(t_r) - t_r$ belongs to the composition
subalgebra $C(\overr{C})$, where $\tau$ is the Auslander-Reiten translation
in $\Rep(\overr{C})$.
\end{proposition}

\begin{proof} The existence of a constant $\gamma \in \widetilde\QQ$ and an element
$u_r \in C(\overr{C})$ such that $t_r = \gamma z_r + u_r$ can be proven  along the same lines
as in \cite[Theorem 14]{Hubery}. In order to determine the value of  $\gamma$ note that
$(c_r, t_r) = q^{-rp}$, $(c_r, z_r) = q^{-rp}(1- q^{-rp})$ and
$(c_r, u_r) = 0$. Hence, $\gamma = \frac{\displaystyle 1}{\displaystyle 1 - q^{-rp}}$. It remains to observe
 that the translation
$\tau$ acts as identity on the center $Z(\overr{C})$ of the Hall algebra $H(\overr{C})$ and maps the composition
algebra $C(\overr{C})$ to itself.
\end{proof}

\begin{lemma}
In the above notations, let  $\alpha \in K$
be an element of the form $\overline{S}_1 + \dots +\overline{S}_t$ for some $1 \le t < p$.
For any $l \in \ZZ_{>0}$ let $T_{l \delta + \alpha}$ be  the unique indecomposable
object of $\cA$ of class $l\delta + \alpha$ with the simple top $S_1$. Then we have:
\begin{equation}\label{E:ident-in-cyclic-quiver}
\bigl[T_{l \delta + \alpha}\bigr] =
\left\{
\begin{array}{cc}
\bigl[T_{\alpha}\bigr] & \mbox{\rm if} \hspace{3mm} l = 0, \\
\bigl[T_{l\delta}\bigr] \circ \bigl[T_{\alpha}\bigr] -
v^2 \bigl[T_\alpha\bigr] \circ \bigl[T_{l\delta}\bigr] & \mbox{\rm if} \hspace{3mm} l > 0.
\end{array}
\right.
\end{equation}
\end{lemma}

\begin{proof}
This formula is trivial  for $l = 0$, so we consider the case $l \ge 1$.
Since $\Hom_\cA(S_i, S_i) = \kk = \Ext_\cA^1(S_i, S_{i+1})$ for all $1 \le i \le p$ (as usual,
we set $S_{p+1} := S_1$), whereas
all other $\Hom$ and $\Ext^1$ spaces  between the simple objects are zero,
we have: $\langle \delta, \beta\rangle = 0 = \langle \beta, \delta\rangle$
for all $\beta \in K$. It is easy to see that
$\Hom_\cA(T_{l\delta}, T_\alpha) = \kk$, hence $\Ext^1_\cA(T_{l\delta}, T_\alpha) = \kk$ as well.
Moreover, a generator of $\Ext^1_\cA(T_{l\delta}, T_\alpha)$ corresponds to a non-split
extension of the form
$
0 \to T_\alpha \to T_{l\delta + \alpha} \to T_{l\delta} \to 0.
$
Note that $\Hom_\cA(T_\alpha, T_{l\delta}) = 0$, hence
$\Ext^1_\cA(T_\alpha, T_{l\delta}) = 0$. Summing up, we have the following equalities
in the Hall algebra $H(\cA)$:
$$
\bigl[T_{l\delta}\bigr] \circ \bigl[T_\alpha\bigr] = \bigl[T_{l\delta + \alpha}\bigr] +
\bigl[T_{l\delta} \oplus T_{\alpha}\bigr] \quad \mbox{and} \quad
\bigl[T_\alpha\bigr] \circ \bigl[T_{l\delta}\bigr] = v^{-2} \bigl[T_{l\delta}
\oplus T_{\alpha}\bigr],
$$
which conclude the proof.
\end{proof}

\subsection{Composition subalgebra of $\PP^1$}\label{SS:CompAlgebraP1} The composition
subalgebra of the category of coherent sheaves on a projective line $\PP^1$ was introduced
by Kapranov \cite{Ka}. Later, it was studied in details by Baumann and Kassel \cite{BK}.
In this subsection,  we recall its definition and main properties.

\medskip
\noindent
$1$. The map $K_0(\PP^1) \xrightarrow{(\rk, \, \deg)} \ZZ^2$ is an isomorphism of abelian groups.
Let $\delta = (0, 1)$ be the class of the simple torsion sheaf supported at a $\kk$--point
of $\PP^1$. Then $\delta$ generates the radical of the Euler form
$\langle \,-\,,\,-\,\rangle$.

\medskip
\noindent
$2$. For any integer $r \ge 1$ consider the element
\begin{equation}\label{E:generInTor(P1)1}
\mathbbm{1}_{r\delta} := \sum\limits_{\kT \in \Tor(\PP^1): \, \overline{\kT} = (0, r)} [\kT] \in
 H(\PP^1).
\end{equation}
The elements $\bigl\{T_r\bigr\}_{r \ge 1}$   are determined
 by  $\bigl\{\mathbbm{1}_{r\delta}\bigr\}_{r \ge 1}$ via
 the generating series
\begin{equation}\label{E:generInTor(P1)2}
1 + \sum\limits_{r = 1}^\infty \mathbbm{1}_{r\delta} t^r = \exp\Bigl(\sum\limits_{r = 1}^\infty
\frac{T_r}{[r]_v} t^r\Bigr).
\end{equation}
Finally, the elements $\bigl\{\Theta_r\bigr\}_{r \ge 1}$ are defined by the generating series
\begin{equation}\label{E:generInTor(P1)3}
1 + \sum\limits_{r = 1}^\infty \Theta_r  t^r = \exp\bigl((v^{-1}-v)\sum\limits_{r = 1}^\infty T_r t^r\bigr).
\end{equation}
In what follows, we  set $\mathbbm{1}_{(0, 0)} = T_0 = \Theta_0 = [0] = \mathbbm{1}$.

\medskip
\noindent
$3$.
In the above notations  we have the following results, see for example \cite[Chapter 4]{OlivierNotes}.

\medskip
\noindent
(a)
These  three sets  $\bigl\{\mathbbm{1}_{(0, r)}\bigr\}_{r \ge 1}$,
$\bigl\{ T_r \bigr\}_{r \ge 1}$ and $\bigl\{ \Theta_r \bigr\}_{r \ge 1}$
 generate the
same subalgebra $U(\PP^1)_{\mathrm{tor}}$ of  the Hall algebra
$H(\PP^1)$;

\medskip
\noindent
(b)
For any $r,s  \ge 1$ we have the equalities:
\begin{equation}\label{E:coprodInP1}
\Delta(T_r) = T_r \otimes \mathbbm{1} + K_{(0,r)} \otimes T_r, \quad \bigl(\Theta_r, T_r\bigr) =
\frac{\displaystyle [2r]}{\displaystyle r} \quad
\mbox{and} \quad \bigl(T_r, T_s\bigr) = \delta_{r,s}
\frac{[2r]}{r(v^{-1}-v)}.
\end{equation}

\medskip
\noindent
$4$. For any $n \in \ZZ$ we have the following formula:
\begin{equation}\label{E:coprodInP1(2)}
\Delta\bigl(\bigl[\kO_{\PP^1}(n)\bigr]\bigr) =
\bigl[\kO_{\PP^1}(n)\bigr] \otimes \mathbbm{1} + \sum\limits_{r = 0}^\infty \Theta_r K_{(1, n-r)}
\otimes \bigl[\kO_{\PP^1}(n-r)\bigr],
\end{equation}
see    \cite[Theorem 3.3]{Ka} or
\cite[Section 12.2]{S2}.

\begin{definition}\label{D:compalgP1}
The composition algebra $U(\PP^1)$ is the subalgebra of the Hall algebra
$H(\PP^1)$ generated by the elements $L_l:= \bigl[\kO_{\PP^1}(l)\bigr]$, $T_r$ and $K_\alpha$,
where $l \in \ZZ$,  $r \in \ZZ_{>0}$ and $\alpha \in K_0\bigl(\Coh(\PP^1)\bigr) \cong  \ZZ^2$.
We also use the notations: $C = K_\delta$ and $O = K_{(1,0)}$. From the equalities  (\ref{E:coprodInP1}) and
(\ref{E:coprodInP1(2)}) it follows that  $U(\PP^1)$
is a topological subbialgebra of  $H(\PP^1)$.
\end{definition}

\vspace{2mm}

\noindent
A complete list of relations between the generators of
 the composition algebra $U(\PP^1)$ was obtained  by Kapranov
\cite{Ka}, Baumann and Kassel \cite{BK}, see also \cite[Section 4.3]{OlivierNotes}.

\begin{theorem}\label{T:relationsinP1}
The elements  $L_n, T_r$, $O$ and $C$
satisfy  the following relations:
\begin{enumerate}
\item $C$ is central;
\item $[O, T_n] = 0 = [T_n, T_m]$ for all $m, n  \in \ZZ_{> 0}$;
\item $O L_n = v^{-2} L_n O$ for all $n \in \ZZ$;
\item $\bigl[T_r, L_n\bigr] = \frac{\displaystyle [2r]}{\displaystyle r} L_{n+r}$ for all $n \in \ZZ$ and $r \in \ZZ_{> 0}$;
\item $L_{m} L_{n+1}  + L_{n}L_{m+1} =  v^2 \bigl(L_{n+1} L_{m} + L_{m+1} L_{n})$ for all $m, n \in
\ZZ$.
\end{enumerate}

\vspace{1mm}
\noindent
Let $U(\PP^1)_{\mathrm{vec}}$ be the subalgebra of $U(\PP^1)$ generated by the elements
$L_n$ $(n \in \mathbb{Z})$.
Then the map $U(\PP^1)_{\mathrm{vec}} \otimes_{\widetilde{\QQ}} U(\PP^1)_{\mathrm{tor}}
\otimes_{\widetilde{\QQ}} \widetilde{\QQ}[K]
\xrightarrow{\mathsf{mult}} U(\PP^1)$
is an isomorphism. Next,   the elements $$B_{\underline{m}, \,  \underline{l}, \, a, \,  b} =
\prod_{n \in \ZZ} L_n^{m_n} \circ \prod_{r \in \ZZ^+}  T_r^{l_r} \circ K^a C^b,$$
where $a, b \in \ZZ$,
$\underline{m} = (m_n)_{n \in \ZZ}$ and $\underline{l} = (l_r)_{r \in \ZZ_{>0}}$ are
sequences of non-negative integers such that all but finitely many entries are zero, form a basis
of $U(\PP^1)$ over the field $\widetilde\QQ$.
\end{theorem}

\subsection{First results  on  the composition subalgebra of a weighted projective line}\label{SS:FirstResultsonCompAlg}
Let $\XX$ be a weighted projective line of type $(\underline{p}, \underline{\lambda})$
and   $\cC$ be the Serre subcategory of $\Coh(\XX)$ generated
by the simple torsion sheaves  $\kS_1^{(1)}, \dots, \kS_1^{(p_1-1)}, \dots,
\kS_n^{(1)}, \dots, \kS_n^{(p_n-1)}$. Recall that  the Serre quotient $\Coh(\XX)/\cC$
is equivalent to  the category $\Coh(\PP^1)$ and the canonical functor
$\Coh(\XX) \stackrel{\PP}\lar \Coh(\XX)/\cC$ has an exact fully faithful
 right adjoint functor $\Coh(\PP^1) \stackrel{\FF}\lar \Coh(\XX).$ The second-named
 author suggested in \cite{S1} the following definition.

\medskip
 \begin{definition}\label{D:compalgXX}
 The composition algebra $U(\XX)$ is the subalgebra of the Hall algebra
 $H(\XX)$ generated by the elements $\bigl[\kS_i^{(j)}\bigr]$ for  $1 \le i \le n$ and
 $1 \le j \le p_i$,   $K_\alpha$ for $\alpha \in K_0(\XX)$ and the image
 of the composition algebra $U(\PP^1)$ under the homomorphism
 $H(\PP^1) \stackrel{\FF}\lar  H(\XX)$. For the sake of simplicity, we  use the same notation
  $T_r$ for a generator of $U(\PP^1)$ and its image in $H(\XX)$ under
   $\FF$. Recall that for any $l \in \ZZ$ we have an  isomorphism
  $\FF\bigl(\kO_{\PP^1}(l)\bigr) \cong \kO_{\XX}(l \vec{c})$. Summing up, we set
  $$
  U(\XX) := \Bigl\langle \bigl[\kO(l\vec{c})\bigr], \, T_r, \, \bigl[\kS_i^{(j)}\bigr], \,  K_\alpha \; \Bigl| \;
  l \in \ZZ, \, r \in \ZZ_{>0}, \, 1 \le i \le n, \, 1 \le j \le p_i, \,\alpha \in K_0(\XX) \Bigr.
  \Bigr\rangle.$$
 \end{definition}

\begin{remark}
At  first glance, the definition of  $U(\XX)$ depends
on the choice of a Serre subcategory $\cC$ and is not canonical. However, as we shall see later,
this is \emph{not} the case and $U(\XX)$ can be redefined in an \emph{intrinsic} way as an invariant
of the category $\Coh(\XX)$.
\end{remark}

\begin{lemma}\label{L:line-bundles-are-inU(X)}
Let $\kL \in \Pic(\XX)$ be a line  bundle. Then the element $[\kL]$ belongs to the composition
subalgebra $U(\XX)$.
\end{lemma}

\begin{proof}
Any line bundle $\kL$ on the  weighted projective line $\XX$ can be written as  $\kO(\vec{x})$, where
$\vec{x} = m \vec{c} + \sum_{l=1}^n a_l \vec{x}_l$ is an element of the group $\LL(\underline{p})$ written in
its canonical form (\ref{E:can-form-inL(p)}).  For any $0 \le i \le n$ consider the line bundle
$\kL_i = \kO\bigl(m\vec{c} + \sum_{l=1}^i a_l \vec{x}_l\bigr)$. Note that $\kL_0 = \kO(m\vec{c})$ and
$\kL_n = \kL$. Then for any $1 \le i \le n$ we have a short exact sequence
$$
0 \lar \kL_{i-1} \xrightarrow{x_i^{a_i}} \kL_i \lar \kT_i \lar 0,
$$
where $\kT_i$ is the torsion sheaf of length $a_i$ supported at the weighted point $\lambda_i$
and corresponding to the following representation of the cyclic quiver $\overr{C}_{p_i}$:
$$
\xymatrix
{ 0 \ar[r] & \dots \ar[r]  & 0 \ar[r] & \kk \ar[r]^{1} &  \dots \ar[r]^{1} & \kk \ar@/^20pt/[lllll]^{0}\\
}
$$
By a result of Ringel on the Hall algebra of a Dynkin quiver \cite{Ringel}, the element
$[\kT_i]$ belongs to the subalgebra generated by  $\bigl[\kS_i^{(l)}\bigr]$ for $1 \le l \le p_i$. Note that we have the isomorphisms:
$$
\Ext^1(\kT_i, \kL_{i-1}) = \kk \quad \mbox{\rm and} \quad
 \Ext^1(\kL_{i-1},\kT_i) =
\Hom(\kT_i, \kL_{i-1}) =  \Hom(\kL_{i-1},\kT_i) = 0.
$$
They  yield the following identities in the Hall algebra $H(\XX)$:
$$
\bigl[\kT_i\bigr] \circ \bigl[\kL_{i-1}\bigr] = v
\bigl(\bigl[\kL_i\bigr] + \bigl[\kL_{i-1} \oplus \kT_i\bigr]\bigr)
\quad
\mbox{\rm and}
\quad
 \bigl[\kL_{i-1}\bigr] \circ
\bigl[\kT_i\bigr]  =  \bigl[\kL_{i-1} \oplus \kT_i\bigr].
$$
Hence, for any $1 \le i \le n$ we obtain  the equality
$
\bigl[\kL_i\bigr] = v^{-1} \bigl[\kT_i\bigr] \circ \bigl[\kL_{i-1}\bigr] -  \bigl[\kL_{i-1}\bigr] \circ
\bigl[\kT_i\bigr]
$
proving the statement of the lemma.
\end{proof}

\subsection{Composition subalgebra $U(\XX)_{\mathsf{tor}}$}

\begin{definition} Let $\XX$ be a weighted projective line of type $(\underline{p}, \underline{\lambda})$.
Consider the subalgebra  $U(\XX)_{\mathsf{tor}}$ of  $U(\XX)$ defined as follows:
$$
  U(\XX)_{\mathsf{tor}} := \Bigl\langle T_r, \, \bigl[\kS_i^{(j)}\bigr], \,  K_\alpha \; \big| \;
  r \in \ZZ_{>0}, \, 1 \le i \le n, \, 1 \le j \le p_i, \,\alpha \in K_0(\XX) : \rk(\alpha) = 0 \Bigr.
  \Bigr\rangle.$$
  In what follows, we shall also use the notation  $K_i^{(j)} := K_{\overline{\kS_i^{(j)}}}$.
\end{definition}

\noindent
In this subsection,  we  give another characterization of the subalgebra
of $U(\XX)_{\mathsf{tor}}$   and  show it is a Hopf algebra. For this purpose it is convenient to
use the technique of stability conditions, see Section  \ref{S:StabilCond} for the basic definitions.

\begin{definition}\label{D:KeyStabilityFctn}
Let $\cA = \Tor(\XX)$ be the category of torsion coherent sheaves on $\XX$. For
the lattice
$
\bigl(K_0(\XX), \langle \,-\,,\,-\,\rangle\bigr)$ let  $K_0(\cA)
\stackrel{\mathsf{ch}}\lar K_0(\XX)$ be the canonical group homomorphism. We construct
the linear map $K_0(\XX) \stackrel{T}\lar \mathbb{R}_{\ge 0}$ as follows:
\begin{enumerate}
\item For any pair $(i, j)$, where $1 \le i \le n$ and $1 \le j \le p_i$ we attach a positive
weight $w_i^{(j)} \in \mathbb{R}_{>0}$.
\item Moreover, we assume that the following conditions are fulfilled:
\begin{enumerate}
\item For all $1 \le i \le n$ we have: $w_i^{(1)} > w_i^{(2)} > \dots > w_i^{(p_i)}$.
\item There exists $w  \in \mathbb{R}_{>0}$ such that for
      all $1 \le i \le n$ we have $\sum_{j=1}^{p_i} w_i^{(j)} = w$.
\item For any $1 \le i \le n$ and $1 \le a \le b \le p_i$ we set:
$$
\sigma^{(a, b)}_i := \frac{p_i}{b-a+1}\bigl(w_{i}^{(a)} + \dots + w_i^{(b)}\bigr).
$$
Then $\sigma^{(a, b)}_i = \sigma^{(c, d)}_j$ if and only if  $i = j$, $a = c$ and $b = d$.
Note that this
condition is automatically satisfied  when we take
arbitrary elements $w_1^{(1)}, \dots, w_1^{(p_1-1)}, \dots,$ $ w_n^{(1)}, \dots, w_n^{(p_n-1)}, w$
from $\mathbb{R}_{>0} \setminus \QQ_{>0}$,  which are linearly independent over $\QQ$ and such that
$w - \sum_{j=1}^{p_i-1} w_i^{(j)} > 0$ for all $1 \le i \le n$.
\end{enumerate}
\item We take the group  homomorphism
 $K_0(\XX) \xrightarrow{T \, = \,  (d, \, \delta)} \mathbb{R}^2$ given by the formulae:
 \begin{enumerate}
 \item $d\bigl(\overline{\kS_i^{(j)}}\bigr) = w_i^{(j)}$ for all $1 \le i \le n$ and
 $1 \le j \le p_i$ and $d(\overline{\kO}) = 1$.
 \item $\delta\bigl(\overline{\kS_i^{(j)}}\bigr) = \frac{\displaystyle 1}{\displaystyle p_i}$ for all $1 \le i \le n$ and
 $1 \le j \le p_i$ and $\delta(\overline{\kO}) = 0$.
 \end{enumerate}
\end{enumerate}
\end{definition}

\begin{lemma}\label{L:list-semist-objects}
In the above notations, the indecomposable semi-stable objects of the category
$\Tor(\XX)$ with respect to the stability function $Z := T \circ \mathsf{ch}: K_0\bigl(\Tor(\XX)\bigr)
\rightarrow \mathbb{R}^2$ are the following:
\begin{enumerate}
\item Any indecomposable torsion sheaf  supported at a regular point of $\XX$. The slope
of such a sheaf is $w$.
\item For a special point  $x = \lambda_i \, (1 \le i \le n)$ the indecomposable semi-stable sheaves
correspond to the  following representations of the cyclic quiver $\overr{C}_{p_i}$:
\begin{enumerate}
\item Representations $\kT_i(l)$, $l \in \ZZ_{>0}$  of slope $w$:
$$
\xymatrix
{ \kk^l \ar[r]^{I} &   \kk^l \ar[r]^{I} &  \dots \ar[r]^{I} & \kk^l \ar@/^15pt/[lll]^{J}\\
}
$$
where $I$ is the  identity matrix and $J $ is a  nilpotent Jordan block.
\item Representations $\kT_i^{(a, b)}$ of slope $\sigma_i^{(a, b)}$:
$$
\xymatrix
{ 0 \ar[r] & \dots \ar[r]  & 0 \ar[r] & \kk \ar[r]^{1} &  \dots \ar[r]^{1} & \kk  \ar[r] & 0 \ar[r] & \dots \ar[r] & 0
\ar@/^20pt/[llllllll]^{0}\\
}
$$
where $1 \le a \le b \le p_i$ and
    $b - a < p_i-1$. In these notations, the top of $\kT_i^{(a, b)}$    is $\kS_i^{(a)}$ and the
    socle of $\kT_i^{(a, b)}$  is $\kS_i^{(b)}$.
\end{enumerate}
\end{enumerate}
\end{lemma}

\begin{proof}
Let $x \in \PP^1$ be a non-special point. Then the unique simple sheaf $\kS_x$ supported at $x$
is stable with respect to any stability condition. Any torsion sheaf supported
at $x$ has a filtration,  whose factors are $\kS_x$. Since any extension of semi-stable sheaves
of a given slope  is again semi-stable with the same slope, any torsion sheaf supported
at $x$ is semi-stable.

Let $x = \lambda_i$ be a  special  point and $1 \le a \le b \le p_i$  are such that
$b-a < p_i - 1$.
It is clear that any  subrepresentation of  $\kT_i^{(a, b)}$ is isomorphic
to some $\kT_i^{(c, b)}$ with  $a \le c \le b$. However, from the assumption
$w_i^{(a)} > \dots > w_i^{(b)}$ it follows that the slope of $\kT_i^{(a, b)}$ is bigger
than the slope of $\kT_i^{(c, b)}$. Hence, $\kT_i^{(a, b)}$ is stable with respect to the
stability function $Z$. Since all representations
$\kT_i^{(a, b)}$ are rigid, the category of semi-stable objects of $\Tor(\XX)$ of slope
$\sigma_i^{(a, b)}$ is equivalent to the semi-simple abelian category $\add\bigl(\kT_i^{(a, b)}\bigr)$.

Next, consider the representation $\kT_i(1)$. Note that its  slope is $w$. Moreover, any proper subrepresentation of $T_i(1)$
is isomorphic to $\kT_i^{(a, p_i)}$ for some $a \ge 2$. Since $w > \sigma_i^{(a, p_i)}$ for any
$2 \le a \le p_i$, the representation $\kT_i(1)$ is stable. Moreover, for any $l \ge 1$ the
 representation $\kT_i(l)$ can be written as a successive extension of
 $\kT_i(1)$. Hence, $\kT_i(l)$ is semi-stable of slope $w$ for all $l \ge 2$.

 It remains to check that all the remaining indecomposable representations of the cyclic quiver
 $\overr{C}_{p_i}$
 are not semi-stable. Indeed, let $\kX$ be an indecomposable representation, which is neither isomorphic
 to $\kT_i(l)$ nor to $\kT_i^{(a, b)}$. Then $\kX$ either has a subrepresentation isomorphic
 to $\kT_i^{(1, b)}$ for some
  $1 \le b < p_i$ or a quotient isomorphic to $\kT_i^{(a, p_i)}$ for some $2 \le a  \le p_i$, see
  Theorem \ref{T:Drozd}.  In both cases,
  $\kX$ is not semi-stable, which completes the proof.
\end{proof}

\begin{proposition}\label{P:onU(X)tor}
In the notations of this subsection we have:
$$
U(\XX)_{\mathsf{tor}} =
\bigl\langle \mathbbm{1}_{\alpha}, K_{\alpha} \; \bigl| \; \alpha \in K_0(\XX): \; \rk(\alpha) = 0 \bigr.\bigr\rangle.
$$
In particular,  $U(\XX)_{\mathsf{tor}}$ is stable  under the action of the
Auslander-Reiten translation $\tau$.
\end{proposition}

\begin{proof}
Let $\Coh(\PP^1) \stackrel{\FF}\lar  \Coh(\XX)$ be the exact fully faithful functor defined at the beginning of   Subsection \ref{SS:FirstResultsonCompAlg}. For any $r \ge 1$ we set
$\widetilde{\mathbbm{1}}_{r\delta} := \FF(\mathbbm{1}_{r\delta}) \in H(\XX)$, where
$\mathbbm{1}_{r\delta} \in H(\PP^1)$ was defined by the equality (\ref{E:generInTor(P1)1}).  Let
$ K_0(\XX) \stackrel{T}\lar \mathbb{R}^2$ be the function introduced in Definition
\ref{D:KeyStabilityFctn}. Note that

\medskip
\noindent
$1$. For any $1 \le i \le n$ and $1 \le a \le b \le p_i$ such that
$b-a < p_i-1$ the  element $\bigl[\kT_i^{(a, b)}\bigr]$ of $H(\XX)$ belongs to the
subalgebra generated by the elements $\bigl[\kS_i^{(a)}\bigr], \dots, \bigl[\kS_i^{(b)}\bigr]$.

\medskip
\noindent
$2$. From the equality (\ref{E:generInTor(P1)2}) it follows that
 the subalgebras  $\bigl\langle \widetilde{\mathbbm{1}}_{r\delta} \; \bigl| \; r \ge 1\bigr\rangle$ and
$\bigl\langle T_r \; \bigl| \; r \ge 1\bigr\rangle$ of the Hall algebra $H(\XX)$ are equal.

\medskip
\noindent
$3$. By Corollary \ref{C:stabcondHallAlg}  we get the equality
$$\bigl\langle \mathbbm{1}_{\alpha} \; \bigl| \; \alpha \in K_0(\XX): \; \rk(\alpha) = 0 \bigr.\bigr\rangle = \bigl\langle \widetilde{\mathbbm{1}}_{r\delta},
    \bigl[\kS_i^{(j)}\bigr] \,  \bigl| \;
  r \in \ZZ_{>0}, \, 1 \le i \le n, \, 1 \le j \le p_i
     \bigr.\bigr\rangle.$$
This concludes the proof.
\end{proof}

\begin{corollary}\label{C:U(X)tor}
The subalgebra $U(\XX)_{\mathsf{tor}}$ of the Hall algebra $H(\XX)$ is a Hopf algebra.
\end{corollary}

\begin{proof} This result follows from Proposition \ref{P:onU(X)tor} and the formula
\begin{equation}\label{E:coprodfundament}
\Delta(\mathbbm{1}_{\gamma}) = \sum\limits_{\alpha + \beta = \gamma}
v^{\langle \alpha, \beta \rangle} \mathbbm{1}_{\alpha} K_{\beta} \otimes \mathbbm{1}_{\beta},
\end{equation}
valid for any $\gamma \in K_0(\XX)$ such that $\rk(\gamma) = 0$, see for example
\cite[Lemma 1.7]{OlivierNotes}.
\end{proof}

\noindent
We conclude this subsection by the following important corollary.

\begin{corollary}\label{C:new-def-ofU(X)}
The composition subalgebra $U(\XX)$ defined  in \cite{S1}
coincides with the following subalgebra $V(\XX)$
of the Hall algebra $H(\XX)$:
$$
V(\XX):=
\bigl\langle [\kO(\vec{x})], \mathbbm{1}_{\alpha}, K_\beta  \;
 \,  \bigl| \;
  \vec{x} \in \LL(\underline{p}); \alpha \in K_0(\XX): \rk(\alpha) = 0; \beta \in K_0(\XX)
     \bigr.\bigr\rangle.
$$
\end{corollary}

\begin{proof}
This result follows from Lemma \ref{L:line-bundles-are-inU(X)} and Proposition
\ref{P:onU(X)tor}.
\end{proof}

\begin{remark}
Corollary \ref{C:new-def-ofU(X)} gives an \emph{intrinsic}  description of the composition subalgebra $U(\XX)$. Moreover,
it shows that $U(\XX)$ is invariant under the natural action of the groups $\Pic(\XX) \cong \LL(\underline{p})$ and
$\Aut(\XX)$.
\end{remark}

\subsection{Composition algebra $U(\XX)$ is a topological bialgebra}
In this subsection,  we complete the proof of the fact that the composition subalgebra
$U(\XX)$ is a topological bialgebra.  Because of the equality
 $\Delta(K_\beta) = K_\beta \otimes K_\beta \bigl(\beta \in K_0(\XX)\bigr)$ and the formula (\ref{E:coprodfundament}) it is sufficient
to compute the comultiplication  of the line bundles on $\XX$. For any
$\vec{x} \in \LL(\underline{p})_+$ consider the element $\Theta_{\vec{x}} \in \overline{H}(\XX)_{\mathsf{tor}}$
 defined as follows:
\begin{equation}\label{E:Coprod(O)}
\Delta\bigl([\kO]\bigr) = [\kO] \otimes \mathbbm{1} +
\sum\limits_{\vec{x} \in \LL(\underline{p})_+} \Theta_{\vec{x}} K_{\overline{\kO(-\vec{x})}}
\otimes \bigl[\kO(-\vec{x})\bigr].
\end{equation}
 In this subsection we show that
 $\Theta_{\vec{x}} \in \overline{U}(\XX)_{\mathsf{tor}}$ for all  $\vec{x} \in \LL(\underline{p})_+$.
 By Corollary \ref{C:new-def-ofU(X)},  the algebra $U(\XX)$ is closed under
the action of the Picard group $\Pic(\XX)$. This will allows us
   to conclude
 that $U(\XX)$ is a topological bialgebra.

 We start with the case when $\XX = \PP^1$ is the usual non-weighted projective line.
 Recall the following formula for the elements $\Theta_r$ defined in
 (\ref{E:generInTor(P1)3}), see  \cite[Example 4.12]{OlivierNotes}:
\begin{equation}
\Theta_r =  v^{-r}\sum\limits_{m=1}^\infty\sum\limits_{\substack{x_1, \dots, x_m \in \PP^1; \,x_i \ne x_j \,1 \le i \ne j \le m
\\
t_1, \dots, t_m \,:\,
\sum_{i=1}^m t_i \deg(x_i) = r}}
\prod_{i = 1}^m \bigl(1 - v^{2\deg(x_i)}\bigr)\bigl[\kS_{t_i, x_i}\bigr],
\end{equation}
where $\kS_{t_i, x_i}$ is the unique indecomposable torsion sheaf of length $t_i$ supported at $x_i$.

\begin{lemma}
For any $r \in \ZZ_{>0}$ consider  the element  $\vec{x} = r \vec{c} \in \LL(\underline{p})_+$.
Then  we have:
$\Theta_{r \vec{c}} = \FF(\Theta_r)$, where $H(\PP^1) \stackrel{\FF}\lar  H(\XX)$ is the algebra
homomorphism  from Definition \ref{D:compalgXX}.
\end{lemma}

\begin{proof}
Let $0 \to \kO(-r \vec{c}) \to \kO \to \kT \to 0$ be a short exact sequence in $\Coh(\XX)$.
Since the essential image of the functor $\Coh(\PP^1) \stackrel{\FF}\lar \Coh(\XX)$
is closed under taking kernels and cokernels, the object $\kT$ belongs to
$\mathsf{Im}(\FF)$ as well. This means that all contributions to $\Theta_{r \vec{c}}$
come from $\PP^1$
and Proposition \ref{P:keyprop-on-morphDD} implies the claim.
\end{proof}

\begin{proposition} Let $\underline{\mu}$ be a non-empty subset of the set $\underline{\lambda}$ of exceptional points of our
weighted projective line $\XX$. For simplicity of notations,  we assume that
$\underline{\mu} = \bigl\{\lambda_1, \dots, \lambda_m\bigr\}$ for some
$1 \le m \le n$. For any $1 \le i \le m$  we fix  a  number $1 \le a_i < p_i$ and  for  $r \in
\ZZ_{>0}$   set
$
\vec{x} := r \vec{c} + \sum_{i= 1}^m a_i \vec{x}_i.
$
Then we have:
\begin{equation}\label{E:key-coprod}
\Theta_{\vec{x}} =
v^{-r-m} \sum\limits_{\underline{\mu} = \underline{\mu}' \coprod
\underline{\mu}''
}\bigl(
(-1)^{m''} T_{\underline{\mu}'} \circ \Theta_{r \vec{c}} \circ T_{\underline{\mu}''} -
v^{2 m''} (-1)^{m'} T_{\underline{\mu}''} \circ \Theta_{r \vec{c}} \circ  T_{\underline{\mu}'}
\bigr),
\end{equation}
where the sum is taken over all decompositions of $\underline{\mu}$ into a disjoint union
of two (possibly empty) subsets, $m' = |\underline{\mu}'|$, $m'' = |\underline{\mu}''|$,
$T_{\underline{\mu}'} := \prod_{i \in \underline{\mu}'} \bigl[\kT_i^{(1, a_i)}\bigr]$
and $T_{\underline{\mu}''} := \prod_{i \in \underline{\mu}''} \bigl[\kT_i^{(1, a_i)}\bigr]$. In particular,
$\Theta_{\vec{y}}$ is an element of $\overline{U}(\XX)_{\mathsf{tor}}$ for all $\vec{y} \in
\LL(\underline{p})_+$.
\end{proposition}

\begin{proof}
Consider  a short exact sequence in the category
$\Coh(\XX)$:
\begin{equation*}
0 \lar \kO(-\vec{x}) \stackrel{a}\lar  \kO \stackrel{b}\lar \kT \lar 0.
\end{equation*}
Note that we have the following equality  for the Euler form:
$\bigl\langle \overline{\kT}, \overline{\kO(-\vec{x})} \bigr\rangle = -
\bigl\langle
\overline{\kO}, \overline{\kT} \bigr\rangle = - r-m.
$
Our next goal is to compute the Hall constant $P^{\kO}_{\kT, \kO(-\vec{x})}$. To do  this,
we assume
\begin{equation}\label{E:isotype}
\kT \cong \kS_{l_1, x_1} \oplus \dots \oplus \kS_{l_p, x_p} \oplus \kT_{t_1 \delta + a_1}^{(1)}
\oplus \dots \oplus  \kT_{t_m \delta + a_m}^{(m)},
\end{equation}
where $x_1, \dots, x_p$ are distinct points of $\PP^1 \setminus \underline{\mu}$, whereas
$l_1, \dots, l_p \in
\ZZ_{>0}$ and $t_1, \dots, t_m \in \ZZ_{\ge 0}$ are such that
$\sum_{i=1}^p l_i \deg(x_i) + \sum_{i=1}^m t_i = r$. If $1 \le i \le p$ is such
that the point $x_i$ is non-special,  then $\kS_{l_i, x_i}$ is the unique torsion coherent sheaf
of length $l_i$ supported at $x_i$. If $x_i \in \underline{\lambda} \setminus
\underline{\mu}$ then the sheaf $\kS_{l_i, x_i}$ corresponds to the representation
$\kT_i(l_i)$ of the cyclic quiver $\overr{C}_{p_i}$ (we follow the notations of Lemma
\ref{L:list-semist-objects}).
For any $1 \le i \le m$  the coherent sheaf $\kT_{t_i \delta + a_i}^{(i)}$ is supported
at the point $x_i \in \underline{\mu}$ and corresponds to the unique indecomposable representation
of $\overr{C}_{p_i}$ of length $t_i p_i + a_i$ with the simple top $S_i^{(1)}$.  Note that
the decomposition type of the sheaf $\kT$ given by (\ref{E:isotype}) determines
the map $\kO(-\vec{x}) \stackrel{a}\lar \kO$ uniquely up to a non-zero constant. Hence, the
Hall number  $P^{\kO}_{\kT, \, \kO(-\vec{x})}$ is equal to the number of epimorphisms from
$\kO$ to $\kT$. Note that a morphism  $\kO \stackrel{a}\lar \kT$ is an epimorphism if and only if
all its components $\kO \to \kS_{l_i, x_i}$ and $\kO \to \kT^{(j)}_{t_j \delta + a_j}$ are epimorphisms
for all $1 \le i \le p$ and $ 1\le j \le m$.

The number of epimorphisms from $\kO$ to $\kS_{l, x}$ is equal to
$q^{\deg(x)l} - q^{\deg(x)(l-1)}$. Similarly, the number of epimorphisms
from $\kO$ to $\kT_{t \delta + a}^{(i)}$ is equal to $q^{t+1} - q^t$. Hence, we obtain:
$$
\Theta_{\vec{x}} =
v^{-r-m} \sum\limits_{\underline{x}, \; \underline{l}, \; \underline{t}}
c_{\underline{x}}
 [\kS_{l_1, x_1}] \circ \dots \circ [\kS_{l_1, x_1}] \circ
[\kT_{t_1 \delta + a_1}^{(1)}] \circ \dots \circ [\kT_{t_m \delta + a_m}^{(m)}],
$$
where
$c_{\underline{x}} = \bigl(1 - v^{2 \deg(x_1)}\bigr) \cdot \dots \cdot \bigl(1 - v^{2 \deg(x_p)} \bigr) (1- v^2)^m$ and
the sum is taken over all  decomposition types (\ref{E:isotype}).
Replacing each term $\bigl[\kT_{t_i \delta + a_i}^{(i)}\bigr]$ by the expression given by (\ref{E:ident-in-cyclic-quiver}), we end up with the formula (\ref{E:key-coprod}).
\end{proof}

\begin{theorem}\label{T:U(X)is-bialgebra} Let $\XX$ be a weighted projective line. Then
the composition algebra $U(\XX)$ is a topological subbialgebra of the Hall algebra
$H(\XX)$. In particular, the canonical multiplication map
$
\overline{U}^{+}(\XX) \otimes_{\widetilde{\QQ}}  \widetilde{\QQ}[K] \otimes_{\widetilde{\QQ}} \overline{U}^{-}(\XX) \xrightarrow{\mathsf{mult}} DU(\XX)
$
is an isomorphism of $\widetilde{\QQ}$--vector spaces, where  $\overline{U}^{\pm}(\XX)$ is
the subalgebra of the reduced Drinfeld double $DU(\XX)$
generated by the elements $\bigl[\kO(l\vec{c})\bigr]^\pm$ $(l \in \ZZ)$, $T_r^\pm$ $(r \in \ZZ_{>0})$
 and $\bigl[\kS_i^{(j)}\bigr]^\pm$ $(1 \le i \le n, 1 \le j \le p_i)$.
\end{theorem}

\begin{proof}
The composition algebra $U(\XX)$ is a topological bialgebra by
Corollary \ref{C:U(X)tor} and formula (\ref{E:key-coprod}).
The triangular decomposition  for $DU(\XX)$ is a general property
of the reduced Drinfeld doubles of (topological) bialgebras, see
 Proposition \ref{P:DrinfDoubleStr} and references therein.
\end{proof}

\subsection{Further structure properties of the reduced Drinfeld double $DU(\XX)$}

\noindent First note that
Theorem \ref{T:U(X)is-bialgebra} implies the following interesting result.

\begin{corollary}\label{C:structuteDU(X)}
The algebra  $DU(\XX)$ is generated by  one of the following sets:
\begin{enumerate}
\item the elements $\bigl[\kO(l\vec{c})\bigr]^\pm$ $ (l \in \ZZ)$, $\bigl[\kS_i^{(j)}\bigr]^\pm \; (1 \le i \le n, 1 \le j \le p_i)$,
    \item the classes of the line bundles $\bigl[\kO(\vec{x})\bigr]^\pm$ $\bigl(\vec{x} \in
    \LL(\underline{p})\bigr)$,
\end{enumerate}
together with the generators
$K_\alpha \; (\alpha \in K)$ of the Cartan part $\widetilde\QQ[K]$.
\end{corollary}

\begin{proof}
Let $\Coh(\PP^1) \stackrel{\FF}\longrightarrow \Coh(\XX)$ be the functor from Definition
\ref{D:compalgXX}.  By Proposition \ref{P:keyprop-on-morphDD},  we get the induced
morphism of the reduced Drinfeld doubles
$DH(\PP^1) \stackrel{\FF}\longrightarrow DH(\XX)$, which restricts to  the
algebra homomorphism  $DU(\PP^1) \stackrel{\FF}\longrightarrow DU(\XX)$. In particular,
for any $l \in \ZZ_{>0}$
we get the following relation in $DU(\XX)$:
$
\bigl[[\kO]^+, [\kO(-l \vec{c})]^-\bigr] = \frac{v}{v^{-1} - v} \, \Theta_{l\vec{c}}^+ \, O C^{-l},
$
where $O = K_{\overline\kO}$ and $C = K_\delta$. Since the elements
$\bigl\{ \Theta_{l\vec{c}} \bigr\}_{l \ge 0}$ and
 $\bigl\{ T_l \bigr\}_{l \ge 0 }$ generate the same subalgebra of the composition algebra
 $U(\XX)$   (see Subsection \ref{SS:CompAlgebraP1}), this shows that for any $r \in \ZZ_{>0}$ the
 generator   $T_r^+$ can be expressed through the elements
$\bigl[\kO(l\vec{c})\bigr]^\pm \, (l \in \ZZ)$. The case of the elements
$T_r^-$ can be treated in the same way. Similarly, for any $1 \le i \le n$ and
$1 \le j \le p_i$ we have the following equality in
$DU(\XX)$:
$
\bigl[[\kO((1-j)\vec{x}_i)]^+, [\kO(-j\vec{x}_i)]^-\bigr] = \frac{v}{v^{-1} - v}
K_{\overline{\kO(-j\vec{x}_i)}} [\kS_i^{(j)}].
$
By Lemma \ref{L:line-bundles-are-inU(X)},
$[\kO(\vec{x})]$ belongs to $U(\XX)$ for all $\vec{x} \in \LL(\underline{p})$.
\end{proof}

\begin{remark}
Notice  that the line bundles and the exceptional simple sheaves are rigid objects in the category
$\Coh(\XX)$. Hence,  Corollary \ref{C:structuteDU(X)} says that the reduced Drinfeld double $DU(\XX)$ is
generated by the classes of (certain) rigid objects. This observation will
be elaborated one step further for  weighted projective lines of domestic and tubular
types.
\end{remark}

\begin{theorem}\label{T:morphDDinducedbyIndweights}
Let $\XX = \XX(\underline{p}, \underline{\lambda})$ be a weighted projective line,
$\cC = \add\bigl(\kS_1^{(1)}\bigr)$  and $\Coh(\YY) = \cC^\perp$
be as in Proposition \ref{P:keyprop-on-morphDD}.
 Then the (injective) algebra  homomorphism of the reduced Drinfeld doubles
$DH(\YY) \stackrel{\FF}\lar  DH(\XX)$ induced by the functor
$\Coh(\YY) \stackrel{\FF}\lar \Coh(\XX)$ restricts to an injective algebra homomorphism
$DU(\YY) \stackrel{\FF}\lar  DU(\XX)$.
\end{theorem}

\begin{proof}
The statement of the theorem is a consequence of  the following observations. Let
$\cD$ be the Serre subcategory of $\Coh(\XX)$ generated
by the  torsion sheaves  $\kS_1^{(1)}, \dots, \kS_1^{(p_1-1)}, \dots, $ $
 \kS_n^{(1)}, \dots, \kS_n^{(p_n-1)}$. Then the perpendicular category $\cD^\perp$
is a full subcategory of $\cC^\perp$ equivalent to $\Coh(\PP^1)$.  Consider
the canonical inclusion functors $\cD^\perp \stackrel{\GG}\lar \cC^\perp
\stackrel{\FF}\lar \Coh(\XX)$.  By Proposition \ref{P:keyprop-on-morphDD},  we get an
injective algebra homomorphisms of the reduced Drinfeld doubles
$DH(\PP^1) \stackrel{\GG}\lar DH(\YY) \stackrel{\FF}\lar DH(\XX)$.
Note that
\begin{itemize}
\item  For any $r \in \ZZ_{>0}$ we have:
$\FF\GG(T_r) \cong  T_r$.
\item For any $(i, j)$ such that $2 \le i \le n$ and  $1 \le j \le p_i$ we have:
$\FF\bigl(\kS_i^{(j)}\bigr) \cong \kS_i^{(j)}$.
\item Assume that $p_1 \ge 3$. Then for any $2 \le j \le p_1-1$ we have: $\FF\bigl(\kS_1^{(j)}\bigr)
\cong  \kS_1^{(j+1)}$,
whereas $\FF\bigl(\kS_1^{(1)}\bigr)$ corresponds to the representation
$$
\xymatrix
{
\kk \ar[r]^{{1}} &
\kk \ar[r] &  0 \ar[r] & \dots \ar[r]   & 0
\ar@/^15pt/[llll]\\
}
$$
of the cyclic quiver $\overr{C}_{p_1}$. If $p_1 = 2$ then the point $\lambda_1$ is non-special for
$\YY$.
\item For any $l \in \ZZ$ we have:
$\FF\bigl(\kO_{\YY}(l \vec{c}')\bigr) \cong  \kO_{\XX}(l \vec{c})$.
\end{itemize}
 By what was said above it follows
that  the image of the reduced Drinfeld double $DU(\YY)$ of the composition algebra $U(\YY)$
under the injective algebra homomorphism
$DH(\YY) \stackrel{\FF}\lar DH(\XX)$ belongs to the algebra $DU(\XX)$. In the terms of Definition
\ref{D:compalgXX},  the action of $\FF$ on the generators of $DU(\YY)$ is the following
(for a simplicity  we assume that $p_1 \ge 3$):
\begin{enumerate}
\item For all $l \in \ZZ$ we have: $\FF\bigl([\kO_\YY(l \vec{c}')]^\pm\bigr) = \bigl[\kO_\XX(l \vec{c})\bigr]^\pm$.
\item For all $r \in \ZZ_{>0}$ we have: $\FF\bigl(T_r^\pm\bigr) = T_r^\pm$.
\item For all $2 \le i \le n$ and $1 \le j \le p_i$ we have:
$\FF\bigl(\bigl[\kS_i^{(j)}\bigr]^\pm\bigr) = \bigl[\kS_i^{(j)}\bigr]^\pm$.
\item We have: $
\FF\bigl(\bigl[\kS_1^{(1)}\bigr]^\pm\bigr) =  v^{-1}
\bigl[\kS_1^{(1)}\bigr]^\pm \circ \bigl[\kS_1^{(2)}\bigr]^\pm -\bigl[\kS_1^{(2)}\bigr]^\pm \circ \bigl[\kS_1^{(1)}\bigr]^\pm
$
and $\FF\bigl(\bigl[\kS_1^{(j)}\bigr]^\pm\bigr) = \bigl[\kS_1^{(j+1)}\bigr]^\pm$ for all
$2 \le j \le p_1-1$.
\item Finally,  $\FF(O) = O$; $\FF\bigl(K_i^{(j)}\bigr) = K_i^{(j)}$ for
$2 \le i \le n$ and $1 \le j \le p_i$; $F\bigl(K_1^{(1)}\bigr) = K_1^{(1)} K_1^{(2)}$ and
$\FF\bigl(K_1^{(j)}\bigr) = K_1^{(j+1)}$ for all $2 \le j \le p_1-1$.
\end{enumerate}
\end{proof}

\section{Composition subalgebra of a weighted projective line: Part II}\label{S:CompAlgII}

\noindent
The goal of this section is to clarify the structure of the subalgebra $U(\XX)_{\mathsf{tor}}$
of the composition algebra $U(\XX)$ and to derive some relations in the reduced
Drinfeld double $DU(\XX)$.

\begin{definition}
Let $\XX$ be a weighted projective line and
$x \in \XX$ be a closed point.
\begin{enumerate}
\item Assume $x$ is a non-special point of $\XX$  and $d = \deg(x)$ is the degree
of $x$. Then $\Tor_x(\XX)$ is equivalent to the category of representations
of the cyclic quiver $\overr{C}_1$ over the field $\FF_{q^d}$. We set
$
T_{r, x} = Z_{r, x} = d \frac{[r]}{r} p_{\frac{r}{d}},
$
where $p_{\frac{r}{d}} \in H\bigl(\Rep_{\FF_{q^d}}(\overr{C}_1)\bigr)$.
\item Assume $x$ is a special point of weight $p$. Then $\Tor_x(\XX)$ is equivalent to the
category of representations of a  cyclic quiver $\overr{C}_p$ and  we set
$$
T_{r, x} = \frac{[r]}{r} t_r \quad \mbox{and} \quad
Z_{r, x} = \frac{[r]}{r} \frac{1}{1-q^{-rp}} z_r.
$$
\end{enumerate}

\noindent
The following lemma is a straightforward corollary of
Theorem \ref{T:Hubery} and Theorem \ref{T:classicalHallprimit}.
\begin{lemma}\label{L:localvalueGreenForm}
Let $x$ be a closed point of $\XX$. Then for any $r \in \ZZ_{>0}$ we have:
$\bigl(T_{r, x}, T_{r,x}\bigr) = \frac{[r]^2}{r} \frac{d}{q^r-1}$, where $d$ is the degree of $x$.  Moreover,
if $x$ is a special point of $\XX$ then
$\bigl(Z_{r, x}, Z_{r,x}\bigr) = \frac{[r]^2}{r} \frac{1}{q^{rp}-1}$, where $p$ is the weight
of $x$.
\end{lemma}

\noindent
Next, for any $r \in \ZZ_{>0}$ we define:
\begin{equation}
T_r = \sum\limits_{d| r} \sum\limits_{x \in \XX: \, \deg(x) = d} T_{r,x} \quad \mbox{and} \quad
Z_r = \sum\limits_{d| r} \sum\limits_{x \in \XX: \, \deg(x) = d} Z_{r,x}.
\end{equation}
\end{definition}
\noindent
Note that this definition of the elements $T_r$ coincides with the one given
by Equation (\ref{E:generInTor(P1)2}) and Definition \ref{D:compalgXX}, see
 \cite[Section 5]{S1} for more details.

\begin{proposition}\label{P:structU(X)}
Let $\XX = \XX(\underline{p}, \underline{\lambda})$ be a weighted projective line. Then
the following statements are true.
\begin{enumerate}
\item The algebra $V(\XX)_{\mathsf{tor}}:=
\bigl\langle Z_r, [\kS_i^{(j)}], K_i^{(j)} \, \bigl| \, r \in \ZZ_{>0},  \, \bigr. 1 \le i \le n, 1 \le j \le p_i\bigr\rangle$
is equal to the algebra $U(\XX)_{\mathsf{tor}}$. In particular, we have a decomposition:
\begin{equation}\label{E:decompU(X)tor}
U(\XX)_{\mathsf{tor}} \cong   \mathcal{Z} \otimes_{\widetilde\QQ}
U_q^+(\widehat{\mathfrak{sl}}_{p_1}) \otimes_{\widetilde\QQ}
\dots \otimes_{\widetilde\QQ} U_q^+(\widehat{\mathfrak{sl}}_{p_n}) \otimes_{\widetilde\QQ}
\widetilde\QQ[K_{\mathsf{tor}}],
\end{equation}
where $\mathcal{Z} = \widetilde\QQ[Z_1, Z_2, \dots, Z_r, \dots ]$ is the ring of symmetric functions
and $K_{\mathsf{tor}}$ is the subgroup of $K = K_0(\XX)$ generated by the classes of the torsion sheaves.
The algebra $\mathcal{Z}$ is invariant under the action of the Picard group
$\Pic(\XX) \cong \LL(\underline{p})$.
\item For any $r \in \ZZ_{>0}$ we have the equalities:
$$
\alpha_r := (T_r, T_r) = \frac{1}{v^{-1}-v}\frac{[2r]}{r}
\quad \mbox{\rm and} \quad (T_r, \Theta_r) = \frac{[2r]}{r}.
$$
Moreover,
$
\beta_r:=
(Z_r, T_r) = (Z_r, Z_r) = \alpha_r + \sum\limits_{x \in \Pi} \mathrm{def}_x,
$
where $\mathrm{def}_x = \frac{\displaystyle [r]^2}{\displaystyle r}\Bigl(\frac{ \displaystyle 1}{
\displaystyle q^{rp_x}-1} - \frac{\displaystyle 1}{\displaystyle q^{r}-1}\Bigr)$ is the
\emph{defect} of the special point $x$ and
$\gamma_r := (Z_r, \Theta_r) = (v^{-1} - v) \beta_r$.
\end{enumerate}
\end{proposition}

\begin{proof}
Since for any $r \in \ZZ_{>0}$ the difference $Z_r - T_r$ belongs
to the algebra generated by the classes of the exceptional simple modules $\bigl[\kS_i^{(j)}\bigr]$,
$1 \le i \le n$, $1 \le j \le p_i$, we have the equality $V(\XX)_{\mathsf{tor}} =
U(\XX)_{\mathsf{tor}}$. Next,  $\mathcal{Z} = \widetilde\QQ[Z_1, Z_2, \dots, Z_r, \dots ]$
belongs to  the center of $U(\XX)_{\mathsf{tor}}$. Hence, the decomposition
(\ref{E:decompU(X)tor}) follows from Ringel's Theorem \ref{T:RingelonCyclicQuiver}.
Since for any closed point $x \in \XX$,  the element $Z_{r,x}$ is invariant under the action
of the Auslander-Reiten translation in the category $\Tor_x(\XX)$, the element
$Z_r$ is invariant under the action of the Picard group $\Pic(\XX)$.
The formulae for the value of the Green's  form on the generators $T_r$, $\Theta_r$
and $Z_r$ follow from the local formulae of Lemma \ref{L:localvalueGreenForm} and
the equality
$$
\sum\limits_{d|r} \sum\limits_{x \in \PP^1: \, \deg(x) = d} d =
\big|\PP^1(\FF_{q^r})\big| = (q^r +1).
$$
\end{proof}
\begin{corollary} Let $\XX$ be a  weighted projective line.
The reduced Drinfeld double $DU(\XX)_{\mathsf{tor}}$ is a subalgebra of the reduced Drinfeld double
$DU(\XX)$. Moreover, we have a decomposition:
\begin{equation}\label{E:decomp-ofDU(X)tor}
DU(\XX)_{\mathsf{tor}} \cong
\mathcal{H} \otimes_\kA U_q(\widehat{\mathfrak{sl}}_{p_1}) \otimes_{\kA}
\dots \otimes_{\kA} U_q(\widehat{\mathfrak{sl}}_{p_t}),
\end{equation}
where $\kA = \widetilde\QQ[C^\pm]$ is the ring of Laurent polynomials in the variable  $C := K_\delta$ and $\kH$ is the \emph{Heisenberg algebra}.
More precisely, $\kH := \widetilde\QQ \bigl\langle Z_r \, | \,  r \in \ZZ\setminus \{0\}\bigr\rangle  \otimes_{\widetilde\QQ} \kA$ subject to   the relations:
\begin{equation}\label{E:HeisenbAlg}
[Z_r, Z_t]= \delta_{r+t, 0} \, \beta_r \,  (C^{-r} - C^r) \quad
\mbox{\rm and} \quad
[Z_r, C^\pm] = 0 \quad r, t \in \ZZ\setminus \{0\}.
\end{equation}
\end{corollary}

\begin{proof}
We have the equalities
$\Delta(Z_r^\pm) = Z_r^\pm \otimes \mathbbm{1} + C^{\pm r} \otimes Z_r^\pm$  for any $r \in \ZZ_{>0}$.
Denoting  $Z_r^\pm  = Z_{\pm r}$ for $r \in \ZZ \setminus \{0\}$ it is easy to see  that
the first relation of (\ref{E:HeisenbAlg}) is just the structure relation $D(Z_r, Z_r)$
of the reduced Drinfeld double.
The decomposition (\ref{E:decomp-ofDU(X)tor}) is a corollary of Proposition \ref{P:DrinfDoubleStr}.
\end{proof}

\noindent
For the reader's convenience,  we give a proof of the following result from \cite{S1}.
\begin{proposition}\label{P:decompU(X)}
Let $\overline{U}(\XX)_{\mathsf{vec}}$ be the subalgebra of $U(\XX)$ generated
by the classes of the line  bundles on $\XX$ and $\overline{U}(\XX)_{\mathsf{tor}}$ be the
subalgebra of $U(\XX)$ generated by the elements $T_r \; (r \in \ZZ_{>0})$ and
$\bigl[\kS_i^{(j)}\bigr] \;
(1 \le i \le n, 1 \le j \le p_i)$. Then the canonical morphism
\begin{equation}\label{E:decompU(X)}
\overline{U}(\XX)_{\mathsf{vec}} \otimes_{\widetilde\QQ} \overline{U}(\XX)_{\mathsf{tor}}
\otimes_{\widetilde\QQ} \widetilde\QQ[K]
\xrightarrow{\mathsf{mult}} U(\XX)
\end{equation}
is an isomorphism of vector spaces over $\widetilde\QQ$.
\end{proposition}

\begin{proof}
Let $\overline{H}(\XX)_{\mathsf{vec}}$ be the subalgebra of the Hall algebra
$H(\XX)$ generated
by the classes of vector   bundles and $\overline{H}(\XX)_{\mathsf{tor}}$ be the
subalgebra of $H(\XX)$ generated by the classes of the torsion coherent sheaves.
Then the map $\overline{H}(\XX)_{\mathsf{vec}} \otimes_{\widetilde\QQ} \overline{H}(\XX)_{\mathsf{tor}}
\otimes_{\widetilde\QQ} \widetilde\QQ[K]
\xrightarrow{\mathsf{mult}} H(\XX)
$
is an isomorphism of vector spaces. Hence, the corresponding map $(\ref{E:decompU(X)})$ for the composition subalgebra
is at least injective. In order to show the surjectivity of $\mathsf{mult}$  first note that
for any line bundle $\kL$ and an arbitrary exceptional simple sheaf $\kS_i^{(j)}$ there exist
a line bundle $\kN$ and
constants $\alpha, \beta \in \widetilde\QQ$ such that
$\bigl[\kS_i^{(j)}\bigr] \circ \bigl[\kL\bigr] =  \alpha \bigl[\kN \bigl] +
\beta \bigl[\kL\bigr]
\circ \bigl[\kS_i^{(j)}\bigr]$.
Next, we  have the relation
$
T_r \circ \bigl[\kO(l\vec{c})\bigr] = \bigl[\kO(l\vec{c})\bigr] \circ T_r +
\frac{\displaystyle [2r]}{\displaystyle r} \bigl[\kO\bigl((l+r)\vec{c}\bigr)\bigr]
$
 for all $l \in \ZZ$ and $r \in \ZZ_{> 0}$.
 Recall that for any $r \in \ZZ_{> 0}$ the difference $Z_r - T_r$ belongs to the subalgebra
 of $U(\XX)_{\mathsf{tor}}$ generated by the exceptional simple sheaves $\bigl[\kS_i^{(j)}\bigr]$.
 Moreover, the element $Z_r$ is invariant under the action of $\Pic(\XX)$. Hence, the product
 $Z_r \circ [\kL]$ belongs to the image of $\mathsf{mult}$ for any
 $\kL \in \Pic(\XX)$ and $r \in \ZZ_{> 0}$. It remains to note, that by Proposition \ref{P:structU(X)}, the algebra
 $\overline{U}(\XX)_{\mathsf{tor}}$ is generated by the elements $Z_r$ $(r \in  \ZZ_{> 0})$ and
 $\bigl[\kS_i^{(j)}\bigr]$ $(1 \le i \le n, 1 \le j \le p_i)$. This concludes the proof.
\end{proof}

\begin{corollary}
Assume that $\kF_1, \dots, \kF_t$ are objects of $\VB(\XX)$ and $\alpha_1, \dots, \alpha_t \in
\widetilde\QQ$ are such
that  $a: = \sum\limits_{i=1}^t \alpha_i [\kF_i]$ belongs to $U(\XX)$. Then
$a \in U(\XX)_{\mathsf{vec}}$.
\end{corollary}

\begin{proof}
Let $\overline{H}(\XX) \stackrel{\mathsf{pr}}\lar \widetilde\QQ$ be the projection
on the class of the zero object $[0]$. We consider $\mathsf{pr}$ as an endomorphism
of $\overline{H}(\XX)$. Since $a \in \overline{U}(\XX)$, there exist
elements $b_1, \dots, b_t \in \overline{U}(\XX)_{\mathsf{vec}}$ and
$c_1, \dots, c_t \in \overline{U}(\XX)_{\mathsf{tor}}$ such that $a = b_1 c_1 + \dots +
b_t c_t$. Let $\mathsf{pr}(c_i) = \lambda_i [0]$ for some $\lambda_i \in \widetilde\QQ$.
Since there is a decomposition $\overline{H}(\XX) = \overline{H}(\XX)_{\mathsf{vec}} \otimes_{\widetilde\QQ}
\overline{H}(\XX)_{\mathsf{tor}}$, we have:
$
a = \bigl(\mathbbm{1} \otimes \mathsf{pr}\bigr)(a) = \sum\limits_{i=1}^t \lambda_i b_i \otimes [0],
$
where $\sum\limits_{i=1}^t \lambda_i b_i \in \overline{U}(\XX)_{\mathsf{vec}}$. This implies the claim.
\end{proof}

\noindent
Our next goal is to study the action of the central elements $Z_r$ of the algebra
$U(\XX)_{\mathsf{tor}}$ on the algebra $U(\XX)_{\mathsf{vec}}$.

\begin{theorem}
For any $\vec{x} \in \LL(\underline{p})$ and any $r \in \ZZ_{>0}$ we have the following equality:
\begin{equation}\label{E:HeckepropZr}
\Bigl[Z_r, \bigl[\kO(\vec{x})\bigr]\Bigr] = \gamma_r \, \bigl[\kO(\vec{x} + r \vec{c})\bigr],
\end{equation}
where $\gamma_r = (Z_r, \Theta_r)$ is the constant introduced in Proposition \ref{P:structU(X)}.
\end{theorem}

\begin{proof}
Since the Picard group $\Pic(\XX)$ acts on $U(\XX)$ by algebra automorphisms and
 the element $Z_r$ is stable under the this action,  it is sufficient to prove
 the equality (\ref{E:HeckepropZr}) in the special case  $\vec{x} = 0$.  By Proposition \ref{P:decompU(X)}
 there exist  $a, b, c \in \widetilde\QQ, \; T \in \bigl\langle [\kS_i^{(j)}], Z_1, \dots, Z_{r-1}\bigr\rangle,$ $c_{\vec\alpha} \in  \widetilde\QQ$ and
 $T_{\vec\alpha} \in \overline{U}(\XX)_{\mathsf{tor}}$ for
 $\vec\alpha \in \LL(\underline{p}), 0 < \vec\alpha < r \vec{c}$  such that
 $$
 Z_r \circ [\kO] = a \, [\kO(r\vec{c})] + \sum\limits_{0 \le \vec\alpha \le r \vec{c}}
 c_{\vec\alpha} \,  [\kO(r \vec{c} - \vec\alpha)] \circ T_{\vec\alpha} + b \, [\kO] \circ Z_r + c \,  [\kO] \circ T.
 $$
 We first show that $b = 1$ and $c = 0$. Let $Z_r = \sum\limits_{\kS} d_\kS [\kS]$, where we sum
 over the isomorphism classes of torsion sheaves of class $r\delta$ in the K--group
 $K_0(\XX)$. Assume we have an extension
 \begin{equation*}
 0 \lar \kO \lar \kO \oplus \kT \lar \kS \lar 0,
 \end{equation*}
 where $\kT$ is a torsion sheaf. Since $\Hom(\kO, \kO) = \kk$, it follows that
 $\kT \cong \kS$ and the above  sequence splits. Note that $\Ext^1(\kO, \kS) = 0$,
 hence  $\bigl\langle \overline\kO, \overline\kS\bigr\rangle = \dim_\kk\bigl(\Hom(\kO, \kS)\bigr)  =  r$ and $F^{\kO \oplus \kS}_{\kS, \; \kO} = q^r$.  This implies that
 $$
 \bigl(\sum\limits_{\kS} d_\kS [\kS]\bigr) \circ [\kO] =
 v^{-r} \sum\limits_{\kS} d_\kS \bigl[\kO \oplus \kS\bigr] +
 \mbox{\rm \; terms involving \; } [\kO(\vec{\gamma})], \; \vec{\gamma} >0.
 $$
 On the other hand,
 $
 [\kO] \circ \bigl(\sum_{\kS} d_\kS [\kS]\bigr)  = v^{-r}
 \sum_{\kS} d_\kS \bigl[\kO \oplus \kS\bigr].
 $
 Hence, $b = 1$ and $c = 0$ as stated. Our next goal is to show that $c_{\vec{\alpha}} = 0$ for all
 $0 < \alpha < r \vec{c}$. Recall that
 $$
\Delta\bigl([\kO]\bigr) = [\kO] \otimes \mathbbm{1} + K_{\overline\kO} \otimes [\kO] +
\sum\limits_{\vec{x} > 0} \widetilde{\Theta}_{\vec{x}}
\otimes \bigl[\kO(-\vec{x})\bigr] \quad \mbox{\rm and} \quad
\Delta(Z_r) = Z_r \otimes \mathbbm{1} + K_{r \, \delta} \otimes Z_r,
$$
where in the notations of the equation (\ref{E:Coprod(O)}) we denote
$\widetilde{\Theta}_{\vec{x}}  =  \Theta_{\vec{x}} K_{\overline{\kO(-\vec{x})}}$.  Then we have:
$$
\Delta\Bigl(\bigl[Z_r, [\kO]\bigr]\Bigr) =
\Bigl[\Delta(Z_r), \Delta([\kO])\Bigr] = \Bigl[
Z_r \otimes \mathbbm{1} + K_{r \, \delta} \otimes Z_r,
[\kO] \otimes \mathbbm{1} +
\sum\limits_{\vec{x} \ge  0} \widetilde{\Theta}_{\vec{x}}
\otimes \bigl[\kO(-\vec{x})\bigr]
\Bigr] =
$$
$$ =
\bigl[Z_r, [\kO]\bigr] \otimes \mathbbm{1} + K_{r \delta + \bar\kO} \otimes \bigl[Z_r, [\kO]\bigr]
+ \sum\limits_{\vec{x} >0} K_{r\delta} \widetilde{\Theta}_{\vec{x}} \otimes \bigl[Z_r, [\kO(-\vec{x})]\bigr].
$$
Here we benefit from the fact that $\bigl[Z_r, \widetilde{\Theta}_{\vec{x}}\bigr] = 0$ because
the elements $Z_r$ are central in $U(\XX)_{\mathsf{tor}}$. On the other hand,
we have already shown that
$$
\bigl[Z_r, [\kO]\bigr] = a \, [\kO(r \vec{c})] + \sum\limits_{0 < \vec{\alpha} < r \vec{c}}
c_{\vec\alpha} \, [\kO(r\vec{c} - \vec{\alpha})] \circ T_{\vec{\alpha}}.
$$
Apply the operator $\Delta$ to the right-hand side of this equality. Let $0 < \vec\alpha < r \vec{c}$
be a maximal element such that $c_{\alpha} \ne 0$. Then
$\Delta\bigl(a \, [\kO(r \vec{c})] + \sum\limits_{0 < \vec{\alpha} < r \vec{c}}
c_{\vec\alpha} \, [\kO(r\vec{c} - \vec{\alpha})] \circ T_{\vec{\alpha}}\bigr)$
contains a summand $c_{\vec\alpha} [\kO(r \vec{c} - \vec{\alpha})] K_{\beta_2}
\otimes T_{\vec{\alpha}}$ and
the remaining part of the comultiplication  has no contributions to the graded piece
$U(\XX)[\beta_1] \otimes U(\XX)[\beta_2]$, where $\beta_1 = \overline{\kO(r \vec{c} - \vec{\alpha})} \in
K_0(\XX)$
and $\beta_2 \in K_0(\XX)$ is the class of the  summands  of  the element $T_{\vec{\alpha}}$.
Contradiction.
Hence,
$c_{\vec\alpha} =  0$ for all $0 < \vec{\alpha} < r \vec{c}$ and we have the equality $\bigl[Z_r, [\kO]\bigr] = a \, [\kO(r \vec{c})]$ for a certain constant $a \in \widetilde\QQ$. Our last goal is to determine
$a$ explicitly. Note that
$$
a \bigl([\kO(r \vec{c})], [\kO(r \vec{c})]\bigr) =
\bigl(Z_r \circ [\kO] - [\kO] \circ Z_r, [\kO(r \vec{c})]\bigr) =
\Bigl(Z_r  \otimes  [\kO], \Delta\bigl([\kO(r \vec{c})]\bigr)\Bigr)
$$
$$
= \bigl(Z_r, \Theta_r\bigr)\bigl([\kO], [\kO]\bigr) = \gamma_r \bigl([\kO], [\kO]\bigr).
$$
Here we use the vanishing $\bigl([\kO] \circ Z_r, [\kO(r \vec{c})]\bigr) = 0$ following from fact that  $\kO$ is not  a quotient of $\kO(r \vec{c})$ for any $r \in \ZZ_{>0}$.
Hence, $a = \gamma_r$ as stated. Theorem is proven.
\end{proof}

\begin{lemma}
For any $\vec{x} \in \LL(\underline{p})$ and $r \in \ZZ_{>0}$ we have the following equality
in the reduced Drinfeld double $DU(\XX)$:
$
\bigl[Z_r^-, [\kO(\vec{x})]^+ \bigr] =  \gamma_r [\kO(\vec{x} - r \vec{c})] C^{-r}.
$
\end{lemma}

\begin{proof}
Recall that $\Delta\bigl([\kO(\vec{x})]^+\bigr) = [\kO(\vec{x})]^+
\otimes \mathbbm{1} + K_{\overline{\kO(\vec{x})}} \otimes [\kO(\vec{x})]^+ +
\sum\limits_{\vec{y} > 0} \widetilde{\Theta}_{\vec{y}}^+
\otimes \bigl[\kO(\vec{x}-\vec{y})\bigr]^+$ and
$\Delta(Z_r^-) = Z_r^- \otimes \mathbbm{1} + C^{-r} \otimes Z_r^-$. Note that $\Theta_{\vec{y}}$
contains contributions of sheaves of class $r \delta$ only for $\vec{y} = r \vec{c}$. Moreover,
$\Theta_{r \vec{c}} = \Theta_r$ and $(\Theta_r, Z_r) = \gamma_r$. This implies
the result.
\end{proof}

\begin{remark}
By Theorem \ref{T:morphDDinducedbyIndweights} we know that  for any $l, t \in \ZZ$
 the following equalities are true in the reduced
Drinfeld double $DU(\XX)$:
$$
\bigl[[\kO(l\vec{c})]^+, [\kO(t\vec{c})]^-\bigr] =
\left\{
\begin{array}{lcc}
 \frac{\displaystyle v}{\displaystyle v- v^{-1}} \, {\Theta}_{l-t}^+ O C^{t} & \mbox{if} & l > t, \\
\hspace{1.5cm} 0 & \mbox{if} & l =  t,\\
\frac{\displaystyle v}{\displaystyle v^{-1}-v} \, {\Theta}_{t-l}^- O^{-1} C^{-l} & \mbox{if} & l < t.
\end{array}
\right.
$$
Unfortunatelly,  we have not succeeded to find  an explicit formula expressing the elements
$\Theta_r$ through the generators $Z_t$ and $\bigl[\kS_{i}^{(j)}\bigr]$ of the algebra
$U(\XX)_{\mathsf{tor}}$,  although  by Proposition \ref{P:onU(X)tor} it is known that such a formula does exist.
\end{remark}

\noindent
\textbf{Summary}. In this subsection we collect  the main structure results on the
composition algebra of a weighted projective line.

Let $\XX$ be a weighted projective line of type $(\underline\lambda, \underline{p})$,
where $\underline\lambda = (\lambda_1, \dots, \lambda_n) \in \PP^1(\kk)$ is a sequence
of pairwise distinct points and $\underline{p} = (p_1, \dots, p_n) \in \ZZ_{>0}$ is a sequence
of weights.  Let $\LL(\underline{p})$ be the abelian group generated by the elements
$\vec{x}_1, \dots, \vec{x}_n$ subject to the relations
$p_1 \vec{x}_1 = \dots = \vec{x}_n = \vec{c}$. Next, let $K= K_0(\XX)$ be the K--group
of the category $\Coh(\XX)$ and for any coherent sheaf  $\kF$
let $K_{\overline\kF}$ be the corresponding element of the Cartan part $\widetilde\QQ[K]$ of the Hall algebra
$H(\XX)$.

\medskip
\noindent
1. For any $l \in \ZZ$ we denote  $L_l = \bigl[\kO(l\vec{c})\bigr] \in H(\XX)$. Then
the composition algebra $U(\XX)$ is the subalgebra of the Hall algebra $H(\XX)$
defined as follows:
$$
  U(\XX) := \Bigl\langle L_l, \, T_r, \, \bigl[\kS_i^{(j)}\bigr], \,  O, K_{i}^{(j)}, C \; \Bigl| \;
  l \in \ZZ, \, r \in \ZZ_{>0}, \, 1 \le i \le n, \, 1 \le j \le p_i  \Bigr.
  \Bigr\rangle.
  $$
  Here $K_{i}^{(j)} = K_{\overline{\kS_i^{(j)}}}$, $O = K_{\overline\kO}$  and
  $C = K_i^{(1)} \dots K_i^{(p_i)}$ for any $1 \le i \le n$.
Note that  the algebra $U(\XX)$  only depends  on the weight sequence
$\underline{p}$ and does \emph{not} depend on the set $\underline{\lambda}$.

\medskip
\noindent
2. The subalgebra $V(\XX) := \bigl\langle L_l, \, T_r, \,  O, C \; \bigl| \;
  l \in \ZZ, \, r \in \ZZ_{>0}  \bigr.
  \bigr\rangle
  $ of the composition algebra $U(\XX)$
  is isomorphic to the composition algebra of a non-weighted projective line $\PP^1$.
The elements  $L_l, T_r$, $O$ and $C$
satisfy  the following relations:
\begin{enumerate}
\item $C$ is central;
\item $[O, T_t] = 0 = [T_t, T_r]$ for all $r, t  \in \ZZ_{> 0}$;
\item $O L_l = v^{-2} L_l O$ for all $l \in \ZZ$;
\item $\bigl[T_r, L_l\bigr] = \frac{\displaystyle [2r]}{\displaystyle r} L_{l+r}$ for all $l \in \ZZ$ and $r \in \ZZ_{> 0}$;
\item $L_{m} L_{l+1}  + L_{l}L_{m+1} =  v^2 \bigl(L_{l+1} L_{m} + L_{m+1} L_{l})$ for all $m, l \in
\ZZ$.
\end{enumerate}
Moreover, this is a complete list of relations of the subalgebra $V(\XX)$.

\medskip
\noindent
3. The algebra $U(\XX)$ is a topological bialgebra. Let $\overline{U}(\XX)$ be the
subalgebra of $U(\XX)$ generated by $L_l, T_r$ and $\bigl[\kS_i^{(j)}\bigr]$ for
$l \in \ZZ$, $r \in \ZZ_{>0}$ and $(i, j)$ such that $1 \le i \le n, 1 \le j \le p_i$. Then we have
a triangular decomposition:
$
DU(\XX) = \overline{U}(\XX)^+ \otimes_{\widetilde\QQ} \widetilde\QQ[K] \otimes_{\widetilde\QQ}
\overline{U}(\XX)^-,
$
where $DU(\XX)$ is the reduced Drinfeld double of $U(\XX)$.
Moreover,
there is an injective algebra homomorphism
$$DV(\XX) = \overline{V}(\XX)^+ \otimes_{\widetilde\QQ} \widetilde\QQ[\widetilde{K}] \otimes_{\widetilde\QQ}
\overline{V}(\XX)^- \lar   \overline{U}(\XX)^+ \otimes_{\widetilde\QQ} \widetilde\QQ[K] \otimes_{\widetilde\QQ}
\overline{U}(\XX)^- = DU(\XX)
$$ respecting the triangular decompositions of $DV(\XX)$ and $DU(\XX)$.
Here $\widetilde\QQ[\widetilde{K}] =
\widetilde\QQ[O^\pm, C^\pm]$ is the group algebra of the subgroup of $K_0(\XX)$ generated
by $O = \overline\kO$ and the class $\delta$ of a simple torsion sheaf supported at a non-special
$\kk$--point
of $\XX$.

\medskip
\noindent
4. For any $\vec{x} \in \LL(\underline{p})$ the element $\bigl[\kO(\vec{x})]$ belongs
to $U(\XX)$. Let $\overline{U}(\XX)_{\mathsf{vec}}$ be the subalgebra
of $U(\XX)$ generated by the classes of line bundles and $\overline{U}(\XX)_{\mathsf{tor}}$
be the subalgebra generated by $T_r, \bigl[\kS_i^{(j)}\bigr]$ for
$r \in \ZZ_{>0}$ and $(i, j): 1 \le i \le n, 1 \le j \le p_i$. Then the canonical morphism
$$
\overline{U}(\XX)_{\mathsf{vec}} \otimes_{\widetilde\QQ} \overline{U}(\XX)_{\mathsf{tor}} \otimes_{\widetilde\QQ} \widetilde\QQ[K]
\xrightarrow{\mathsf{mult}} U(\XX)
$$
is an isomorphism of vector spaces over $\widetilde\QQ$.

\medskip
\noindent
5. Let ${U}(\XX)_{\mathsf{tor}} = \overline{U}(\XX)_{\mathsf{tor}}
\otimes_{\widetilde\QQ} \widetilde\QQ[K_{\mathsf{tor}}]$, where
$K_{\mathsf{tor}}$ is the subgroup of $K_0(\XX)$ generated by the classes of torsion sheaves.
Then ${U}(\XX)_{\mathsf{tor}}$ is a Hopf algebra. Moreover, it is a subbialgebra of the composition
algebra $U(\XX)$. The  algebra $\overline{U}(\XX)_{\mathsf{tor}}$ decays  into a tensor product:
$$
\overline{U}(\XX)_{\mathsf{tor}}  = \mathcal{Z} \otimes_{\widetilde\QQ} \overline{U}(\XX)_{\mathsf{tor}}^{\mathsf{exc},\,  1}
\otimes_{\widetilde\QQ} \dots \otimes_{\widetilde\QQ} \overline{U}(\XX)_{\mathsf{tor}}^{\mathsf{exc}, \, n}.
$$
Here $\mathcal{Z} = \widetilde\QQ[Z_1, Z_2, \dots, Z_r, \dots]$ is the ring of symmetric functions.
It is generated by the elements $Z_r \in \overline{U}(\XX)_{\mathsf{tor}}[r\delta]$, which
are central in ${U}(\XX)_{\mathsf{tor}}$ and primitive. The last condition means that
$
\Delta(Z_r) = Z_r \otimes \mathbbm{1} + C^r \otimes Z_r
$
for any $r \in \ZZ_{>0}$.
Moreover, the Picard group $\Pic(\XX) \cong \LL(\underline{p})$
acts trivially on the algebra $\mathcal{Z}$.

For any $1 \le i \le n$,  the algebra $\overline{U}(\XX)_{\mathsf{tor}}^{\mathsf{exc},\,  i}$ is
generated by the exceptional simple torsion sheaves
 $\bigl[\kS_i^{(1)}\bigr], \dots, \bigl[\kS_i^{(p_i)}\bigr]$ and  is isomorphic to the positive part of the quantized enveloping algebra
$U_q^+(\widehat{\mathfrak{sl}}_{p_i})$. The algebra ${U}(\XX)_{\mathsf{tor}}^{\mathsf{exc},\,  i}$,
generated by $\overline{U}(\XX)_{\mathsf{tor}}^{\mathsf{exc},\,  i}$ and the elements
$K_i^{(j)}$ for $1 \le j \le p_i$,  is a Hopf subalgebra of ${U}(\XX)_{\mathsf{tor}}$.
All these Hopf algebras ${U}(\XX)_{\mathsf{tor}}^{\mathsf{exc},\,  i}$ are embedded
in the same Hopf algebra ${U}(\XX)_{\mathsf{tor}}$ and ``share''  the ``same'' central
part $\kA = \widetilde\QQ[C^+, C^-]$.
The reduced Drinfeld double $D{U}(\XX)_{\mathsf{tor}}$ decomposes into a tensor
product of algebras:
$$
D{U}(\XX)_{\mathsf{tor}} = \mathcal{H} \otimes_{\kA} U_q(\widehat{\mathfrak{sl}}_{p_1})
\otimes_{\kA} \dots \otimes_{\kA} U_q(\widehat{\mathfrak{sl}}_{p_n}),
$$
where $\mathcal{H}$ is the Heisenberg algebra.

\medskip
\noindent
6. More precisely, for any $r \in \ZZ_{>0}$ the difference
$Z_r - T_r$ belongs to the subalgebra
$\overline{U}(\XX)_{\mathsf{tor}}^{\mathsf{exc}} := \overline{U}(\XX)_{\mathsf{tor}}^{\mathsf{exc},\,  1}
\otimes_{\widetilde\QQ} \dots \otimes_{\widetilde\QQ} \overline{U}(\XX)_{\mathsf{tor}}^{\mathsf{exc},\,  n}.
$ Next, Green's inner product takes the following values:
$
\beta_r:=
(Z_r, T_r) = (Z_r, Z_r) = \alpha_r + \sum\limits_{x \in \underline\lambda} \mathrm{def}_x,
$
where $\mathrm{def}_x = \frac{\displaystyle [r]^2}{\displaystyle r}\Bigl(\frac{ \displaystyle 1}{
\displaystyle q^{rp_x}-1} - \frac{\displaystyle 1}{\displaystyle q^{r}-1}\Bigr)$ is the
{defect} of a  special point $x$ of weight $p_x$ and
$\gamma_r := (Z_r, \Theta_r) = (v^{-1} - v) \beta_r$.
 The Heisenberg algebra
 $\kH$ is generated over $\widetilde\QQ$ by the elements
$\bigl\{Z_r\bigr\}_{r \in \ZZ\setminus \{0\}}$ and $C^{\pm 1}$  subject to   the relations:
$
[Z_r, Z_t]= \delta_{r+t, 0} \, \beta_r \,  (C^{-r} - C^r)$
and $
[Z_r, C^\pm] = 0$ for $ r, t \in \ZZ\setminus \{0\}.
$

\medskip
\noindent
7. For any $r \in \ZZ_{>0}$ and $\vec{x} \in \LL(\underline{p})$ we have the following equalities
in $DU(\XX)$:
$$
\Bigl[Z_r^\pm, \bigl[\kO(\vec{x})\bigr]^\pm\Bigr] = \gamma_r \, \bigl[\kO(\vec{x} + r \vec{c})\bigr]^\pm
\quad
\mbox{\rm and}
\quad \Bigl[Z_r^-, \bigl[\kO(\vec{x})\bigr]^+\Bigr] = \gamma_r \, \bigl[\kO(\vec{x} - r \vec{c})\bigr]
C^{-r}.
$$

\medskip
\noindent
8. Let $\YY$ be a weighted projective line of type $(\underline{\mu}, \, \underline{q})$.
Assume that $(\underline{\lambda}, \, \underline{p})$ dominates $(\underline{\mu}, \, \underline{q})$.
This means that $\underline{\mu}$ is a subset of  $\underline{\lambda}$ and for
any $i \in \underline{\mu}$ we have: $q_i \le p_i$. Then  there
exists an injective algebra homomorphism
$$
\overline{U}(\YY)^+ \otimes_{\widetilde\QQ} \widetilde\QQ[K_0(\YY)] \otimes_{\widetilde\QQ}
\overline{U}(\YY)^- \stackrel{\FF}\lar
\overline{U}(\XX)^+ \otimes_{\widetilde\QQ} \widetilde\QQ[K_0(\XX)] \otimes_{\widetilde\QQ}
\overline{U}(\XX)^-
$$
preserving the triangular decompositions.  Consider the following basic cases.

\vspace{1mm}
\noindent
\underline{Case 1}. Let $\underline{\mu} = \underline{\lambda} \setminus \{\lambda_1\} =
\{\lambda_2, \dots, \lambda_n\}$ and $q_i = p_i$ for all $2 \le i \le n$. Then
the homomorphism $U(\YY) \stackrel{\FF}\lar U(\XX)$ maps the generators
$O$, $C$, $L_l$, $T_r$, $\bigl[\kS_i^{(j)}\bigr]$ and $K_i^{(j)}$ of the algebra $U(\YY)$
to  the generators of $U(\XX)$ denoted  by the same symbols for all $l \in \ZZ$,  $r \in \ZZ_{>0}$ and all
 $(i, j)$ such that $2 \le i \le n$ and $1 \le j \le p_i$.

 \vspace{1mm}
\noindent
\underline{Case 2}. Assume $\underline{\mu} = \underline{\lambda}$
and $q_i = p_i$ for all $2 \le i \le n$ whereas $p_1 = q_1+1 \ge 3$.
Then $\FF$ maps the generators $O$, $C$, $L_l \, (l \in \ZZ)$, $T_r \, (r \in \ZZ_{>0})$,
$\bigl[\kS_i^{(j)}\bigr]$ and $K_i^{(j)}$  $(2 \le i \le n, 1 \le j \le p_i)$  of $U(\YY)$ to the
generators of $U(\XX)$ denoted  by the same symbols.
Next,  $
\FF\bigl(\bigl[\kS_1^{(1)}\bigr]^\pm\bigr) =  v^{-1}
\bigl[\kS_1^{(1)}\bigr]^\pm \circ \bigl[\kS_1^{(2)}\bigr]^\pm -\bigl[\kS_1^{(2)}\bigr]^\pm \circ \bigl[\kS_1^{(1)}\bigr]^\pm
$
and $\FF\bigl(\bigl[\kS_1^{(j)}\bigr]^\pm\bigr) = \bigl[\kS_1^{(j+1)}\bigr]^\pm$ for all
$2 \le j \le q_1$.

\section{Composition algebra of  domestic and tubular weighted projective lines}
In the previous  section, we derived some general properties of the composition algebra
of a weighted projective line $\XX$. In this section,  we deal with two most important cases:
domestic and tubular weighted projective lines.
The composition algebra of a domestic weighted projective line was also studied in a recent article
of Dou, Jiang and Xiao \cite{PapierBarbare}.

We start with a reminder of some
results on the reflection functors and the Hall algebras for quiver representations.

\subsection{Reflection functors and Hall algebras of quiver representations}

Let $\overr{\Delta} = (\Delta_0, \Delta_1, s, t)$ be a finite quiver without loops and oriented cycles.
Here $\Delta_0$ is the set of vertices of $\overr{\Delta}$, $\Delta_1$ is its set of arrows and
$s, t: \Delta_1 \rightarrow \Delta_0$ are the maps assigning to an arrow its source and target respectively.
In this section we consider  the Hall algebra of the category $\cA = \Rep(\kk\overr{\Delta})$ of
finite dimensional representations
of $\overr{\Delta}$ over a $\kk$. If $A$ is  the path algebra of $\overr{\Delta}$, then the categories
$A-\mod$ and $\cA$ are equivalent.

\vspace{2mm}
\noindent
The following proposition is due to Happel, see  \cite[Section 4.6]{Happel}.

\begin{proposition}\label{P:SerreFunct} Let $\DD := \Hom_\kk(\,-\,,\kk)$ be the duality
over $\kk$. Then
the derived functor $$\tau_D:= \LL\bigl(\DD\Hom_A(\,-\,,A)\bigr)[-1]: D^b(\cA) \lar  D^b(\cA)$$ satisfies
the following property: for any $X, Y \in \mathrm{Ob}\bigl(D^b(\cA)\bigr)$ we have an isomorphism
$$
\Hom_{D^b(\cA)}\bigl(X, \tau_D(Y)\bigr) \lar  \DD \Ext^1_{D^b(\cA)}(Y, X),
$$
functorial in both arguments. In other words, $\tau_D[1]$ is the Serre functor
of  $D^b(\cA)$.
\end{proposition}

\begin{remark}\label{R:SerreFexplicit}
Let $i \in \Delta_0$ be a vertex and $P_i = A e_i$ be the  indecomposable
projective module, which is the projective cover of the simple module $S_i$. Then
$\tau_D(P_i) = I_i[-1]$, where $I_i$ is the injective envelope of $S_i$.
\end{remark}

\begin{definition}
The left  exact functor $\tau = \tau^+ := \DD\Ext_A^1(\,-\,,A): \,\,  \cA \rightarrow  \cA$ is called the Auslander--Reiten translation.
\end{definition}

\begin{proposition}\label{P:DervsNonder} In the above  notations we have:
\begin{enumerate}
\item the functor $\tau$ is isomorphic to the composition
  $\cA \xrightarrow{\mathsf{can}} D^b(\cA) \xrightarrow{\tau_D} D^b(\cA) \xrightarrow{H^0} \cA;
  $
  \item Assume that there are no non-zero objects in $\cA$ which are both projective and injective.
  Then we have an isomorphism of triangle  functors $\tau_D \cong \RR \tau$;
\item
The functor $\tau^+$  has a left  adjoint functor
$\tau^- = \Ext^1_A\bigl(\DD(A), \,-\,\bigr)$. Moreover, for any objects
$X, Y \in \mathrm{Ob}(\cA)$ we have bi-functorial isomorphisms:
$$
\Hom_\cA\bigl(X, \tau^+(Y)\bigr) \cong \DD \Ext^1_\cA(Y, X) \cong
\Hom_\cA\bigl(\tau^{-}(X), Y\bigr).
$$
\end{enumerate}
\end{proposition}

\begin{proof}
The first part of this proposition is trivial. To show the second statement,
consider the  right exact
functor $\nu = \DD \Hom_A(\,-\,,A)$. Note that $\tau \cong  \LL_1(\nu)$.
Let $I_i$ be an indecomposable  injective module and $P_j$ be an indecomposable projective module
corresponding to the vertices $i, j \in \Delta_0$ respectively.
Since $I_i$ is non-projective, by \cite[Section 4.7]{Happel} we have:
$\tau_D(I_i) = \tau(I_i) =: X \in \mathrm{Ob}(\cA)$. Hence, we have:
$$
\Hom_A(I_i, P_j) \cong \Hom_{D^b(\cA)}\bigl(\tau_D(I_i), \tau_D(P_j)\bigr) \cong
\Hom_{D^b(\cA)}\bigl(X, I_j[-1]\bigr) = 0.
$$
Then by \cite[Proposition I.7.4]{RD} we have: $\LL \nu \cong \RR \tau[1]$.

Let $Y$ be an object of $\cA$. Then $\tau_D(Y)$ is a complex with at most two non-vanishing
cohomologies. Moreover, $H^0\bigl(\tau_D(Y)\bigr) \cong  \tau(Y)$ and
$H^1\bigl(\tau_D(Y)\bigr) \cong  \nu(Y)$. Using Proposition \ref{P:SerreFunct} we have:
$$
\DD \Ext^1_\cA(X, Y) \cong \Hom_{D^b(\cA)}\bigl(X, \tau_D(Y)\bigr)
\cong \Hom_{\cA}\Bigl(X, H^0\bigl(\tau_D(Y)\bigr)\Bigr) \cong
\Hom_\cA\bigl(X, \tau^+(Y)\bigr),
$$
where all the isomorphisms are bi-functorial. The proof of the second isomorphism  is similar.
\end{proof}

\begin{remark}
Let $\overr{\Delta} = (1 \lar 2)$ be a quiver of type $A_2$. Then the module
$I: = \bigl(\kk \stackrel{1}\lar \kk\bigr)$ is both projective and injective.
In particular, we have: $\RR \tau(I) \cong \tau(X) = \DD \Ext_A(I, A) = 0.$
 Since
$\tau_D$ is an auto-equivalence of $D^b(\cA)$ and $\RR \tau$ is not,
we have: $\tau_D \not\cong \RR \tau$.
\end{remark}

\begin{definition}\label{D:reflfunct}
Let $\overr{\Delta}$ be a finite quiver without loops and oriented cycles and
$* \in \Delta_0$ be a sink (i.e. there is no arrow $\alpha \in \Delta_1$ such that $s(\alpha) = *$). Let
$\overl{\Delta}$ be the quiver obtained from $\overr{\Delta}$ by inverting all the arrows
ending at $*$.  Recall, that we have an adjoint pair of the so-called \emph{reflection
functors}  of Bernstein, Gelfand and Ponomarev \cite{BGP}
$$\SS_{*}^{+}: \Rep(\overr{\Delta}) \lar
 \Rep(\overl{\Delta}) \quad \mbox{and} \quad \SS_{*}^{-}: \Rep(\overl{\Delta}) \lar
 \Rep(\overr{\Delta})$$
  defined as follows.
 For any object $X = \bigl((V_i)_{i \in Q_0}, (A_\alpha)_{\alpha \in Q_1}\bigr) \in
 \mathrm{Ob}\bigl(\Rep({\overr{\Delta}})\bigr)$ consider
 the exact sequence of vector spaces
 $$
 0 \lar W_* \xrightarrow{\oplus_{\alpha \in Q_1: \, t(\alpha) = *} B_\alpha}
 \bigoplus\limits_{\alpha \in Q_1: \, t(\alpha) = *}V_{s(\alpha)} \xrightarrow{\oplus_{\alpha \in Q_1: \,  t(\alpha) = *} A_\alpha} V_*.
 $$
Then the representation  $Y = \SS_*^{+}(X) = \bigl((U_i)_{i \in Q_0}, (C_\alpha)_{\alpha \in Q_1}\bigr)\in \mathrm{Ob}\bigl(\Rep({\overl{\Delta}})\bigr)$ is defined by:
$$
U_i =
\left\{
\begin{array}{ccc}
V_i & \, \mbox{if}\, & \, i \ne * \\
W_* & \, \mbox{if}\, & \,  i = *
\end{array}
\right.
\quad \mbox{\rm and} \quad
C_\alpha =
\left\{
\begin{array}{ccc}
A_\alpha & \, \mbox{if}\, & \, s(\alpha) \ne * \\
B_\alpha & \, \mbox{if}\, & \,  s(\alpha) =  *.
\end{array}
\right.
$$
The definition of the adjoint  reflection functor $\SS_*^{-}$ is dual, see \cite{BGP}.
\end{definition}

\noindent
The following result is well-known, see for example \cite[Section VII.5]{ASS}.

\begin{theorem}\label{T:basicsonrefl}
The derived  functors $ D^b\bigl(\Rep(\overr{\Delta})\bigr)
\stackrel{\RR \SS_*^{+}}\lar  D^b\bigl(\Rep(\overl{\Delta})\bigr)$  and $ D^b\bigl(\Rep(\overl{\Delta})\bigr)
\stackrel{\LL \SS_*^{-}}\lar  D^b\bigl(\Rep(\overr{\Delta})\bigr)$
are mutually inverse equivalences
of triangulated categories. Moreover, for an indecomposable object $X \in \mathrm{Ob}\bigl(\Rep({\overr{\Delta}})\bigr)$ we have:
$$
\RR^{1} \SS_*^{+}(X) =
\left\{
\begin{array}{ccc}
0  & \, \mbox{if}\, & \, X \not\cong S_* \\
S_* & \, \mbox{if}\, & \, X \cong S_*,
\end{array}
\right.
$$
where $S_*$ is the simple module corresponding to the vertex  $*$.
Next, for an indecomposable object $Y \in \mathrm{Ob}\bigl(\Rep({\overl{\Delta}})\bigr)$ we have:
$$
\LL^{-1} \SS_*^{-}(Y) =
\left\{
\begin{array}{ccc}
0  & \, \mbox{if}\, & \, Y \not\cong S_* \\
S_* & \, \mbox{if}\, & \, Y \cong S_*.
\end{array}
\right.
$$
In particular, the reflection functors $\SS_*^{+}$ and $\SS_*^{-}$ yield
  mutually inverse equivalences between the categories $\Rep\bigl(\overr{\Delta}\bigr)^\circ$ and $\Rep\bigl(\overl{\Delta}\bigr)^\circ$, which
are the full subcategories of $\Rep\bigl(\overr{\Delta}\bigr)$ and $\Rep\bigl(\overl{\Delta}\bigr)$
consisting of objects without direct summands isomorphic to $S_*$.
\end{theorem}

\noindent
The following fundamental result gives  a link between  reflection functors
and Auslander--Reiten translations.

\begin{theorem}\label{T:CoxetervsTau}
Let $\overr{\Delta}$ be a finite quiver without loops and oriented cycles with a prescribed labeling
$\Delta_0 = \{1,2, \dots,n\}$. Assume that for any $1 \le i < j \le n$ there is no oriented path starting at $j$ and ending at $i$ (such a labeling is called admissible). Then there is  an isomorphism of functors
$
\tau \cong  \TT \circ \SS_1^+ \circ \dots \circ \SS_{n}^+,
$
where $\TT: \Rep(\overr{\Delta}) \rightarrow  \Rep(\overr{\Delta})$ is
defined by the following rule.
For $X = \bigl((V_i)_{i \in \Delta_0}, (A_\alpha)_{\alpha \in \Delta_1}\bigr) \in
 \mathrm{Ob}\bigl(\Rep{\overr{\Delta}}\bigr)$ we set: $\TT(X) = \bigl((U_i)_{i \in Q_0}, (B_\alpha)_{\alpha \in Q_1}\bigr)$, where
$
U_i = V_i  \quad \mbox{and} \quad B_\alpha = - A_\alpha.
$
In a similar way, we have an isomorphism of functors: $
\tau^- = \TT \circ \SS_n^- \circ \dots \circ \SS_{1}^-.
$ In particular, the Coxeter functors $\AA^+ = \SS_1^+ \circ \dots \circ \SS_{n}^+$ and
$\AA^- = \SS_n^- \circ \dots \circ \SS_{1}^-$ do not depend on the choice
of an admissible labeling of vertices.
\end{theorem}

\noindent
For a proof of this result, we refer to \cite[Section 5.3]{Gabriel} and
 \cite[Proposition II.3.2]{Tepetla}.

\medskip
\noindent
The following important statement  seems to be well-known. Nevertheless, we have not succeeded to find
its proof in the literature and therefore give it here.

\begin{theorem}\label{P:aboutDerRefl}
Let $\overr{\Delta}$ be a finite quiver without loops and oriented cycles and
$\Delta_0 =  \{1,2, \dots,n\}$ be an admissible labeling of its vertices. Assume
$\overr{\Delta}$ is not a  Dynkin
quiver of type $A_n$ with linear ordering. Then we have an isomorphism of triangle functors
$$\tau_D \cong T \circ \RR\SS_1^+ \circ \dots \circ \RR\SS_{n}^+.$$
\end{theorem}

\begin{proof}
By Proposition \ref{P:DervsNonder} and Theorem \ref{T:CoxetervsTau}
we know that $\tau_D \cong \RR\tau \cong T\circ \RR\bigl(\SS_1^+ \circ \dots \circ \SS_n^+\bigr).$
Using the universal property of the right derived functor of a left exact functor, we obtain
a sequence of
natural transformation of triangle functors
$$
\RR\bigl(\SS_1^+ \circ \dots \circ \SS_n^+\bigr) \stackrel{\xi}\lar
\RR \SS_1^+ \circ \RR\bigl(\SS_2^+ \circ \dots \circ \SS_n^+\bigr)
\stackrel{\zeta}\lar \dots \stackrel{\kappa}\lar \RR \SS_1^+ \circ \RR \SS_2^+ \circ \dots \circ \RR \SS_n^+.
$$
Let us show that the first natural transformation $\xi$ is an isomorphism (the proof for the remaining ones is the
same). By  Theorem \ref{T:basicsonrefl}, for
a non-zero  indecomposable representation $X$ we have: if $\SS_*^+(X)$ is non-zero, then
it is indecomposable and  $\RR^1 \SS_*^+(X) = 0$.
Next, by Theorem \ref{T:CoxetervsTau} we have:
$\TT \circ \SS_1^+ \circ \dots \circ \SS_n^+(X) \cong \tau(X)$. Hence, if $X$ is a non-projective
object of $\Rep(\overr\Delta)$
then  $\SS_1^+ \circ \dots \circ \SS_n^+(X) \ne 0$. In particular, for any
injective module $I \in \Rep(\overr\Delta)$ (which is automatically not projective) we get:
$\RR^1 \SS_1^+\bigl(\SS_2^+ \circ \dots \circ \SS_n^+(I)\bigr) = 0$. Hence, the functor $\SS_2^+ \circ \dots \circ \SS_n^+$ maps injective modules
into $\SS_1^+$--acyclic modules. This shows that  the natural transformation $\xi$ is an isomorphism of functors.
\end{proof}

\begin{definition}
Let $C  \in \Mat_{n \times n}(\ZZ)$ be a symmetric matrix such that
$c_{ii} = 2$ for all $1 \le i \le n$ and $c_{ij} < 0$ for all $1 \le i \ne j \le n$.
Consider the $\widetilde\QQ$--algebra $U_q(C)$ generated by the elements
$E_1, \dots, E_n$; $F_1, \dots, F_n$ and $K_1^\pm, \dots, K_n^\pm$ subject to
the following relations:
\begin{itemize}
\item $K_i^\pm K_i^\mp = 1 = K_i^\mp K_i^\pm$, $1 \le i \le n$;
\item $K_i K_j = K_j K_i$, $1 \le i, j \le n$;
\item  $K_i  E_j = v^{-c_{ij}} E_j K_i$ and   $K_i  F_j = v^{c_{ij}} F_j K_i$, $ \,1 \le i, j \le n$;
\item $\bigl[E_i, F_j\bigr] = \delta_{ij} v \frac{\displaystyle K_i - K_i^{-1}}{\displaystyle v - v^{-1}}$,
     $\, 1 \le i, j \le n$;
\item $\sum_{k=0}^{1 -c_{ij}} (-1)^k E_i^{(k)} E_j  E_i^{(1-c_{ij}-k)} = 0$  for
$1 \le i \ne j \le n$;
\item $\sum_{k=0}^{1 - c_{ij}} (-1)^k F_i^{(k)} F_j  F_i^{(1-c_{ij}-k)} = 0$ for $1 \le i \ne j \le n$.
\end{itemize}
\end{definition}

\noindent
The following result is due to Ringel \cite{Ringel}  and Green \cite{Green}.

\begin{theorem}\label{T:RingelGreen}
Let $\overr{\Delta}$ be a quiver without loops and oriented cycles, $|\Delta_0| = n$ and
and $C= C(\Delta) \in \Mat_{n \times n}(\ZZ)$ be the Cartan matrix of $\Rep(\overr{\Delta})$, i.e.
$c_{ij} =
(\bar{S}_i, \bar{S}_j) = \langle \bar{S}_i, \bar{S}_j \rangle + \langle \bar{S}_j, \bar{S}_i \rangle,
\, 1 \le i,j \le n.$
Then there exists an isomorphism of algebras  $U_q(C) \stackrel{\pi_{\overr{\Delta}}}\lar
DC(\overr{\Delta})$ mapping $E_i$ to $[S_i]^+$, $F_i$ to $[S_i]^-$
and $K_i^\pm$ to $K_{\pm\overline{S}_i}$ for $1 \le i \le n$.
\end{theorem}

\noindent
By Theorem \ref{T:Cramer}, the reflection functors $\RR \SS_*^+ $ and $\LL \SS_*^-$ induce
mutually inverse algebra isomorphisms of $DH(\overr{\Delta})$ and $DH(\overl{\Delta})$.
The following result was first  proven  by
Sevenhant and  van den Bergh \cite{SevenhBergh}.

\begin{theorem}\label{C:ReflonCompAlg}
Let $\overr{\Delta}$ be a finite quiver without loops and oriented cycles,
$* \in \Delta_0$ be a sink and $\overl{\Delta}$ be the quiver obtained from $\overr{\Delta}$ by
inverting all the arrows ending at $*$. Then
the derived reflection functor $\RR \SS_*^{+}$  induces an algebra isomorphism of the
 Drinfeld doubles of the composition algebras
$ DC(\overr{\Delta}) \stackrel{\SS_*^{+}}\longrightarrow  DC(\overl{\Delta}),
$
whose inverse is the isomorphism induced by the reflection functor
$\LL \SS_{*}^{-}$. Moreover, they induce a pair  of mutually inverse
 automorphisms  $\SS_*^{\pm}$
of the quantized enveloping algebra
 $U_q(C)$ determined by the quiver $\overr\Delta$:
$$
\xymatrix
{
U_q(C) \ar[rr]^{\SS^{+}_*} \ar[d]_{\pi_{\overr\Delta}} & & U_q(C) \ar[d]^{\pi_{\overl\Delta}} \\
DC(\overr{\Delta}) \ar[rr]^{\SS^{+}_*} & & DC(\overl{\Delta})
}
\qquad
\xymatrix
{
U_q(C) \ar[rr]^{\SS^{-}_*} \ar[d]_{\pi_{\overl\Delta}} & & U_q(C) \ar[d]^{\pi_{\overr\Delta}} \\
DC(\overl{\Delta}) \ar[rr]^{\SS^{-}_*} & & DC(\overr{\Delta})
}
$$
which
are given by the following  formulae:
$$
\begin{array}{|l|cc|l|}
\hline
E_i \xrightarrow{\SS_*^{+}} E_i, \quad F_i  \xrightarrow{\SS_*^{+}}
 F_i &  \mbox{\rm if} & c_{i*} = 0 &
 E_i \xrightarrow{\SS_*^{-}} E_i, \quad F_i  \xrightarrow{\SS_*^{-}} F_i  \\
 \hline
 E_*  \xrightarrow{\SS_*^{+}}  v^{-1} F_* K_*, \quad    F_* \xrightarrow{\SS_*^{+}}  v^{-1}
E_* K_*^{-1} & \mbox{\rm if} & c_{i*} = 2 &
E_*  \xrightarrow{\SS_*^{-}}  v F_* K_*^{-1}, \quad    F_* \xrightarrow{\SS_*^{-}}  v E_* K_*  \\
\hline
 E_i  \xrightarrow{\SS_*^{+}}       \sum\limits_{a+b \,=\, -c_{i*}}
(-1)^a v^{-b} E_*^{(a)}  E_i  E_*^{(b)} & \mbox{\rm if}   & c_{i*} <  0 &
E_i  \xrightarrow{\SS_*^{-}}       \sum\limits_{a+b \,=\, -c_{i*}}
(-1)^a v^{-b} E_i^{(a)}  E_*  E_i^{(b)}
 \\
 \hline
 F_i \xrightarrow{\SS_*^{+}}  \sum\limits_{a+b \, = \, -c_{i*}}
(-1)^a v^{-b} F_*^{(a)}  F_i  F_*^{(b)} &   \mbox{\rm if}   & c_{i*} < 0 &
F_i \xrightarrow{\SS_*^{-}}  \sum\limits_{a+b \, = \, -c_{i*}}
(-1)^a v^{-b} F_i^{(a)}  F_*  F_i^{(b)}
 \\
 \hline
K_i^\pm   \xrightarrow{\SS_*^{+}} K_i^{\pm} & \mbox{\rm if}   & c_{i*} =  0 &
K_i^\pm   \xrightarrow{\SS_*^{-}} K_i^{\pm} \\
\hline
K_i^\pm   \xrightarrow{\SS_*^{+}} K_i^{\mp  c_{i*}} K_{*}^{\pm}  & \mbox{\rm if}   & c_{i*} \ne 0 &
K_i^\pm   \xrightarrow{\SS_*^{-}} K_*^{\mp  c_{i*}} K_{i}^{\pm}.   \\
\hline
\end{array}
$$
As it was explained in \cite[Section 13]{SevenhBergh}, up to a certain twist these automorphisms  coincide with
the symmetries discovered by Lusztig \cite{Lusztig}.
\end{theorem}

\noindent
In the next subsection, the following notion will be important.

\begin{definition}\label{D:CoxAutonComp}
Let $\overr{\Delta}$ be a quiver without loops and oriented cycles and
 $\Delta_0 = $ \linebreak $\{1,2, \dots, n\}$ be an admissible labeling of vertices. Then
$\AA:= \SS_1^+ \circ \dots \circ \SS_n^+:  DC(\overr{\Delta}) \rightarrow DC(\overr{\Delta})$
is the  \emph{Coxeter automorphism} of $DC(\overr{\Delta})$. Using Corollary \ref{C:ReflonCompAlg}, we
also  obtain  the corresponding  automorphism of the algebra $U_q(C)$, given by the commutative diagram
$$
\xymatrix
{
U_q(C) \ar[rrrr]^{\AA^+} \ar[d]_{\pi_{\overr\Delta}} & & & & U_q(C)  \ar[d]^{\pi_{\overr\Delta}} \\
DC(\overr{\Delta}) \ar[r]^{\SS_n^+} &  DC(\overr{\Delta}') \ar[r] & \dots \ar[r] & DC(\overr{\Delta}'') \ar[r]^{\SS_1^+} & DC(\overr{\Delta}).
}
$$
The inverse automorphism $U_q(C) \xrightarrow{\AA^-} U_q(C)$ is defined in a similar way.
\end{definition}

\subsection{Composition algebra of a domestic  weighted projective line}
Let $\XX = \XX(\underline{p})$ be a weighted projective line
of domestic type $\underline{p} = (p_1, \dots, p_n)$. The weight sequence $\underline{p}$ determines the affine Dynkin diagram
$\Delta = \Delta(\underline{p})$ by the following table.

\begin{equation}\label{E:weigtvsDynkintype}
\begin{array}{|c|c|}
\hline
\mathrm{weight \; sequence} \;  \underline{p} & \mathrm{affine \; Dynkin \;  diagram}\;  {\Delta}\\
\hline
(p, q)    & \widehat{A}_{p+q-1} \\
\hline
(2, 2, n): \; (n \ge 2)  & \widehat{D}_{n+2} \\
\hline
(2, 3, 3)                & \widehat{E}_{6}  \\
\hline
(2, 3, 4)              & \widehat{E}_{7} \\
\hline
(2, 3, 5)              & \widehat{E}_{8} \\
\hline
\end{array}
\end{equation}
Note that in the first line of this table $\min(p, q)$ can be equal to one, in that case
the corresponding weight has  to be omitted.

\vspace{1mm}
\noindent
The following theorem  follows from a result of Geigle and Lenzing \cite[Proposition 4.1]{GL}.

\begin{theorem}
Let $\XX = \XX(\underline{p})$ be a weighted projective line of domestic type and
${\Delta}$ be the corresponding affine Dynkin diagram. For a sake of simplicity
 assume that $\overr{\Delta}$ has the ``star-shaped'' orientation in the case
of $\widehat{D}_{n+2}$ $(n \ge 2)$, $\widehat{E}_{6}$, $\widehat{E}_{7}$ and $\widehat{E}_{8}$
and it has $p$ subsequent arrows going clockwise and $q$ subsequent arrows going anti-clockwise
in the case  $\underline{p} = (p, q)$. Then there exists  a derived equivalence
$D^b\bigl(\Rep(\overr{\Delta})\bigr) \stackrel{\GG}\lar D^b\bigl(\Coh(\XX)\bigr)$.
\end{theorem}

\noindent
Applying Cramer's Theorem \ref{T:Cramer}, we immediately obtain the following corollary.

\begin{corollary}\label{C:CramerForAffine}
The  equivalence $D^b\bigl(\Rep(\overr{\Delta})\bigr) \stackrel{\GG}\lar
D^b\bigl(\Coh(\XX)\bigr)$ induces an algebra isomorphism of the reduced
Drinfeld doubles of the Hall algebras $DH(\overr{\Delta}) \stackrel{\GG}\lar DH(\XX)$.
\end{corollary}

\noindent
The next result is a refinement of this statement.

\begin{theorem}\label{T:isomAffineCase}
The  equivalence $D^b\bigl(\Rep(\overr{\Delta})\bigr) \stackrel{\GG}\lar D^b\bigl(\Coh(\XX)\bigr)$ induces an algebra isomorphism of the reduced
Drinfeld doubles of the composition  Hall algebras $DC(\overr{\Delta}) \stackrel{\GG}\lar DU(\XX)$.
Moreover, the following diagram of algebra homomorphisms is commutative:
\begin{equation}\label{E:commCoxTransf}
\begin{array}{c}
\xymatrix{
DC(\overr{\Delta}) \ar[rr]^{\GG} \ar[d]_{\AA_{\overr{\Delta}}} & &  DU(\XX) \ar[d]^{\AA_{\XX}} \\
DC(\overr{\Delta}) \ar[rr]^{\GG}  & &  DU(\XX).
}
\end{array}
\end{equation}
Here $\AA_{\overr{\Delta}}$ is the Coxeter transformation introduced in
Definition \ref{D:CoxAutonComp} and $\AA_{\XX}$ is the automorphism of
$DU(\XX)$ induced by the Auslander-Reiten translation $\kF \mapsto \kF(\vec{\omega})$.
\end{theorem}

\begin{proof}
First recall that the equivalence $\GG$ induces an isomorphism of the K--groups \linebreak
$K_0(\overr{\Delta}) \stackrel{\GG}\lar K_0(\XX)$.
By \cite[Proposition 7.4]{S1},  the composition algebra $U(\XX)$ contains classes of all
indecomposable locally free sheaves on $\XX$.
Next, by  \cite[Theorem 3]{HuberyAffine} it is known that for any indecomposable
preprojective or preinjective  object  $X \in \Ob\bigl(\Rep(\overr{\Delta})\bigr)$,  the
element $[X]$ belongs to the composition
algebra
$C(\overr\Delta)$. Moreover, $\GG(X) \cong \kF[i]$, where $\kF$ is some  indecomposable
vector bundle and $i$ is some integer. Hence, the algebra homomorphism $DH(\overr{\Delta}) \stackrel{\GG}\lar DH(\XX)$ restricts to an injective homomorphism
$DC(\overr{\Delta}) \stackrel{\GG}\lar DU(\XX)$. By Corollary \ref{C:structuteDU(X)}, the reduced
Drinfeld double $DU(\XX)$ is generated by the elements $\bigl[\kO(\vec{x})\bigr]^\pm$ for
$\vec{x} \in \LL(\underline{p})$ and the Cartan part $\widetilde{\QQ}\bigl[K_0(\XX)\bigr]$.
Hence, the map $DC(\overr{\Delta}) \stackrel{\GG}\lar DU(\XX)$ is surjective, hence
an isomorphism.  The commutativity of the diagram (\ref{E:commCoxTransf}) follows from the general fact
that an  equivalence of categories $D^b\bigl(\Rep(\overr{\Delta})\bigr) \stackrel{\GG}\lar
D^b\bigl(\Coh(\XX)\bigr)$ commutes with Serre functors:
$\AA_\XX \circ \GG \cong \GG \circ \AA_{\overr{\Delta}}$, see for example  \cite{ReitenBergh}.
\end{proof}

\begin{remark}
The automorphism $\AA = \AA_\XX$  preserves
the triangular decomposition $DU(\XX)$ $ = \overline{U}(\XX)^+ \otimes_{\widetilde\QQ} \widetilde\QQ\bigl[K_0(\XX)\bigr] \otimes_{\widetilde\QQ} \overline{U}(\XX)^+$. Moreover, it maps the algebra
$DU(\XX)_{\mathsf{tor}}$ to itself. This action is given by the following formulae:
$$
\AA(Z_r^\pm) = Z_r^\pm, \quad \AA(C) = C, \quad \AA\bigl(\bigl[\kS_i^{(j)}\bigr]^\pm\bigr) = \bigl[\kS_i^{(j+1)}\bigr]^\pm
\quad \mbox{\rm and} \quad \AA\bigl(K_i^{(j)}\bigr) = K_i^{(j+1)}
$$
for all $r \in \ZZ_{>0}$ and $(i, j)$ such that $1 \le i \le n, 1 \le j \le p_i$ (as usual, we set
$(i, p_i+1) = (i, 1)$). The explicit action of the Coxeter automorphism $\AA_{\overr\Delta}$
is given by  Theorem \ref{P:aboutDerRefl} and  Theorem \ref{C:ReflonCompAlg}. The commutative diagram
(\ref{E:commCoxTransf}) yields a practical rule to compute  the images
of elements of the algebra $DU(\XX)$ in the algebra $DC(\overr{\Delta})$ under the map $\GG^{-1}$.
\end{remark}

Recall that for any object $\kF$ of the category $\Coh(\XX)$ such that
$\Hom(\kF, \kF) = \kk$ and $\Ext^1(\kF, \kF) = 0$ we have the following equality in
the Hall algebra $H(\XX)$:
$$
\bigl[\kF^{\oplus n}\bigr] = v^{n(n-1)}\frac{[\kF]^n}{[n]!} = v^{n(n-1)} [\kF]^{(n)}.
$$

\begin{definition}
Let $\kP := \bigl\{\alpha \in K_0(\XX) \bigl| \langle \alpha, \alpha\rangle = 1 \;
\mbox{\rm and} \;  \rk(\alpha) > 0 \bigr.\bigr\}$.
\end{definition}

\noindent
It is well-known that an  indecomposable vector bundle $\kF$ on a domestic
projective line $\XX$ is determined by its class in the K--group
$\alpha =  \overline\kF \in \kP \subset K_0(\XX)$.
On the other hand, any real root
$\alpha \in \kP$ corresponds to some indecomposable vector bundle
$\kF = \kF_\alpha$. It is also known that for any pair of indecomposable vector bundles
$\kF_1$ and $\kF_2$ on $\XX$ we have:  $\Ext^1(\kF_1, \kF_2) = 0$ or $\Ext^1(\kF_2, \kF_1) = 0$.
Hence, one can choose on order on  $\kP$ satisfying the following properties:
\begin{itemize}
\item if $\alpha \ne \beta$ and $\langle \alpha, \beta\rangle > 0$ then $\alpha > \beta$,
\item if $\langle \alpha, \beta\rangle < 0$ then $\alpha < \beta$.
\end{itemize}
Let $\kF$ be an arbitrary vector bundle on $\XX$. Then it splits into a direct sum
of indecomposable ones:
$
\kF \cong \kF_{\alpha_1}^{m_1} \oplus \dots \oplus \kF_{\alpha_t}^{m_t}
$
for some uniquely determined $\alpha_1, \dots, \alpha_t \in \kP$ such that
$\alpha_1 > \dots > \alpha_t$ and
$m_1, \dots, m_t \in \ZZ_{>0}$.
Then we have the following equality in the Hall algebra $H(\XX)$:
$
\bigl[\kF\bigr] = v^m \bigl[\kF_{\alpha_1}\bigr]^{(m_1)} \circ \dots \circ \bigl[\kF_{\alpha_t}\bigr]^{(m_t)},
$
where $m = m_1(m_1-1) + \dots + m_t(m_t-1) + \sum_{i < j} m_i m_j \langle \alpha_i, \alpha_j\rangle$.
This implies the following corollary on the structure of the composition algebra $U(\XX)$.

\begin{corollary}
Let $\XX = \XX(\underline{p})$ be a domestic weighted projective line.
Then the  subalgebra $\overline{U}(\XX)_{\mathsf{vec}}$ has a monomial basis
$\bigl[\kF_{\alpha_1}\bigr]^{(m_1)} \circ \dots \circ \bigl[\kF_{\alpha_t}\bigr]^{(m_t)}$
parameterized
by the  sequences $\bigl((\alpha_1, m_1), \dots, (\alpha_t, m_t)\bigr)$, where
$\alpha_1 > \dots > \alpha_t$ are from  $\kP$ and $m_1, \dots, m_t \in \ZZ_{>0}$.
\end{corollary}

\medskip
\noindent
\textbf{Summary}. Let $\XX = \XX(\underline{p})$ be a weighted projective line of domestic type,
$\mathfrak{g}$ the corresponding affine Lie algebra  (see Table (\ref{E:weigtvsDynkintype}))
 and $\kP = \bigl\{\alpha \in K \big|
\langle \alpha, \alpha \rangle = 1, \rk(\alpha) > 0 \bigr\}$ equipped with an order as above.
Then the following results are true.

\medskip
\noindent
1.~There exists an algebra isomorphism $DU(\XX) \stackrel{\FF}\lar
U_q(\mathfrak{g})$ such that $\AA_{\mathfrak{g}}  \circ \FF = \FF \circ \AA$,
where $\AA_{\mathfrak{g}}: U_q(\mathfrak{g}) \rightarrow U_q(\mathfrak{g})$ is the
Coxeter transformation and $\AA = \AA_{\XX}$ is the automorphism of $DU(\XX)$ induced by the
Auslander--Reiten translate.

\medskip
\noindent
2.~We have a triangular decomposition $DU(\XX) = \overline{U}(\XX)^+ \otimes_{\widetilde{\QQ}}
\widetilde{\QQ}[K] \otimes_{\widetilde{\QQ}} \overline{U}(\XX)^-$, which is preserved
under the action of $\AA$, where $\overline{U}(\XX)^\pm \cong \overline{U}(\XX)$ is the non-extended
composition Hall algebra of $\XX$.

\medskip
\noindent
3.~We have an injective algebra homomorphism $DU(\XX)_{\mathsf{tor}} \rightarrow
DU(\XX)$ and an isomorphism of Hopf algebras
$$
DU(\XX)_{\mathsf{tor}} \lar U_q(\widehat{\mathfrak{sl}}_{p_1})
\otimes_\kA \dots \otimes_\kA U_q(\widehat{\mathfrak{sl}}_{p_n}) \otimes_\kA \kH.
$$
In this decomposition, $\kA = \widetilde{\QQ}[C, C^{-1}]$, where $C$ is the ``common'' central element
of all components  $U_q(\widehat{\mathfrak{sl}}_{p_1}), \dots, U_q(\widehat{\mathfrak{sl}}_{p_n})$ and $\kH$
is the Heisenberg algebra. The algebra automorphism $\AA$ acts trivially on $\kH$ and rotates the
Drinfeld--Jimbo generators $\bigl[\kS_i^{(j)} \bigr]$ and $K_i^{(j)}$,  $1 \le j \le p_i$ in each component
 $U_q(\widehat{\mathfrak{sl}}_{p_i})$, $1 \le i \le n$.

 \medskip
\noindent
4.~The multiplication map $\overline{U}(\XX)_{\mathsf{vec}} \otimes_{\widetilde\QQ}
\overline{U}(\XX)_{\mathsf{tor}} \rightarrow \overline{U}(\XX)$ is an isomorphism of vector spaces.
Moreover, $\overline{U}(\XX)_{\mathsf{vec}}$ has a basis parameterized
by the sequences $\bigl((\alpha_1, m_1), \dots, (\alpha_t, m_t)\bigr)$, where
$\alpha_1 > \dots > \alpha_t \in \kP$ and $m_1, \dots, m_t \in \ZZ_{>0}$. This basis
is orthogonal with respect to the Green's form on $U(\XX)$.

\subsection{Composition algebra of a tubular weighted projective line}
The most beautiful  applications of our approach concern the case of a weighted
projective line $\XX$ of a tubular type $(\underline{\lambda}, \underline{p})$.
There are actually only four cases of such curves:
\begin{enumerate}
\item $\underline{\lambda} = (0, \infty, 1, \lambda)$ and
$\underline{p} = (2, 2, 2, 2)$ (type $\widehat{\widehat{D}}_4$).
\item  $\underline{\lambda} = (0, \infty, 1)$ and
$\underline{p} = (3, 3, 3)$ (type $\widehat{\widehat{E}}_6$),
$(2, 4, 4)$ (type $\widehat{\widehat{E}}_7$) or $(2, 3, 6)$ (type $\widehat{\widehat{E}}_8$).
\end{enumerate}
As it was already observed by Geigle and Lenzing in \cite[Example 5.8]{GL}, in the case of the base field $\kk = \mathbb{C}$,
the category $\Coh(\XX)$ is equivalent to the category of equivariant coherent sheaves
on an elliptic curve $\mathbb{E}$ with respect to an appropriate finite group action.
For example, if $\mathbb{E}$ is the  elliptic curve defined by the cubic equation
$zy^2 = x (x- z) (x - \lambda z)$ in $\PP^2$ for some $\lambda \in \kk\setminus \{0, 1\}$
and $\ZZ_2$ is a cyclic group
acting by the rule $(x: y:z) \mapsto (x: -y: z)$ then
$\Coh\bigl(\XX((2,2,2,2), \lambda)\bigr)$ is equivalent to $\Coh^{\ZZ_2}(\mathbb{E})$.

Below we give  a list of  the most important properties of the category of coherent sheaves on a tubular weighted
projective line $\XX$, which are due to Geigle and  Lenzing \cite{GL}, Lenzing and Meltzer
\cite{LenzMeltzer, LenzingMeltzer},
and Meltzer \cite{Meltzer}.

\medskip
\noindent
1. Let $\underline{p} = (p_1, \dots, p_n)$  be the weight type of $\XX$ and $p$ be the least
common multiple of $p_1, \dots, p_n$. Then there exists an isomorphism of functors
$\tau^p \stackrel{\simeq}\lar  \mathbbm{1}$, where $\Coh(\XX) \stackrel{\tau}\lar \Coh(\XX)$ is the
Auslander--Reiten translation in $\Coh(\XX)$. In particular, all Auslander--Reiten
components of $\Coh(\XX)$ are the so-called \emph{tubes}, hence the name ``tubular''.

\medskip
\noindent
2. Let $Z = (\rk,  \deg)\,:\, K_0(\XX) \rightarrow \ZZ^2$ be the standard stability function
on $\Coh(\XX)$. Then
any indecomposable coherent sheaf is automatically semi-stable with respect to $Z$.
For any $\mu \in \QQ \cup \{\infty\}$ let $\SSt^\mu = \SSt^\mu(\XX)$ be the abelian category
of semi-stable coherent sheaves of slope $\mu$ (in these notations,
$\SSt^\infty = \Tor(\XX)$). Then the Auslander-Reiten functor $\tau$
maps $\SSt^\mu$ to itself.

\medskip
\noindent
3. The derived category $D^b\bigl(\Coh(\XX)\bigr)$ has a rich group of exact auto-equivalences.
Let $\kF$ be an arbitrary coherent sheaf such that $\End(\kF) = \kk$ and  $p(\kF)$ be the
smallest positive integer such that $\tau^{p(\kF)}(\kF) \cong \kF$. Then $\kF$ induces  the following
auto-equivalence $T_\kF$ of the derived category $D^b\bigl(\Coh(\XX)\bigr)$: for any object
$\kH$ of $D^b\bigl(\Coh(\XX)\bigr)$ we have
$$
T_\kF(\kH) \cong  \mathsf{cone}\Bigl(\bigoplus\limits_{i \in \ZZ} \bigoplus\limits_{l=1}^{p(\kF)}
\Hom_{D^b(\XX)}\bigl(\tau^l(\kF[-i]), \kH\bigr) \otimes_{\kk} \tau^l(\kF[-i])
\stackrel{\mathsf{ev}}\lar \kH \Bigr).
$$
Moreover,  $T_\kF$ induces an isometry of the K--group $K_0(\XX)$ given by the formula
$$
K_0(\XX) \xrightarrow{T_\kF} K_0(\XX), \quad
a \mapsto a - \sum\limits_{l=1}^{p(\kF)} \bigl\langle \tau^l(\kF), a \bigr\rangle \bigl[\tau^l(\kF)\bigr].
$$
This auto-equivalence $T_\kF$ is called \emph{tubular mutation} or \emph{twist functor}.

\medskip
\noindent
4. Let $R := \bigl\{a \in K_0(\XX) \,| \, \langle a, \, -\, \rangle  = 0  \bigr\}$ be the
left radical of $K_0(\XX)$. Then $R \cong \langle \delta, \omega \rangle \cong \ZZ^2$,
where $\delta$ is the class of a simple torsion sheaf supported at a non-special point of
$\XX(\kk)$ and $\omega = \overline\kO +  \tau(\overline\kO) + \dots + \tau^{p(\kO)-1}(\overline\kO)$.
Moreover, the canonical group homomorphism $\Aut\bigl(D^b\bigl(\Coh(\XX)\bigr)\bigr)
\stackrel{\pi}\lar   \Iso\bigl(K_0(\XX)\bigr)$ restricts to a surjective group
homomorphism $\Aut\bigl(D^b\bigl(\Coh(\XX)\bigr)\bigr) \rightarrow \SL(2, \ZZ) \subseteq \Aut(R)$.
Here,  $\Iso\bigl(K_0(\XX)\bigr)$ is the group of isometries of $K_0(\XX)$, i.e.~the group
of linear automorphisms of $K_0(\XX)$ respecting the Euler form $\langle \,-\,,\,-\,\rangle$.

\medskip
\noindent
5. Let $\kS$ be the simple torsion sheaf supported at a non-special
point of $\XX(\kk)$. Then the tubular mutations $\FF: = T_\kO$ and $\GG:= T_\kS$ satisfy the braid group
relation: $\FF \GG \FF \cong \GG \FF \GG$, see
also \cite{SeidelThomas} for a discussion of twist functors and the induced braid group actions
in more general situations.
If the field $\kk$ is algebraically
 closed then there exists an exact sequence
 $$
 1 \lar \Pic^0 \rtimes  \Aut(\XX) \lar \Aut\bigl(D^b\bigl(\Coh(\XX)\bigr)\bigr)
 \lar B_3 \lar 1
 $$
 where $B_3$ is the braid group on three strands and  $\Pic^0(\XX)$ is the
 group of line bundles on $\XX$ of degree zero.  Moreover, $\Pic^0 \rtimes  \Aut(\XX)$
 consists of all elements of $\Aut\bigl(D^b\bigl(\Coh(\XX)\bigr)\bigr)$ having finite order.

\medskip
\noindent
6. For any slope $\mu \in \QQ$,  the category of semi-stable sheaves
$\SSt^\mu(\XX)$ is equivalent to the category of torsion sheaves $\Tor(\XX) = \SSt^\infty(\XX)$. Any
element $\FF \in \Aut\bigl(D^b(\Coh(\XX)\bigr)$ establishes  an equivalence
between $\SSt^\mu$  and $\SSt^\nu[i]$ for some $\nu \in \QQ \cup \{\infty\}$
and $i \in \ZZ$.

\begin{lemma}
Let $\XX$ be a weighted projective line of tubular type. Then
there is a canonical isomorphism of vector spaces over $\widetilde\QQ$:
\begin{equation}\label{E:HNNdecompH(X)}
\stackrel{\rightarrow}\otimes_{\mu \in \QQ \cup \{\infty\}}
\overline{H}(\SSt^\mu) \xrightarrow{\mathsf{mult}} \overline{H}(\XX),
\end{equation}
where $\overline{H}(\SSt^\mu)$ is the non-extended Hall algebra of the abelian category $\SSt^\mu$
of semi-stable sheaves of slope $\mu$ and $\stackrel{\rightarrow}\otimes$ is the
restricted directed tensor product. This means that  a simple tensor in the left-hand side
of (\ref{E:HNNdecompH(X)}) has all but finitely many entries equal to  $\mathbbm{1}$.
\end{lemma}

\begin{proof}
Since any indecomposable sheaf on $\XX$ is automatically semi-stable, for any
$\kF \in \Ob\bigl(\Coh(\XX)\bigr)$ there exist a unique decomposition
$\kF \cong \kF_1 \oplus \dots \oplus \kF_t$,  where all sheaves
$\kF_1, \dots, \kF_t$ are semi-stable and $\mu(\kF_1) < \dots < \mu(\kF_t)$. Then we
have: $\Ext^1(\kF_i, \kF_l) = 0 = \Hom(\kF_l, \kF_i)$ for any $1 \le i < l\le t$.
Hence, there is the
following equality in the Hall algebra $H(\XX)$:
$
\bigl[\kF\bigr] = v^m \bigl[\kF_1\bigr] \circ \dots \circ \bigl[\kF_t\bigr],
$
where $m = \sum_{i<l} \bigl\langle \overline{\kF_i}, \overline{\kF_l} \bigr\rangle$.
It shows that the map $\mathsf{mult}$ is surjective. It is also not difficult to
see that $\mathsf{mult}$ is injective, hence isomorphism.
\end{proof}

\begin{theorem}\label{T:actionAutDU(tubular)}
Let $\XX$ be a tubular weighted projective line.
Then the  group of exact auto-equivalences $\Aut\bigl(D^b(\Coh(\XX)\bigr)$ acts on the reduced
Drinfeld double $DH(\XX)$ by algebra automorphisms. For any
$\FF \in \Aut\bigl(D^b(\Coh(\XX)\bigr)$ its  action on $DH(\XX)$ is uniquely determined by the following rules: for $\alpha \in K_0(\XX)$ we have
$\FF(K_\alpha) = K_{\FF(\alpha)}$ and
for any semi-stable sheaf  $\kF$ such that
$\FF(\kF) \cong \widehat{\kF}[n]$ with $\widehat{\kF} \in \Coh(\XX)$ and
$n \in \ZZ$ we have:
$$
\FF\bigl([\kF]^\pm\bigr) = v^{-n \bigl\langle \bar{\kF}, \bar{\kF}\bigr\rangle}
\bigl[\widehat{\kF}\bigr]^{\pm \varepsilon(n)}
\bigl(K_{\widehat{\kF}}^{\pm \varepsilon(n)}\bigr)^n,
$$
where $\varepsilon(n) = (-1)^n$.
Moreover, any element  $\FF \in \Aut\bigl(D^b\bigl(\Coh(\XX)\bigr)\bigr)$ maps the reduced
Drinfeld double $DU(\XX)$ to itself. In other words, the group
of exact auto-equivalences of the derived category $D^b\bigl(\Coh(\XX)\bigr)$
acts on the reduced Drinfeld double of the composition algebra $U(\XX)$ by algebra
automorphisms.
\end{theorem}

\begin{proof} The fact that $\FF$ is an algebra automorphism of $DH(\XX)$ can be proven along
the same lines as \cite[Theorem 3.8]{HallEll}. It is also a special  case of a general Theorem
\ref{T:Cramer} proven by Cramer \cite{Cramer}.

Let $\kF$ be an exceptional  coherent sheaf on $\XX$, i.e.~$\End(\kF) = \kk$ and $\Ext^1(\kF, \kF) = 0$.
By \cite[Proposition 8.7]{S1} it is known that the element $[\kF]$ belongs to the composition algebra $U(\XX)$.
By Corollary \ref{C:structuteDU(X)}, the reduced Drinfeld double $DU(\XX)$ is generated
by the Cartan part $\widetilde\QQ[K]$ and classes of certain exceptional objects (e.g.~by the line bundles).
But this implies that the algebra $DU(\XX)$ is invariant under the action of
$\Aut\bigl(D^b\bigl(\Coh(\XX)\bigr)\bigr)$.
\end{proof}

\begin{remark}
In the  work of Lin and Peng \cite{LinPeng}, the Hall--Lie algebra $\mathfrak{g}(A)$ of  the \emph{root
category} $D^b(A-\mod)/[2]$ of a tubular algebra $A$ was studied. They have shown that
$\mathfrak{g}(A)$ is isomorphic to the toroidal Lie algebra of the corresponding Dynkin type.
\end{remark}

\begin{lemma}
For any $\alpha \in K_0(\XX)$,  let $\mathbbm{1}^{\mathrm{ss}}_{\alpha} := \sum\limits_{[\kF] \in \cI: \,
\kF \in \SSt_\alpha}
[\kF] \in \overline{H}(\XX)[\alpha],
$
where $\SSt_\alpha$  is the category of semi-stable objects of  $\Coh(\XX)$ of class
$\alpha$ with respect to the standard stability function $Z = (\rk, \deg)$. Then $\mathbbm{1}^{\mathrm{ss}}_{\alpha}$ belongs to the composition algebra $U(\XX)$ (note
that the category $\SSt_\alpha$ can be empty).
\end{lemma}

\begin{proof}
By Proposition \ref{P:onU(X)tor} for any $\alpha \in K_0(\XX)$ such that
$\rk(\alpha) = 0$,  the element $\mathbbm{1}_{\alpha} = \mathbbm{1}^{\mathrm{ss}}_{\alpha}$ belongs
to the algebra $U(\XX)_{\mathsf{tor}}$. Moreover, there exists an auto-equivalence
$\FF \in \Aut\bigl(D^b\bigl(\Coh(\XX)\bigr)\bigr)$ such that
$\rk\bigl(\FF(\alpha)\bigr) = 0$. Up to a shift
$\FF$ maps the category $\SSt^\mu$ to $\Tor(\XX)$, where
$\mu$ is the slope corresponding to the class $\alpha$.  Let $\gamma:= \FF(\alpha)$, then
by Theorem \ref{T:actionAutDU(tubular)} the element
$\mathbbm{1}^{\mathrm{ss}}_{\alpha} = \FF^{-1}\bigl(\mathbbm{1}^{\mathrm{ss}}_{\gamma})$
belongs to  $DU(\XX) = \overline{U}(\XX)^+ \otimes_{\widetilde\QQ}
\widetilde\QQ[K] \otimes_{\widetilde\QQ} \overline{U}(\XX)^-.$
Since $\mathbbm{1}^{\mathrm{ss}}_{\alpha}$ lies  in  $\overline{H}(\XX)^+$, it is an element
of  $\overline{U}(\XX)^+$, too.
\end{proof}

\begin{theorem}\label{T:PBWintoroidal}
Let $\XX$ be a weighted projective line of tubular type $\underline{p} = (p_1, \dots, p_n)$.
For any slope $\mu \in \QQ \cup \{\infty\}$,  let
$U^\mu = U(\XX)^\mu$ be the subalgebra of the composition algebra
$\overline{U}(\XX)$ generated by the set
$\bigl\{\mathbbm{1}^{\mathrm{ss}}_{\alpha}\bigr\}_{\alpha \in K_0(\XX): \mu(\alpha) = \mu}$. Then we have:
\begin{enumerate}
\item $U^\mu$ is isomorphic to the algebra $\overline{U}(\XX)_{\mathsf{tor}} \cong
\mathcal{Z} \otimes_{\widetilde\QQ} U_q^+(\widehat{\mathfrak{sl}}_{p_1}) \otimes_{\widetilde\QQ}
\dots \otimes_{\widetilde\QQ} U_q^+(\widehat{\mathfrak{sl}}_{p_n})$.
\item The canonical map $\stackrel{\rightarrow}\otimes_{\mu \in \QQ \cup \{\infty\}}
U^\mu \xrightarrow{\mathsf{mult}} \overline{U}(\XX)$ is an isomorphism of vector spaces.
\end{enumerate}
\end{theorem}

\begin{proof}
We know that there exists an equivalence of categories $\SSt^\mu(\XX)
\stackrel{\FF}\lar
\Tor(\XX)$. Hence, $\FF$ induces an isomorphism  of non-extended
Hall algebras $\overline{H}\bigl(\SSt^\mu(\XX)\bigr) \stackrel{\FF}\lar \overline{H}\bigl(\Tor(\XX)\bigr)$. Moreover, $\FF$ induces a bijection between the sets
$\bigl\{\mathbbm{1}_\gamma | \rk(\gamma) = 0 \bigr\}$ and $\bigl\{\mathbbm{1}^{\mathrm{ss}}_\alpha | \mu(\alpha) = \mu\bigr\}$. Hence, the algebras $U^\mu$ and $U^\infty = \overline{U}(\XX)_{\mathsf{tor}}$
are isomorphic. By Proposition \ref{P:structU(X)} we know that $\overline{U}(\XX)_{\mathsf{tor}} \cong \mathcal{Z} \otimes_{\widetilde\QQ} U_q^+(\widehat{\mathfrak{sl}}_{p_1}) \otimes_{\widetilde\QQ}
\dots \otimes_{\widetilde\QQ} U_q^+(\widehat{\mathfrak{sl}}_{p_n})$. This implies the first part of the claim.

Since the map $\stackrel{\rightarrow}\otimes_{\mu \in \QQ \cup \{\infty\}}
U^\mu \xrightarrow{\mathsf{mult}} \overline{U}(\XX)$ is a restriction of the isomorphism (\ref{E:HNNdecompH(X)}), it is at least injective. Hence, we only have to prove the
surjectivity of $\mathsf{mult}$. For any slope $\nu$ denote
$H^\nu = \overline{H}(\SSt^\nu)$. It is sufficient to show  that for any $\mu \in \QQ \cup \{\infty\}$
the algebra $\overline{U}(\XX)$ is contained in the image of the map
$
\bigl(\stackrel{\rightarrow}\otimes_{\nu < \mu }
H^\nu\bigr) \otimes U^\mu \otimes  \bigl(\stackrel{\rightarrow}\otimes_{\nu > \mu }
H^\nu\bigr)
\xrightarrow{\mathsf{mult}} \overline{H}(\XX).
$
By Proposition \ref{P:decompU(X)},  it is true for $\mu = \infty$.  Let $\FF$ be an auto-equivalence
of $D^b\bigl(\Coh(\XX)\bigr)$ mapping the given slope $\mu \in \QQ$ to $\infty$. If
  $\kappa$ is the image
of the slope $\infty$ under $\FF$ then  we have a commutative diagram
$$
\xymatrix
{
\FF\bigl(\overline{U}(\XX)^+\bigr) \ar@{^{(}->}[rr] \ar@{^{(}->}[d] & & \overline{U}(\XX)^+ \otimes_{\widetilde\QQ} \widetilde\QQ[K]
\otimes_{\widetilde\QQ} \overline{U}(\XX)^- \ar@{_{(}->}[d] \\
\bigl(\stackrel{\rightarrow}\otimes_{\nu > -\kappa }
H^{\nu, +}\bigr) \otimes_{\widetilde\QQ} \widetilde\QQ[K]  \otimes_{\widetilde\QQ} \bigl(\stackrel{\rightarrow}\otimes_{\nu \le \kappa }
H^{\nu, -}\bigr)
\ar@{^{(}->}[rr]  & & \overline{H}(\XX)^+ \otimes_{\widetilde\QQ} \widetilde\QQ[K]
\otimes_{\widetilde\QQ} \overline{H}(\XX)^-.
}
$$
Let
$a \in \overline{U}(\XX) = \overline{U}(\XX)^+$ be an arbitrary element. Since its image
$b := \FF(a)$ belongs both to $DU(\XX) = \overline{U}(\XX)^+ \otimes_{\widetilde\QQ} \widetilde\QQ[K]
\otimes_{\widetilde\QQ} \overline{U}(\XX)^-$ and to $\bigl(\stackrel{\rightarrow}\otimes_{\nu > -\kappa }
H^{\nu, +}\bigr) \otimes_{\widetilde\QQ} \widetilde\QQ[K]  \otimes_{\widetilde\QQ} \bigl(\stackrel{\rightarrow}\otimes_{\nu \le \kappa }
H^{\nu, -}\bigr)$, it can be written as a sum of monomials  of the form
$c_{1} \dots c_{t} d e_{1} \dots e_{r}$, where
$c_{i} \in H^{\mu_i, +}$, $d \in \widetilde\QQ[K]$ and $d_{l} \in H^{\nu_l, -}$ are such that
$ - \kappa < \mu_1 < \dots < \mu_t$, $\nu_1 < \dots < \nu_r \le \kappa$ and
$c_t \in U^{\infty, +}$ (it is also possible that $c_t = 1$). Hence,
$a = \FF^{-1}(b) \in \bigl(\stackrel{\rightarrow}\otimes_{\nu < \mu }
H^\nu\bigr) \otimes U^\mu \otimes  \bigl(\stackrel{\rightarrow}\otimes_{\nu > \mu }
H^\nu\bigr)$.
\end{proof}

\begin{corollary}
Starting with  some basis of   $\overline{U}(\XX)_{\mathsf{tor}}^{\mathsf{exc}} \cong U_q^+(\widehat{\mathfrak{sl}}_{p_1}) \otimes_{\widetilde\QQ}
\dots \otimes_{\widetilde\QQ} U_q^+(\widehat{\mathfrak{sl}}_{p_n})$, orthonormal with respect to
the Green's form, Theorem \ref{T:PBWintoroidal} gives a construction
of a PBW-type basis of the algebra $U(\XX)$, which is \emph{orthonormal}   with respect to
the Green's form. In a similar way, we get a PBW-type basis of the reduced Drinfeld double
$DU(\XX)$.
\end{corollary}

\begin{remark}
Let $\XX$ be a weighted projective line of tubular type $\widehat{\widehat{{\Delta}}}$,
where $\Delta \in \bigl\{D_4, E_6, E_7, E_8\bigr\}$
and $\mathfrak{g}$ be the simple Lie algebra of the Dynkin type $\Delta$. In the  work
of the second-named author \cite{S1} it was shown that the composition algebra
$U(\XX)$ can be identified with a certain (quite non-standard) Borel subalgebra
$U_q(\mathfrak{b})$ of
the quantized double loop algebra $U_q\bigl(\mathfrak{L} \mathfrak{g}_{\widehat\Delta}\bigr)$, where
$\mathfrak{g}_{\widehat\Delta}$ is the affine Lie algebra of type $\widehat\Delta$ (see
the survey article \cite{Hernandez} for a definition and applications
of quantized double loop algebras).
In \cite{S2} it was shown how a
Lusztig-type approach \cite{Lusztig} leads to  a construction of  a \emph{canonical basis} of $U_q(\mathfrak{b})$.
Therefore,  it is natural to conjecture that the reduced Drinfeld double $DU(\XX)$ is isomorphic to
the whole
quantized double loop algebra  $U_q\bigl(\mathfrak{L} \mathfrak{g}_{\widehat\Delta}\bigr)$.
\end{remark}

\begin{remark}
The general results on the structure of the reduced Drinfeld double of a weighted projective
line $\XX$, listed at the end of  Section \ref{S:CompAlgII},  lead to very interesting consequences for the
theory  of the quantized toroidal (or double loop) enveloping algebras of types $\widehat{\widehat{D}}_4$,
$\widehat{\widehat{E}}_6$, $\widehat{\widehat{E}}_7$ and $\widehat{\widehat{E}}_8$.
For example, let $\underline{p} = (2, 3, 6)$ and $\XX = \XX(\underline{p})$ be the corresponding
tubular weighted
projective line of type  $\widehat{\widehat{E}}_8$. Then we know that
\begin{enumerate}
\item
The algebra $DU(\XX)$ contains a subalgebra isomorphic to
$$\mathcal{H} \otimes_{\kA} U_q\bigl(\widehat{\mathfrak{sl}}_{2}\bigr) \otimes_{\kA}
U_q\bigl(\widehat{\mathfrak{sl}}_{3}\bigr) \otimes_{\kA} U_q\bigl(\widehat{\mathfrak{sl}}_{6}\bigr),$$
where $\mathcal{H}$ is the Heisenberg algebra and $\kA$ is the ring of Laurent polynomials
in the ``common'' central element $C$.
\item The algebra $DU(\XX)$ contains a subalgebra isomorphic to
$DU(\PP^1) \cong U_q\bigl(\mathfrak{L} \mathfrak{sl}_{2}\bigr)$.
\item More generally, let $\underline{q} = (2, 3, 5)$ and $\YY = \YY(\underline{q})$
be the corresponding weighted projective line of domestic type. Then there
exists  an embedding $DU(\YY) \rightarrow DU(\XX)$. This should correspond to some rather  non-trivial
embedding of the quantized enveloping algebra $U_q\bigl(\widehat{E}_8\bigr)$ into  the
algebra $U_q\bigl(\widehat{\widehat{E}}_8\bigr)$.
\end{enumerate}
\end{remark}

\section{Appendix: Stability Conditions and Hall Algebras}\label{S:StabilCond}

\noindent
In this section we collect some definitions and basic results on  stability conditions on abelian categories
and their use in the theory of Hall algebras.

\begin{definition}\label{D:stability} Let $\cA$ be an abelian category and $K_0(\cA)$ be its Grothendieck group.
\begin{enumerate}
\item A stability condition on $\cA$ is
an additive function $Z = (Z_1, Z_2) \,:\, K_0(\cA) \rightarrow  \RR^2$ such that
 for  any non-zero object $X \in \Ob(\cA)$ we have:
 \begin{enumerate}
 \item $Z_1(\overline{X}) \ge  0$,
\item if $Z_1(\overline{X}) = 0$ then $Z_2(\overline{X}) >0$.
\end{enumerate}
\item For a non-zero object $X \in \Ob(\cA)$ the fraction
$$
\mu(X) := \frac{Z_2(\overline{X})}{Z_1(\overline{X})} \in \RR \cup \{+\infty\}
$$
is called the slope of $X$.
\item $X \in \Ob(\cA)$ is called $Z$--stable (respectively  $Z$--semi--stable)
if for any proper subobject $Y$ of $X$ we have: $\mu(Y) < \mu(X)$ (respectively $\mu(Y) \le \mu(X)$).
\end{enumerate}
\end{definition}

\begin{example}\label{E:curves}
Let $\mathbb{X}$ be a smooth projective curve over the field $\kk$ and $\cA = \Coh(\XX)$
be the category of coherent sheaves on $\XX$. Then $\cA$ admits the following canonical stability
function $Z = (\rk, \deg) \, : \,  K_0(\cA)  \rightarrow \ZZ^2$, where
$\rk(\kF)$ is the rank of a coherent sheaf $\kF$  and $\deg(\kF)$ is its degree.
\qed
\end{example}

\begin{example}
Let $\overr{Q}$ be a finite quiver without loops and oriented cycles with
the set of vertices $\{1, \dots, n\}$  and
$\cA = \Rep(\overr{Q})$ be its category of finite dimensional representations over the field $\kk$.
Note that  $K_0(\cA) = \langle \overline{S}_1, \dots, \overline{S}_n\rangle \cong \ZZ^n$, where
$S_i$ is the simple representation corresponding to the vertex $i$.

For any $X \in \Ob(\cA)$ we denote by $\underline\dim(X) = (d_1, \dots, d_n) \in
\ZZ^n$ the dimension vector of $X$ and by $\dim(X) = d_1 + \dots + d_n$ its dimension
as a vector space. For any $\underline{\theta} = (\theta_1, \dots, \theta_ n) \in \RR^n$
we get  the following  stability function
$$
Z: K_0(\cA)  \rightarrow \RR^2, \quad
\overline{X} \mapsto \bigl(\dim(X), \underline{\theta} \cdot \underline{\dim}(X)\bigr),
$$
where $\underline{\theta} \cdot \underline{\dim}(X) = \sum_{i=1}^n \theta_i d_i$.
\qed
\end{example}

\noindent
The proof of the following proposition is an easy exercise.

\begin{proposition} In the notations of Definition \ref{D:stability} we have the following.
\begin{enumerate}
\item Let $X \in \Ob(\cA)$ be $Z$--stable. Then $\End_{\cA}(X)$ is a skew field.
\item For any  $\mu  \in \RR \cup \{+\infty\}$, the category $\mathsf{SS}^\mu(\cA)$
of $Z$--semi--stable objects
of slope $\mu$ is abelian. In other words, $\mathsf{SS}^\mu(\cA)$  is closed under taking extensions,
kernels and cokernels
inside of the category  $\cA$.
\item For any pair of $Z$--semi--stable objects $X$ and $Y$ such that $\mu(X) > \mu(Y)$
we have: $\Hom_\cA(X, Y) = 0$.
\end{enumerate}
\end{proposition}

\noindent
The following fundamental result is due to Harder and Narasimhan, who have shown
it for the stability condition from Example \ref{E:curves}. In the stated generality,
it can be found in the paper of Rudakov \cite{Rudakov}.

\begin{theorem}\label{T:HarderNarasimhan}
Assume our category $\cA$ to be Noetherian. Then for any $0 \not\cong X \in \Ob(\cA)$
there exists a sequence of monomorphisms
\begin{equation}\label{E:HarderNarasimhanFilt}
0 = X_{m+1} \stackrel{f_{m+1}}\lar X_m \stackrel{f_{m}}\lar \dots
\stackrel{f_2}\lar X_1 \stackrel{f_1}\lar X_0 = X
\end{equation}
satisfying the following two properties.
\begin{enumerate}
\item For any $0 \le l \le m$ the object $A_l := \coker(f_{l+1})$ is
$Z$--semi--stable.
\item Moreover, we have: $\mu(A_{m}) > \dots > \mu(A_0)$.
\end{enumerate}
A  sequence  (\ref{E:HarderNarasimhanFilt}) is unique in the following sense: if
$$
0 = Y_{n+1} \stackrel{g_{n+1}}\lar Y_n \stackrel{g_{n}}\lar \dots
\stackrel{g_2}\lar Y_1 \stackrel{g_1}\lar Y_0 = X
$$
is another sequence of monomorphisms satisfying the above properties then $n = m$ and
for any $1 \le l \le m$ there exists a unique isomorphism $h_l: X_l \rightarrow Y_l$ such that
the following diagram is commutative:
$$
\xymatrix{
X_m \ar[r]^{f_{m}} \ar[d]_{h_m} & X_{m-1} \ar[r]^{f_{m-1}}
\ar[d]^{h_{m-1}} & \dots \ar[r]^{f_{2}} & X_{1} \ar[r]^{f_{1}}  \ar[d]^{h_1} & X \ar[d]^{\mathbbm{1}_X} \\
Y_m \ar[r]^{g_m}  & Y_{m-1} \ar[r]^{g_{m-1}}  & \dots \ar[r]^{g_{2}} & Y_{1} \ar[r]^{g_{1}}  & X.
}
$$
A sequence (\ref{E:HarderNarasimhanFilt}) is called \emph{Harder--Narasimhan filtration} of
$X$ and the semi--stable objects $A_0, \dots, A_m$ are called \emph{Harder--Narasimhan quotients} of
$X$.
\end{theorem}

\noindent
Let
 $\bigl(\Gamma, \{\,-\,,\,-\,\}\bigr)$ be an abelian group
equipped with a bilinear form  and $$
\bigl(K_0(\cA), \langle \,-\,,\,-\,\rangle\bigr) \stackrel{\mathsf{ch}}\lar
\bigl(\Gamma, \{\,-\,,\,-\,\}\bigr)$$ be any isometry.
Assume $\Gamma \stackrel{T}\lar  \mathbb{R}^2$ is a group homomorphism such that
$Z  = T \circ \mathsf{ch}$ is a stability function for $\cA$.  Next, assume that for  any
$\gamma  \in \Gamma$ there are  only finitely many $[X] \in \cI$ such that $\mathsf{ch}(\overline{X}) = \gamma$.
Consider the following element of Hall algebra $H(\cA)$:
$$
\mathbbm{1}_\gamma  := \sum\limits_{[X] \in \cI: \, \mathsf{ch}(\bar{X}) = \gamma}
[X]
\quad
\mbox{\rm and}
\quad
\mathbbm{1}^{\mathrm{ss}}_\gamma  = \mathbbm{1}^{\mathrm{ss}}_{\gamma, Z} := \sum\limits_{[X] \in \cI: \, X \in \mathsf{SS}(\cA)[\gamma]}
[X],
$$
where $\mathsf{SS}(\cA)[\gamma]$ is the category of $Z$--semi-stable objects of  $\cA$ having  class
$\gamma$ in $\Gamma$.  The following result is a direct consequence of Theorem \ref{T:HarderNarasimhan}.

\begin{proposition}\label{P:knownequal}
For any element $\gamma \in \Gamma$ we have the equality
\begin{equation}\label{E:HNFiltandHallAlg}
\mathbbm{1}_\gamma = \mathbbm{1}^{\mathrm{ss}}_\gamma +
\sum\limits_{t\ge  2}  \sum\limits_{\substack{\gamma_1 + \dots + \gamma_t = \gamma \\
\mu(\gamma_1) < \dots < \mu(\gamma_t) }}
v^{\sum_{i < j} \langle \gamma_i \gamma_j \rangle}
\mathbbm{1}^{\mathrm{ss}}_{\gamma_1} \circ \dots \circ \mathbbm{1}^{\mathrm{ss}}_{\gamma_t}.
\end{equation}
\end{proposition}
\begin{remark}
Note that the sum in the right-hand side of  (\ref{E:HNFiltandHallAlg}) is finite.
\end{remark}

\begin{corollary}\label{C:stabcondHallAlg} In the above notations,
the elements $\bigl\{\mathbbm{1}_\alpha\bigr\}_{\alpha \in \Gamma}$ and
$\bigl\{\mathbbm{1}^{\mathrm{ss}}_{\beta, Z}\bigr\}_{\beta \in \Gamma}$ generate the same
subalgebra of the Hall algebra $H(\cA)$.
\end{corollary}

\begin{proof}
For any $\alpha \in \Gamma$ the  element $\mathbbm{1}_{\alpha}$ belongs to the
 algebra $\bigl\{\mathbbm{1}^{\mathrm{ss}}_{\beta, Z}\bigr\}_{\beta \in \Gamma}$.
  From a ``triangular'' form of the equality (\ref{E:HNFiltandHallAlg})
 it  follows  that also other way around, for any $\alpha \in K_0(\cA)$ the  element  $\mathbbm{1}^{\mathrm{ss}}_{\alpha}$
 belongs to the
 subalgebra generated by $\bigl\{\mathbbm{1}_\beta\bigr\}_{\beta \in \Gamma}$.
\end{proof}

\begin{remark}\label{R:ReinekeInv}
A result  of Reineke \cite[Theorem 5.1]{Reineke} provides an explicit formula expressing
the elements $\mathbbm{1}^{\mathrm{ss}}_\alpha$ via $\mathbbm{1}_\beta$ for an arbitrary
stability function $Z$:
$$
\mathbbm{1}^{\mathrm{ss}}_{\alpha} = \mathbbm{1}_\alpha +
\sum\limits_{t\ge  2} (-1)^{t-1}  \sum\limits_{\substack{\alpha_1 + \dots + \alpha_t = \alpha: \,
\forall \,  1 \le s \le t-1 \\
\mu(\alpha_1 + \dots + \alpha_s) > \mu(\alpha)}}
v^{\sum_{i < j} \langle \alpha_i \alpha_j \rangle}
\mathbbm{1}_{\alpha_1} \circ \dots \circ \mathbbm{1}_{\alpha_t}.
$$
\end{remark}

\end{document}